\documentclass[12pt,a4paper]{amsbook}

\usepackage{amssymb,latexsym,graphicx,amsxtra,amsmath,amsthm}
\usepackage{amsmath,amsthm,amssymb}

\usepackage{amscd}

\newcommand{\ori}{orientation}
\newcommand{\tri}{\bigtriangleup^}
\newcommand{\sg}{\sigma}
\newcommand{\h}{\mathbb{H}}

\newcommand{\C}{\mathbb{C}}
\newcommand{\R}{\mathbb{R}}
\newcommand{\p}{\mathbb{P}^}

\newcommand{\D}{\mathbb{D}}

\newcommand{\la}{\lambda}

\newcommand{\N}{\mathbb{N}}
\newcommand{\Z}{\mathbb{Z}}
\newcommand{\I}{\mathbb{I}}
\newcommand{\s}{\mathbb{S}}
\newcommand{\al}{\alpha}

\numberwithin{equation}{section}
\numberwithin{section}{chapter}
\newtheorem{theorem}{Theorem}[section]
\newtheorem{lemma}[theorem]{Lemma}
\newtheorem{corollary}[theorem]{Corollary}
\newtheorem{proposition}[theorem]{Proposition}

\theoremstyle{remark}

\begin{document}
\frontmatter
\thispagestyle{empty}
\begin{center} \large
\textbf{TRIANGULATION AND CLASSIFICATION OF 2-MANIFOLDS}
\end{center}

\vfill

\begin{center}
Dissertation submitted to the University of Delhi\\
 in partial fulfillment of the requirements for the degree of
\end{center}

\vfill

\begin{center}\large
\textbf{MASTER OF PHILOSOPHY\\
 IN  \\
 MATHEMATICS }
\end{center}

\vfill

\vfill

\begin{center}
By\\ \large \textbf{KUSHAL LALWANI}
\end{center}
\vfill


\vfill

\begin{center}\bf
DEPARTMENT OF MATHEMATICS \\ UNIVERSITY OF DELHI\\ DELHI-110007\\
APRIL, 2014
\end{center}


\newpage

\bigskip

\bigskip
\centerline{\textsc{Declaration}}
\bigskip

\bigskip
This dissertation entitled ``\textbf{Triangulation and  Classification of  2-Manifolds}'' contains a comprehensive study of the topology of 2-manifolds and a complementary analysis of the work done by Edwin E. Moise ~\cite{moise}, L. V. Ahlfors ~\cite{ahlfors} and Ian Richards ~\cite{rich}. This study has been carried out by me under the supervision of \textbf{Dr.~Sanjay~Kumar}, Associate Professor, Department of Mathematics, Deen Dayal Upadhyaya College, University of Delhi, Delhi, for the award of the degree of Master of Philosophy in Mathematics.

I hereby also declare that, to the best of my knowledge, the work included in this dissertation has not been submitted earlier, either in part or in full, to this or any other University/institution for the award of any degree or diploma.
 \bigskip
  \bigskip

\bigskip

\bigskip
\begin{flushright}
\textbf{(Kushal Lalwani)}
\end{flushright}

\bigskip
 \bigskip

\bigskip \bigskip

\bigskip \bigskip

\bigskip
\bigskip
\bigskip
\bigskip

\begin{tabular}{l}
  (Supervisor)\\
 \textbf{Dr. Sanjay Kumar}\\
  Department of Mathematics\\
  Deen Dayal Upadhyaya College\\
  (University of Delhi)\\
  Karam Pura, New Delhi-110015\\
\end{tabular}
\hfill
\begin{tabular}{l}
(Head of the Department)\\
\textbf{Prof. Ajay Kumar}\\
Department of Mathematics\\
University of Delhi\\
Delhi-110007\\
\end{tabular}


\newpage

\centerline{\textbf{ Acknowledgment }}
\bigskip
\bigskip

I take this opportunity to express my deep sense of gratitude to my supervisor, \emph{Dr. Sanjay Kumar}, for his constant encouragement, cooperation and invaluable guidance in the successful accomplishment of this dissertation. I also express my gratitude to \emph{Prof. Ajay Kumar}, Head,  Department of Mathematics, University of Delhi for providing necessary facilities and constant encouragement during the course of this study.

I also wish to extend my thanks to all the faculty members of the Department of Mathematics, University of Delhi for their help, guidance and motivation for the work presented in the dissertation. They have always been there for me whenever I needed support from them, providing me critical research insights and answering my questions with their valuable time. Their academic excellence has also been a great value to my dissertation. I am also thankful to the organizers of ATM schools of geometry and topology, which I attended in CEMS Almora, NEHU Shillong and HRI Allahabad, which helped me to learn many facts related to this field.

I am also thankful to Prof. Ravi S. Kulkarni, who gave the idea of this work and discussed the problem, and Prof. Anant R. Shastri, for his guidance in better understanding of the subject.  I am also thankful to my friends and fellow research scholars (specially  Dinesh Kumar and Gopal Datt) for their help and discussion during the course of my study. 

I also wish to express my gratitude to the U.G.C. for granting me the fellowship which was a great financial assistance in the completion of my M. Phil program.

I am sincerely thankful to my parents for motivating me to do higher studies and encouraging me to achieve my long cherished goal.

Above all, I thank, The Almighty, for all his blessings bestowed upon me in completing this work successfully.

\tableofcontents

\mainmatter
\chapter{ What is a Manifold}\label{ch1}
\section{Introduction}

The starting point of the theory of 2-dimensional manifolds was the Euler-theorem. Euler's famous formula for a polyhedron is: $v-e+f=2$, where $v$ is the number of vertices, $e$ is the number of edges and $f$ is the number of faces of the polyhedron.

A manifold is a figure or a space as a geometric object and is an important notion in modern mathematics. Examples of manifolds are curves, surfaces etc, beside they might be of higher dimensions. The dimension of a manifold is the number of independent parameters needed to specify a point on it. Since a point on a curve can be described by a single parameter, it is called a 1-dimensional manifold. The simplest example of a 1-manifold is the real line, a point on the real line is a real number. Other examples include graph of a real valued continuous function taking real values, or a circle, we can specify a point on the circle by its angle.

Manifolds of dimension two are called surfaces. We see examples of surfaces in our daily life. Balloons, cylindrical cans, doughnut shaped surface, called a torus, soap films etc.  In surfaces two coordinates are needed to determine a point. For instance, if we cut  a cylinder lengthwise, then we can unroll it, to lie  flat on a plane. This  indicates that surfaces are inherently 2-dimensional objects and should be described by two coordinates. Although a parameter value  determines a point, different parameter value may correspond to the same point. In general, it is impossible to obtain the whole surface, however we round a region in the plane. But in every case, as long as we stay close to any point, there is a one to one correspondence between nearby points on the surface and the corresponding parameter. Then every surface can be parameterized locally by a small region in the plane.

There are different kind of manifolds, but the simplest are the topological manifolds, which we only require to be looks like $\R^n$ locally. As a preliminary definition, we can define an $n$-dimensional manifold as a topological space which is locally homeomorphic to a subset of $\R^n$. A connected 2-manifold is called a surface.

In geometry one develops a definition of when figures are geometrically same or congruent. Congruent figures will have same geometric properties such as lengths, angle measurement, areas, volumes etc. Topology is the branch of mathematics that is concerned with properties of spaces which are left unchanged by continuous deformations that we have called a homeomorphism,   defined below. The geometric properties such as length of curve, areas, surface areas or volumes are not preserved by homeomorphisms in general. A standard problem in topology is the classification of spaces upto homeomorphisms. Such a classification is known for curves, manifolds of dimension 1:

\begin{theorem}
(see ~\cite{lee}.) A connected 1-manifold is homeomorphic to a circle $\s^1$ if it is compact or to real line $\R^1$ if it is not.
\end{theorem}

Although the classification of surfaces was already completed in the beginning of twentieth century. A rigorous proof requires a precise definition of a surface, a precise notion of triangulation and a precise way of determining whether two surfaces are topologically same or not. This requires the notion of Euler-characteristic and orientability of surfaces and some notions of algebraic topology such as fundamental group and homology groups.

\textbf{Definition:} Let $X$ and $Y$ be two topological spaces. A function $f:X \to Y$ is continuous if the inverse image of each open set of $Y$ is open in $X$. Moreover, $f$ is a homeomorphism if it is one-one, onto and has continuous inverse. Such spaces are called topologically equivalent or same.

In 1860s A. F. M$\ddot{o}$bius and C. Jordan independently gave the classification of compact orientable surfaces. The classification of compact nonorientable  surfaces was published in the paper of W. van Dyck in 1888. The first rigorous proof of the classification theorem for compact surfaces was given in 1907 by Max Dehn, a student of Hilbert, and Poul Heegard. After Brahana's exact algebraic proof in 1921, (see ~\cite{brahana}), some additional proofs were made in the twentieth century (see  ~\cite{massey},  ~\cite{ahlfors},  ~\cite{moise}).

Every proof of the classification theorem for surfaces requires that every surface can be triangulated. All the papers written before 1930, assumed all surfaces were triangulated. The first rigorous proof that surfaces can be triangulated was published by Tibor Rad$\acute{o}$ in 1925 (see  ~\cite{rado}).  Intuitively, a surface can be triangulated if it is homeomorphic to a space obtained by pasting triangles (2-simplexes) together along edges. Any such homeomorphism is called a $triangulation$ of the surface, and any surface that admits such a homeomorphism is said to be triangulable. For this we have to define the notion of a simplicial complex. These are spaces constructed from simplexes which are points, line segments and triangles etc. They provide a highly useful way of constructing topological spaces and play a fundamental role in geometry and algebraic topology.

The main problem for the the classification of surfaces is to find invariants to decide whether two surfaces are homeomorphic or not. Here we study  various properties of surfaces and classify them accordingly. For this we search for topological invariants of surfaces, that is preserved under homeomorphisms, together  with technique of calculating them. The invariants describe the structure of surfaces in terms of numbers. We can assign a numerical invariant to every compact surface, its Euler characteristic. For a triangulated surface, if $|V|$ is the number of vertices, $|E|$ is the number of edges, and $|F|$ is the number of faces in any triangulation of the surface. Then its Euler characteristic is given by
$$\chi=|V|-|E|+|F| \ .$$
We can show that homeomorphic surfaces have the same Euler characteristic. For proving that the Euler characteristic is a topological invariant of surfaces, we  will have to define homology groups.

It is not easy to define precisely the notion of orientability of a surface and to prove that it is topologically invariant. Johann Listing, a student of Gauss, published his description of the M$\ddot{o}$bius band in 1861-1862, but A. F. M$\ddot{o}$bius independently described its properties 
\begin{center}
\includegraphics[width=1\columnwidth]{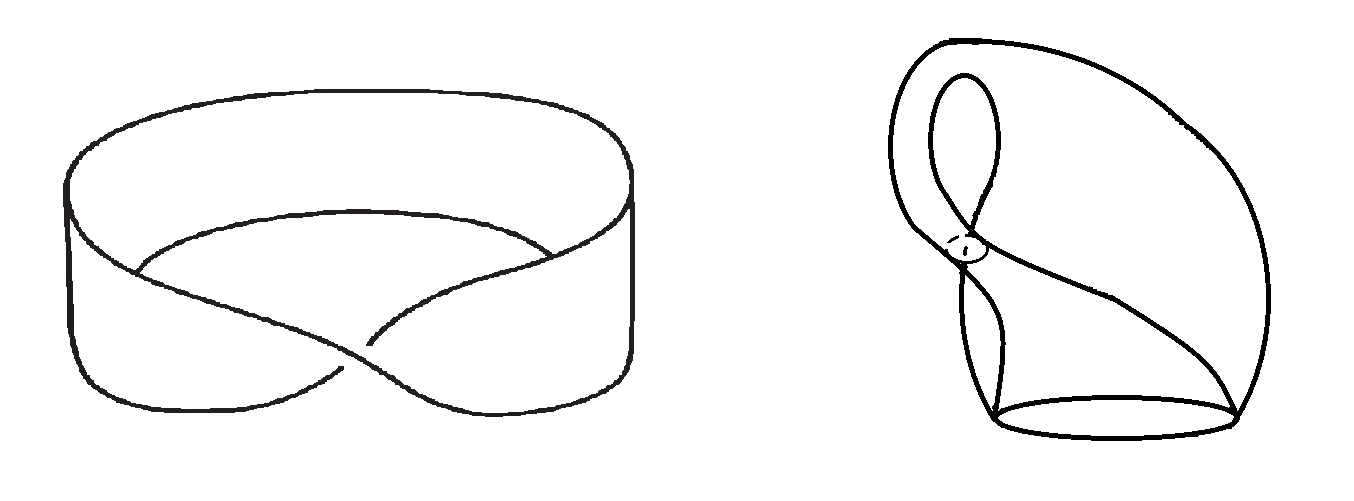}
\end{center}
\begin{center}
Figure 1.1
\end{center}
in terms of nonorientability  in 1865. In 1882 Felix Klein described the Klein bottle, one of the most famous nonorientable surface which cannot be embedded in $\R^3$ without self intersection, see Figure 1.1.

Consider a transparent surface and place a small circle on it, so that it is visible on both sides. Assign a direction on it, clockwise or anti-clockwise. On one side it is clockwise and on the other side it is anti-clockwise. Consider an ant that travels on this surface, started out at any point on this circle with a choice of direction on the circle and carried this direction with him. After traveling along any closed curve, when he comes back to initial point, either the direction of circle is opposite or same. If it is possible to travel along a closed path in such a way that when it returns to its original position, the direction is reversed, then the surface is nonorientable. 

The problem of noncompact surfaces did not occur in the nineteenth century. In 1923, B.  V. Ker$\acute{e}$kj$\acute{a}$rt$\acute{o}$  gave the classification theorem of noncompact surfaces (see ~\cite{kerekjarto}), which is known as Ker$\acute{e}$kj$\acute{a}$rt$\acute{o}$'s theorem. Ker$\acute{e}$kj$\acute{a}$rt$\acute{o}$'s main idea was that he defined the ideal boundary of a noncompact surface with this he compactify the noncompact surface to a compact surface. Later, in 1963 I. Richards and in 1971 M. E. Goldman proved the result more precisely.  There the case of surfaces without boundaries were considered. In 2007, K. I. Mischenko and A. O. Prishlyak gave a complete classification of noncompact surfaces with boundary.


\section{Definition and Examples of Manifolds}

 \textbf{Definition:}
Let $M$ be a topological space. The conditions for $M$ to be a manifold are as follows.

First, $M$ should satisfies the \emph{Hausdorff  axiom}. That is, for any two distinct points $p,q \in$ $M$, there exist disjoint neighborhoods $U$ and $V$ of  $p$ and $q$  respectively in $M$. Such a space is called a Hausdorff space.

The second condition is that for any arbitrary point $p$ of $M$, there exist an open neighborhood $U$ of $p$ homeomorphic to an open set $V$ of $ \R^n $.  Such a neighborhood is called a $Euclidean \ neighborhood$ of $p$. In this case `$n$' is called the local dimension of $M$ at $p$. If  $M$ is connected then the local dimension is constant and is called the dimension of $M$.

The third condition, is the $second \ countability$, claims that there exist a countable basis for $M$. That is there exist countably many open sets $ \{U_i\}_i$ of $M$ such that for each open set $U$ of $M$ and for each $p\in U $, there is some natural number $i$ such that $ p\in U_i \subset U $ . Whereas the Hausdorff axiom ensures that there are enough open sets, second countability ensures that they are not too many.

A topological space satisfying the above three conditions is called an \emph{n-dimensional  manifold} or simply an \emph{n-manifold}. In particular a connected 2-manifold is called a \emph{surface}. Also a compact connected 2-manifold is called a \emph{closed surface} whereas a noncompact connected 2-manifold is called an \emph{open surface}.

\textbf{Examples:}\\
(1) The most trivial example of an $n$-manifold is $\R^n$ itself. Infact any open subset of $\R^n$ is an $n$-manifold. More generally, any open subset of any $n$-manifold is an $n$-manifold.\\
(2) Another example of an $n$-manifold is surface of an $n$-sphere, $\s^n = \{ (x_1, x_2, \ldots ,x_{n+1})\in \R^{n+1} | \  x^2_1+x^2_2+\ldots + x^2_{n+1}=1 \}$. One way to see that it is locally Euclidean is by stereographic projection. Let $p_+$=(0,0,\ldots ,0,1) and $p_-$=(0,0,\ldots ,0,$-$1) be the north and south pole of $\s^n$. Then the stereographic projection from $U_+=\s^n \setminus p_+$ and  $U_-=\s^n \setminus p_-$ onto $ \R^n$ are homeomorphisms defined as:
$$f:U_+ \to \R^n$$
$$f(x_1, x_2, \ldots ,x_{n+1})= \frac{1}{1-x_{n+1}}(x_1, x_2, \ldots ,x_n) \ .$$
Similarly we can define a homeomorphism from $U_-$ to $\R^n$.  It is Hausdorff and second countable being a subset of $\R^{n+1}$, and hence an $n$-manifold.\\

\begin{center}
\includegraphics[width=2.2\columnwidth]{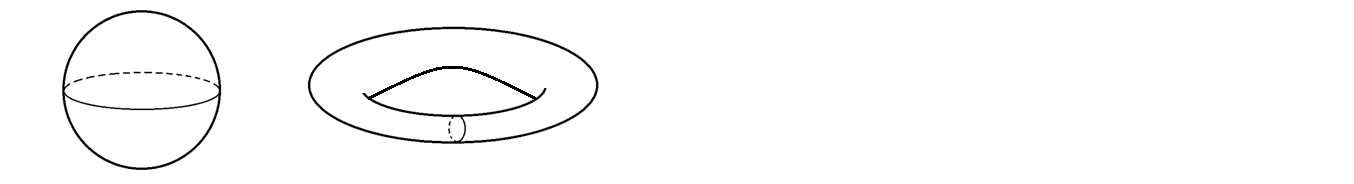}
\end{center}
\begin{center}
Figure 1.2
\end{center}
(3) If $ M_1$ and $M_2$ are manifolds of dimension $n_1$ and $n_2$, then the  product $M_1\times M_2$ is again a manifold and the dimension adds up. Given any $p=(p_1, p_2) \in M_1\times M_2$, there exists a neighborhood $U_i$ of $p_i$ and a homeomorphism $\Phi_i$ from $U_i$ to an open subet of $\R^{n_i}$, for each $i=1,2$. Then the product map $\Phi_1 \times \Phi_2$ is a homeomorphism from $U_1 \times U_2$ to a neighborhood of $\R^{n_1+n_2}$.  A particularly important example is of a torus $\mathbb{T}^2 = \s^1\times \s^1 $  which is a 2-manifold.\\
(4) Next consider the real projective plane $\p2$. It is the set of all lines through the origin in $\R^3$. $\p2$ can be interpreted as the quotient, $\p2 = \R^3$$\setminus$$\{0\} / \sim$, where `$\sim$' denotes the equivalence relation of points lying on the same line through origin: for $x,y \in \R^3$, we have $ x \sim y$  if and only if $x = \la y$, for some nonzero real number $\la$. The equivalence classes are regarded as the points of space $\p2$. The point $x = (x_1,x_2,x_3) \in \p2$, where atleast one of the coordinate must be non zero, is usually denoted by one of its representative as follows:
$$[x_1 : x_2 : x_3] = \{\la(x_1,x_2,x_3) | \ \la \ne 0 \}.$$
We topologize $\p2$ by giving it the quotient topology with respect to the map $q : \R^3 \setminus\{0\} \to \p2$. Then a subset $U$ of $\p2$ is open if and only if its inverse image $q^{-1}(U)$ is open in $\R^3 \setminus \{0\}$.

Consider the set $U_i = \{[x_1 : x_2 : x_3] \in \p2 |\  x_i  = 1 \}$. Then sets $U_i$, for $i=1,2,3$ covers $\p2$ and  each of the set $U_i$ is homeomorphic to $\R^2$. We can define a homeomorphism as  follows:  for $i=1$ the map $f_1 : U_1 \to \R^2$ is defined by 
$$f_1([x_1 : x_2 : x_3]) = (x_2 , x_3).$$

\begin{center}
\includegraphics[width=1\columnwidth]{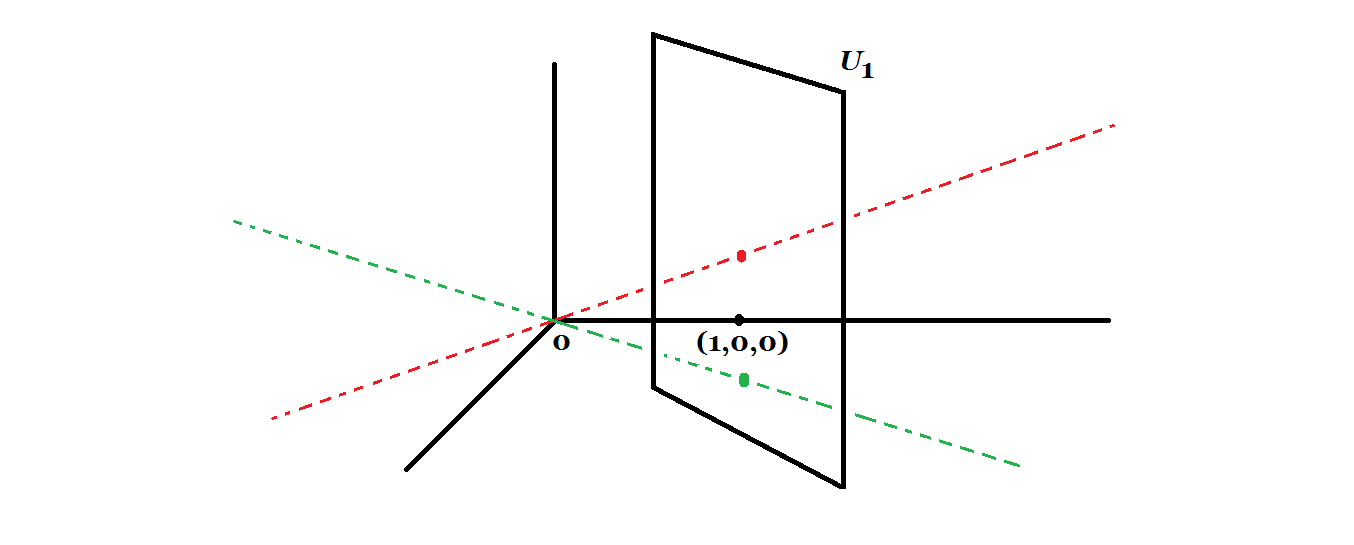}
\end{center}

\begin{center}
Figure 1.3
\end{center}

Similarly we can define for other $U_i ' s$. Since this quotient map `$q$' is both an open map as well as closed map, one can show that $\p2$ is second countable Hausdorff space and hence a 2-manifold. On the same line we can define an $n$-dimensional real projective space $\p n$.

At the same time one can define $\p2$ to be the quotient space of $\s^2$ obtained by identifying each point $x$ with its antipodal point $-x$ and let $p : \s^2 \to \s^2 /\sim$ denotes the quotient map. It makes exactly the same identification as $q|_{\s^2}$. Also one can go a step further to define $\p2$ as a quotient of a closed unit disk $\D^2 \subset \R^2$ by identifying the antipodal points on $\s^1$. For $\D^2$ is homeomorphic to the closed upper hemisphere $\s^2_+ = \{(x_1,x_2,x_3) \in \s^2 : x_3 \ge 0 \}$. And $p|_{\s^2_+}$ is the quotient map that identifies only the antipodal points of the equator $\s^1 \times \{0\}$.

\textbf{Manifolds with boundary:} An $n$-manifold with boundary is a second countable Hausdorff space in which every point $p\in M$ has a neighborhood homeomorphic to\\
(N1) either an open subset of $\R^n$; or\\
(N2) a neighborhood of origin in the upper half space $\h^n=\{(x_1,x_2,...,x_n)\in \R^n \colon x_n \ge 0\}$ with $p$ mapped to origin.

Any point $p \in M$ with neighborhood of type (N2) is called a boundary point. We denote  the set of all points of $M$ that are mapped to $ \partial \h^n $ under the homeomorphism by $\partial M$. If $\partial M \neq \phi $ we say that $M$ is a manifold with boundary. A surface with boundary is called a \emph{bordered  surface}. The interior of $M$ is denoted by $int(M)=M \setminus\partial M$.

\textbf{Examples:}
\\
(1) The upper half space $\h^n$ is obviously a manifold with boundary, as is any closed ball in $\R^n$ or a closed interval [$a,b$] for $a,b\in \R$. \\
(2) Disks and cylinders are all 2-manifolds with boundary homeomorphic to a space obtained by removing open disk(s) from surface of a sphere $\s^2$, one in case of a disk and two in case of a cylinder.\\
(3) The space obtained by removing an open disk from the projective plane $\p2$ is a bordered surface homeomorphic to the M$\ddot{o}$bius band, we shall explore this space in detail later.

\begin{center}
\includegraphics[width=2.2\columnwidth]{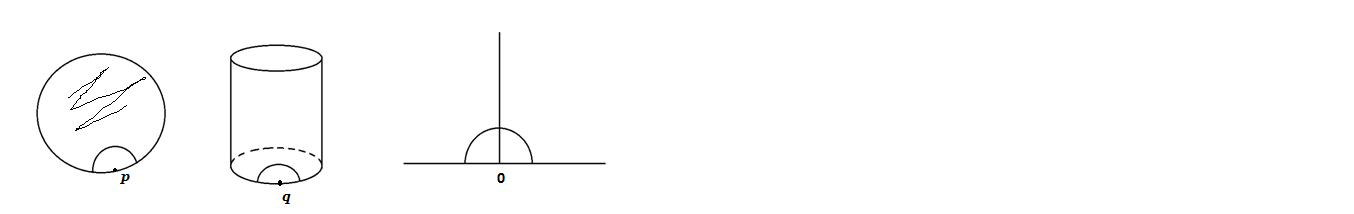}
\end{center}

\begin{center}
\includegraphics[width=1\columnwidth]{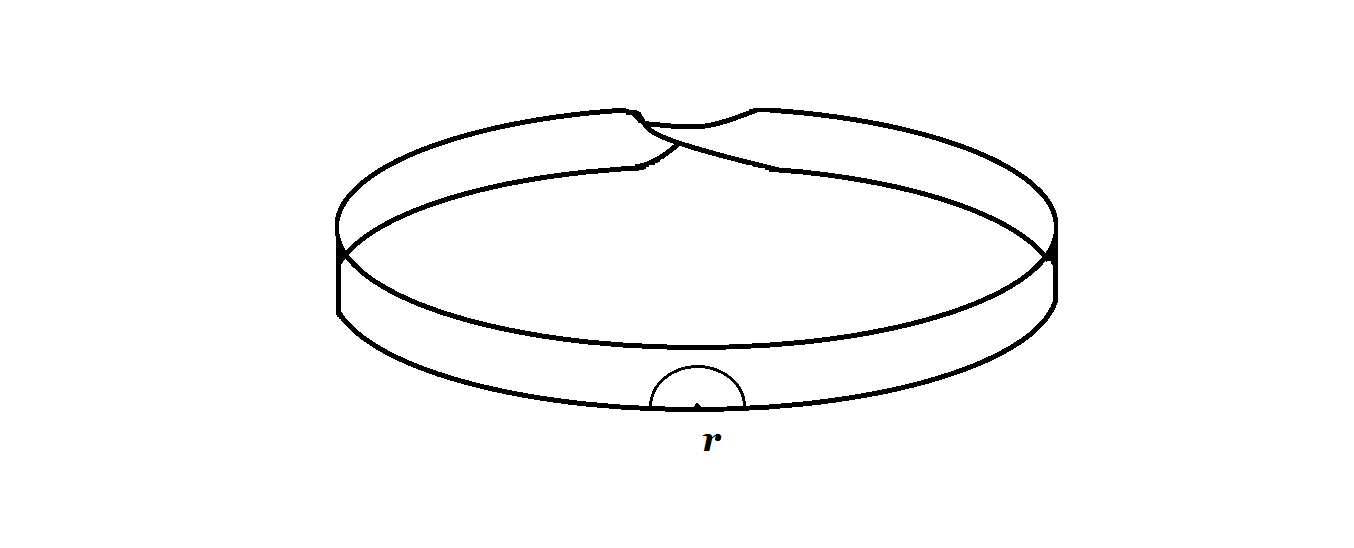}
\end{center}

\begin{center}
Figure 1.4
\end{center}

\begin{theorem}\label{ch1,sec2,th1}
Let $M$ be an $n$-manifold with boundary. Then either $ \partial M $ is empty or a manifold of dimension ($n$-1).
\end{theorem}
\begin{proof}
If $\partial M$ is empty then by definition of $\partial M$ each point has a neighborhood of type (N1). In that case $M$ will be a manifold .

Now let $\partial M \neq \phi$. Then for each $x \in \partial M$, there is a neighborhood $U$ of $x$ of type (N2), that is, there is a homeomorphism $ f\colon U \to V$, where  $V$ is a neighborhood of origin in $ \h^n$. Put $U'=f^{-1}(V\cap (\R^{n-1}\times \{0\})) $. Then  $ U'=\partial M \cap U $ is open in $\partial M$  and $ g=f\vert_{U'} $ is a homeomorphism from a neighborhood of $x$ to an open subset of $\R^{n-1}$. Thus $\partial M$ is a ($n$-1)-manifold.
\end{proof}

\textbf{Remark:} A manifold is a manifold with boundary such that $\partial M = \phi$, but a manifold with boundary such that $\partial M \ne \phi$ is not a manifold.

\begin{theorem}
The boundary of a compact manifold has a finite number of components.
\end{theorem}

\begin{proof}
Let $M$ be a compact $n$-manifold with boundary. Since boundary is a closed subset of the compact space $M$ and so is a compact set. It follows from Theorem \ref{ch1,sec2,th1} that $N = \partial M$ is a compact ($n$-1) manifolds. For each point $x \in C$ a component of $N$, there is a connected open set $U \subset C \subset N$ and $x \in U$, so every component of $N$ is an open set of $N$. Since $N$ is compact the components are finite.
\end{proof}


\chapter{Triangulation of 2-Manifolds}\label{ch2}
\section{Simplicial Complex}\label{ch2,sec1}

A set $X=\{x_0,x_1,\ldots,x_k \}$ of points of $\R^n$ is said to be $linearly \ independent$  if for any arbitrary reals $\alpha_0,\alpha_1,\ldots , \alpha_k$,
$$\alpha_0x_0+\alpha_1x_1+\ldots+\alpha_kx_k=0$$
$$\qquad \qquad \qquad \Rightarrow \alpha_i=0, \qquad {\rm{for}}\  i=0,1,\ldots,k.$$

We say that $k+1$ points $\{x_0,x_1,\ldots,x_k \}$ are in $general \ position$ (or {\it{affinely  independent}}) if the set $\{ x_1- x_0,x_2- x_0,\ldots , x_k- x_0 \}$ is linearly independent.

\textbf{Definition:}
Given a set $A=\{ a_0,a_1,\ldots , a_k\}$ of points in general position in $\R^n, n \ge k$, the {\it{k-dimensional  simplex}} or $k$-$simplex$ $\sigma ^k =\ <a_0,a_1,\ldots ,a_k>$ is the $convex \ hull$ of $A$ given by
$$\qquad \qquad \sigma ^k = \{ \sum_{i=0}^{k} \beta _i a_ i  \	\vert \sum_{i=0}^{k} \beta_i =1 \	{\rm{and}}  \	0 \le \beta_i \le 1, \ {\rm{for \ each}} \  i=0,1,\ldots ,k\}$$

The points $a_i \in \sigma ^k$ are called vertices of the  $k$-simplex  $\sigma ^k$.

\textbf{Examples:}
\\
(1) The 0-simplex consist of one point only.\\
(2) The 1-simplex $\sigma^1 = <a_0,a_1>$ spanned by two distinct points of $\R^n$ is a line segment with end points $a_0$ and $a_1$.\\
(3) The 2-simplex $\sigma^2=<a_0,a_1,a_2>$ spanned by 3 non-collinear points $a_0,a_1$ and $a_2$ is the triangle with the vertices $a_0,a_1,a_2$.\\
(4) The 3-simplex $\sigma^3=<a_0,a_1,a_2,a_3>$ spanned by 4 non-planar points is a tetrahedron with vertices $a_0,a_1,a_2,a_3$.

\begin{center}
\includegraphics[width=1\columnwidth]{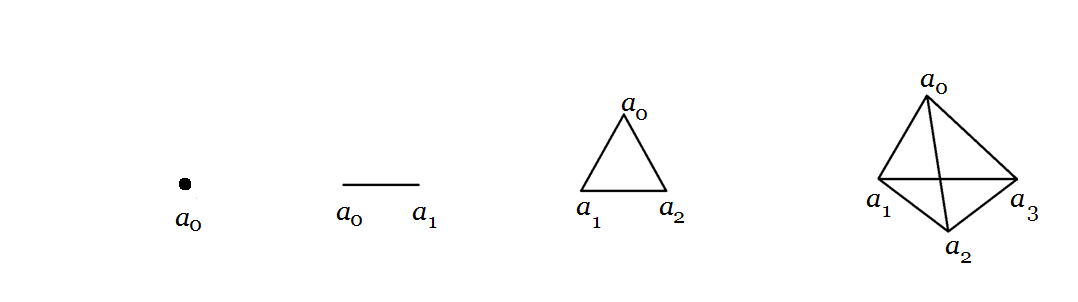}
\end{center}

\begin{center}
Figure 2.1 
\end{center}

Let $B$ be a subset of $A$. For $p \le q$, the convex hull $\sigma^p$ of $B$  is called a {\it{p-face}}  of $ \sigma^q$, we denote it by $\sigma^p \le  \sigma^q$. A 0-face is simply a vertex whereas a 1-face is an edge. If $\sigma^p \ne \sigma^q$ we say that $\sigma^p$ is a proper face of $\sigma^q$. 

One can see that every $k$-simplex is homeomorphic to $\D^k$, thence a $k$-manifold with boundary. Thus the boundary of a $k$-simplex is the union of its ($k-1$) faces and its interior (also called {\it{open k-simplex}}) is the simplex minus its boundary. 

\textbf{Remark:} It is convenient to regard the empty set as the unique ($-$1)-dimensional simplex.

\textbf{Definition:} A $Euclidean \ complex$ or $Simplicial \ complex \ K$  is a collection of simplexes in $\R^n$, which satisfies the following conditions: \\
(1) If $\sigma \in K$, then all faces of $\sigma$ are also in $K$. \\
(2) If $\sigma,\tau \in K$ then either $\sigma \cap \tau = \phi \ or \  \sigma \cap \tau$ is a common face of both $\sigma$  and  $\tau$.\\
(3) Any $x \in \sigma$ in $K$ lies in an open set $U$ that intersects only finite number of simplexes in $K$.

\textbf{Example:} Let $\sigma ^2 =<a_0,a_1,a_2>$ be a 2-simplex. The set $$K=\{<a_0>,<a_1>,<a_2>,<a_0,a_1>,<a_0,a_2>,<a_1,a_2>,<a_0,a_1,a_2>\}$$
 of all faces of $\sigma ^2$ is a simplicial complex.

 The dimension of a simplicial complex $K$ is defined to be the maximum dimension of any simplex in $K$. A subset $L$ of $K$ is said  to be a subcomplex of $K$ if $L$ contains all the faces of every simplex in $L$, that is, $L$ is itself a simplicial complex. The vertices of every 
\begin{center}
\includegraphics[width=1\columnwidth]{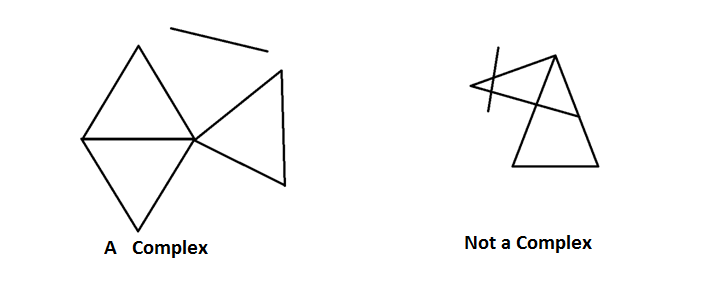}
\end{center}

\begin{center}
Figure 2.2
\end{center}
simplexes of simplicial complex $K$ are briefly referred as the vertices of $K$. Also for any natural number $i,\   K^i $, the $i$-$skeleton$ of $K$, is the subcomplex consisting of all simplexes of $K$ having dimension less than or equal to $i$.

Given a simplicial complex $K$, let $|K|$ denotes the union of all the simplexes of $K$. Then $|K|$ with the subspace topology induced from $\R^n$ is a topological space. A topological space $M$ for which there exist a simplicial complex $K$ such that $|K|$ is homeomorphic to $M$ is called a $polyhedron$, in this case $M$ is said to be triangulable and $K$ is called a triangulation of $M$. Evidently a polyhedron may have several triangulations.

\textbf{Examples:}\\
(1) Consider the 2-simplex $\sg^2 = <a_0,a_1,a_2>$. Then the set of all faces of $\sg^2$
$$K=\{ <a_0>,<a_1>,<a_2>,<a_0,a_1>,<a_0,a_2>,<a_1,a_2>,<a_0,a_1,a_2>\}$$
is a simplicial complex. Then $|K|$ is a triangle homeomorphic to a unit disk $\D^2$ in $\R^2$. In general the complex $K$ consisting of all faces of some $n$-simplex $\sg^n$ is a triangulation of unit disk $\D^n \subset \R^n$.\\
(2) Let $\sg^{k+1}$ be a $(k+1)$-simplex and $K$ be the complex consisting of all faces of $\sg^{k+1}$. Then the subcomplex $K^k$ consists of all proper faces of $\sg^{k+1}$ is a triangulation of $\s^k$.

\begin{center}
\includegraphics[width=1\columnwidth]{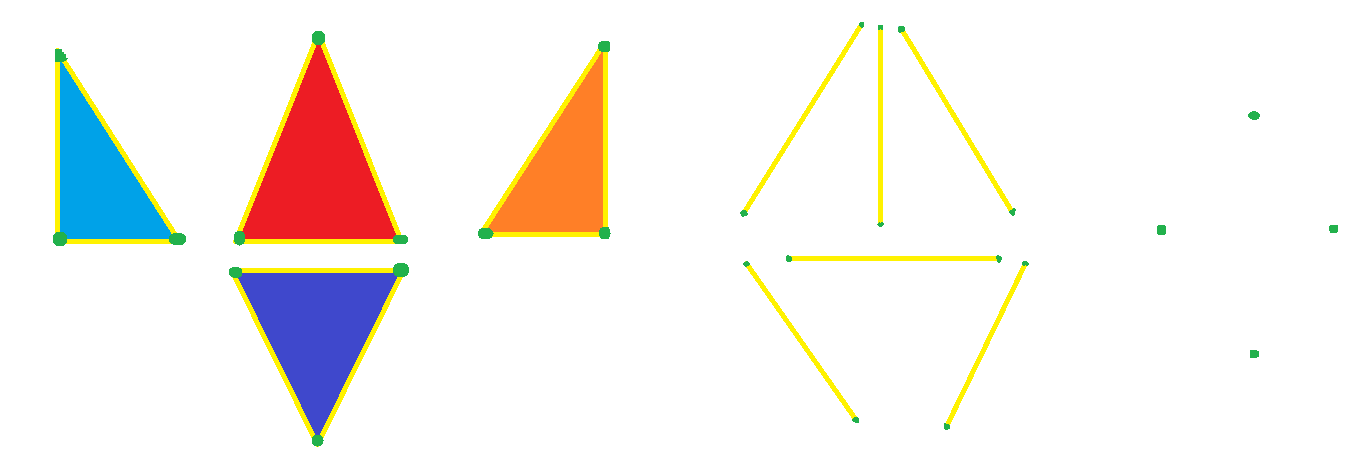}
\end{center}
\begin{center}
\includegraphics[width=1.2\columnwidth]{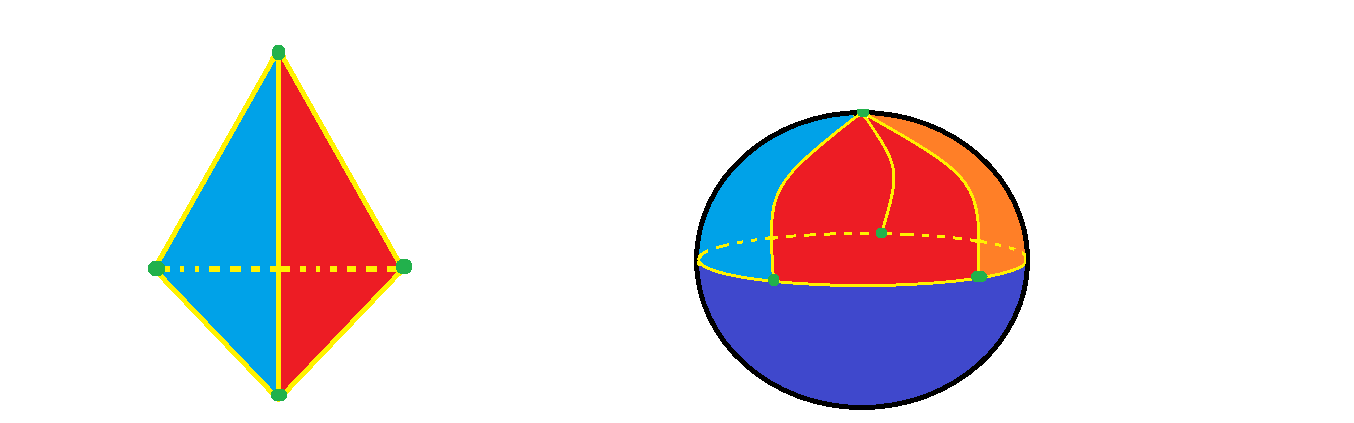}
\end{center}

\begin{center}
Figure 2.3
\end{center}

\section{Subdivisions of a Complex}
 
We begin by defining the notion of subdivision or refinement of a complex. It gives very simple idea to transform a simplicial complex to obtain another one.

\textbf{Definition:} A complex $K$ is called a subdivision of the complex $L$ if \\
(1) $\vert K \vert =\vert L \vert$; and \\
(2) for every simplex $\sigma' \in K$ there is a simplex $\sigma \in L$ such that $\sigma' \subset \sigma$.

Suppose that $K$ is a subdivision of $L$. We say that it is an $elementary \ subdivision$ if $K$ contains exactly one more vertex than $L$ see Figure 2.4 below. If $L$ is a finite complex 
\begin{center}
\includegraphics[width=2.2\columnwidth]{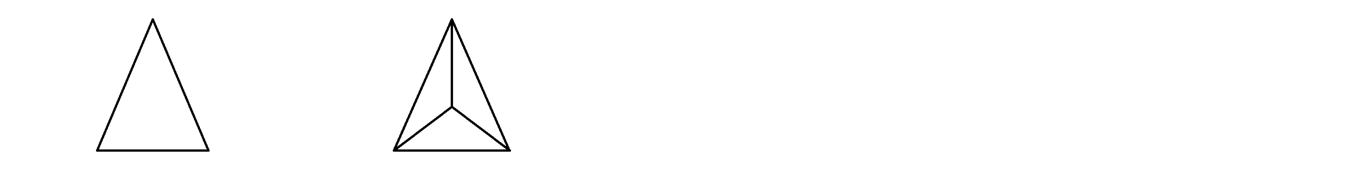}
\end{center}
\begin{center}
Figure 2.4
\end{center}
and $K$ is any subdivision of $L$, there is a finite sequence $L$=$L_0,L_1,\ldots ,L_k$=$K$ of complexes such that $L_{i+1}$ is an elementary subdivision of $L_i$, for $i=0,1,2,\ldots,k-1$.

\textbf{Barycentric Subdivision:}
Given any $n$-simplex $\sigma^n=<v_0,v_1,\ldots ,v_n> \subset \R^m$. The \emph{barycentric coordinates} of any point $p \in \sigma^n$ is $(\alpha_0,\alpha_1,\ldots ,\alpha_n)$  where 
$$p=\sum_{i=0}^n \alpha_iv_i \ ,\ \sum_{i=o}^n \alpha_i=1, \  \alpha_i \ge0 \ \ {\rm{for \ each}} \  i.$$
The $barycenter \ b_{\sigma^n}$ of $\sigma^n$ is the point of $\sigma^n$, all of whose barycentric coordinates are equal, and is  given by
$$b_{\sigma^n}=\sum_{i=0}^n \frac{1}{n+1} v_i$$

The barycenter of  a vertex $v$ is $v$ itself. The barycenter of an edge is simply its midpoint.

\textbf{Definition:} Let $K$ be a Euclidean complex. We define a complex $b$K called the $barycentric \ subdivision$ of $K$ by induction on the dimension of $K$ as follows \\
(1) Set $bK^0=K^0$. \\
(2) Assuming that $bK^i$ is defined, $bK^{i+1}$ is the union of $bK^i$ and the set of all simplexes of the form $v\sigma^i=<v,v_0,\ldots ,v_i>$, where $v$ is the barycenter of a simplex $\sigma^{i+1}$ of $K$ and $<v_0,\ldots ,v_i>=\sigma^i \in bK^i, \ \sigma^i \subset \sigma^{i+1}$.\\
Since the dimension of the simplexes in a Euclidean complex is bounded, so this process terminates resulting $b$K.

\begin{center}
\includegraphics[width=1\columnwidth]{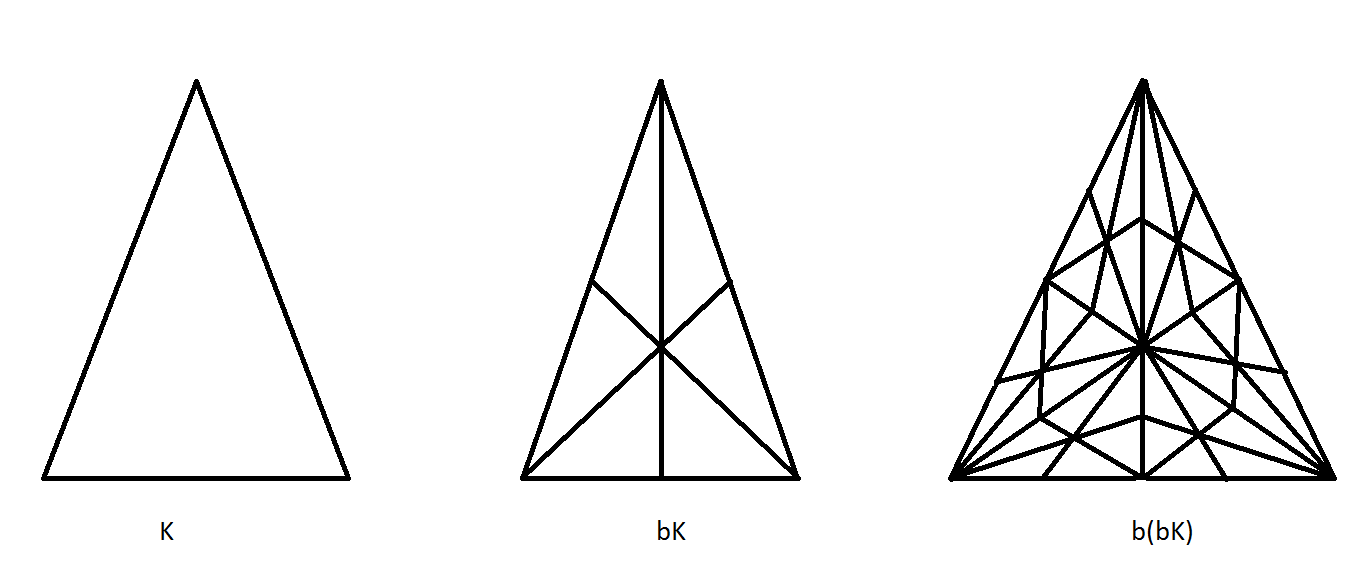}
\end{center}

\begin{center}
Figure 2.5
\end{center}

\textbf{Definition:} In a complex $K$, for each vertex $v$, $St(v)$, the star of $v$ in $K$, is the complex consisting of all simplexes of $K$ that contain $v$ together with all their faces. Also for $v\in K$ the barycentric star of $v$ denoted by $b(St(v))$ is the complex consisting of all simplexes of $b$K that contain $v$, together with all their faces.

\textbf{Definition:} Let $K$ be a triangulation of $M$ and $A$ be a subcomplex of $K$. Then the union of all simplexes of $b(bK)=b^2K$ that intersect $A$ is called the \emph{second  barycentric  neighborhood} (or  SB-\emph{neighborhood}) of $A$ in $K$, and is  denoted by $N(A)$.

Consider a 2-simplex $\sigma=<v_0,v_1,v_2>$ and let $e_0=\ <v_0,v_1>,\ e_1= \ <v_0,v_2>,\ e_2= \ <v_1,v_2>$ be the edges. Then $\sigma$ can be expressed as a union 

$$\sigma=N(v_0)\cup N(v_1)\cup N(v_2)\cup N(b_{e_0})\cup N(b_{e_1})\cup N(b_{e_2})\cup N(b_{\sigma}) \ .$$
Each of them is homeomorphic to a closed disk with disjoint interior.

\begin{theorem}\label{ch5,th2}
If a simplicial complex K is a triangulation of a compact surface M (with or without boundary) then \\
\emph{(1)} every simplex of K is a face of a \emph{2}-simplex.\\
\emph{(2)} every edge of K is the common face of at most two \emph{2}-simplexes. \\
\emph{(3)} given any two \emph{2}-simplexes $\sigma '$ and $\sigma ''\in K$ there is a sequence of \emph{2}-simplexes $\sigma'=\sigma_1, \sigma_2, \ldots ,\sigma_n = \sigma''$ such that $\sigma_i$ and $\sigma_{i+1}$ have a common edge $e_i$ for i=1,\emph{2},\ldots ,n$-$\emph{1}. \\
\emph{(4)} the edges of K which are the faces of precisely one \emph{2}-simplex of K, together with all of their vertices form a triangulation of the boundary $\partial M$ of M.
\end{theorem}
The result is practically evident and we have the following 
\begin{proof}
(1) Let $\sigma\in K$ be a $k$-simplex which is not a proper face of any simplex of $K$. Let $b_{\sigma}$ be the barycenter of $\sigma$, then $b_\sigma$ lies in the interior of $\sigma$. Also there exist an open set $U$, such that $b_{\sigma} \in U$, homeomorphic to $\mathbb{B}^2$, open disk in $\R^2$. It follows that $k=2$.\\
(2) Suppose that the 2-simplexes $\sigma_1$ and $\sigma_2$ have a common edge $e$. It is easy to see that
$$C^2 = int( \sigma_1) \cup int(e)\cup int(\sigma)$$ 
is homeomorphic to the interior of a 2-simplex and is an open subset of $M$. It follows that $e$ is not a face of any 2-simplex other than $\sigma_1$ and $\sigma_2$. \\
(3) Suppose $\sigma' \in K$ is a 2-simplex and $K'$ denotes the subcomplex of $K$ consisting of all those 2-simplex $\sigma''$ together with all the faces such that there is a sequence of 2-simplexes $\sigma'=\sigma_1, \sigma_2, \ldots ,\sigma_n = \sigma'' \in K$ and $\sigma_{i+1}$ has an edge $e_i$ common with $\sigma_i$ for $i=1,2,\ldots ,n-1$. Let if possible there exist a 2-simplex $\sigma \notin K'$ then $\sigma$ has no edge in common with $K'$. It follows that $|K'|$ contains all or none of each  set $int(|St(v)|)$, interior of $|St(v)|$ for $v\in K^0$. Since M is connected we have $|K'|=|K|$. \\
(4) Since the boundary $\partial M$ of $M$ is a 1-manifold and so is contained in $|K^1|$ union of 1-skeleton of $K$. Let $e'_1,e'_2,\ldots ,e'_p$ denote the edges of the complex $K$ which are the faces of precisely one 2-simplex of $K$ and $e''_1,e''_2,\ldots ,e''_q$ denote the edges of the complex $K$ which are the faces of exactly two 2-simplexes of $K$.

If $x\in int(e'_i) , 1 \le i \le p$, then evidently $x\in \partial M$ so that $\bigcup_{i=1}^p  int( e'_i) \subset \partial M$. Since $\partial M$ is compact we have $\bigcup_{i=1}^p e'_i \subset \partial M$. On the other hand, if $x \in  \partial M \setminus \bigcup_{i=1}^p e'_i $,then $x \in e''_k$ for some $1\le k \le q$. And if $x\in int(e''_k)$, then there are two 2-simplexes $\sigma_1$ and $\sigma_2$ whose common face is $e''_k$, thus we obtain  the open set $int(\sigma_1) \cup int( e''_k) \cup int(\sigma_2)$ which is homeomorphic to $\mathbb{B}^2$ and contains the point $x$ so that $x \in int( M)$ a contradiction. Hence
$$\partial M \setminus \bigcup_{i=1}^p e'_i  \subset \bigcup_{i=1}^q \partial e''_i.$$
Since $\partial M$ is a 1-manifold whereas $\bigcup_{i=1}^q \partial e''_i $ is of dimension 0, we have $\partial M= \bigcup_{i=1}^p e'_i$.
\end{proof}

\textbf{Definition:} Let $K$ and $L$ be two simplicial complexes. A map $f:K\to L$ is called a simplicial map if $f$ maps $K^0$ to $L^0$ such that whenever $<a_0,a_1,\ldots ,a_n>$ is a simplex of $K$, then $\{f(a_0),f(a_1),\ldots ,f(a_n)\}$ determine a simplex in $L$, that is, cancelling repetitions, $<f(a_0),f(a_1),\ldots ,f(a_n)>$ is a simplex in $L$. Moreover if $f : K^0 \to L^0$ is a bijection and $<a_0,a_1,\ldots ,a_n>$ is a simplex of $K$ if and only if $<f(a_0),f(a_1),\ldots ,f(a_n)>$ is a simplex in $L$, then $f$ is said to be an $isomorphism$.

\textbf{Definition:} If $K$ and $L$ are two complexes, and have subdivisions $K'$ and $L'$ which are isomorphic, then $K$ and $L$ are combinatorially equivalent or simply equivalent.

Intuitively, $K$ and $L$ are equivalent if it is possible to pass from one to the other by a sequence $K$=$K_1, K_2, \ldots ,K_n$=$L$ of complexes, where either $K_{i+1}$ is a subdivision of $K_i$, or $K_i$ is a subdivision of $K_{i+1}$.

A simplicial map  $f : K \to L$ induces a continuous map  $|f| : |K| \to |L|$ defined as follows: for the restriction $f: \ <a_0,a_1,\ldots ,a_n> \  \to \ <f(a_0),f(a_1),\ldots ,f(a_n)>$ the coordinates of a point $|f|(x)$ are linear functions of those of  $x$. Moreover,  if  $f : K \to L$ is an isomorphism then $|f| : |K| \to |L|$ is a homeomorphism.


\section{Existence of Triangulation}\label{ch1,sec2}

 The problem of classification of surfaces requires that all surfaces are triangulable. In this section we shall discuss the existence of triangulation of a surface or a bordered surface. The proof of the existence of triangulation on a surface is due to Rad$\acute{o}$ (see~\cite{rado}). Also E. Hartman in ~\cite{hartman} introduces a procedure to build a mesh of triangles successively by starting with a point or a prescribed polygon.   For the proof of existence of a triangulation we shall need the Jordan-Sch$\ddot{o}$nflies Theorem:
\begin{theorem}
Let J be a 1-sphere in $\R^2$. Then every homeomorphism of J into $\R^2$ can be extended to give a homeomorphism of   $\R^2$ onto $\R^2$.
\end{theorem}

\textbf{Definition:} A Jordan arc  $f$ is a homeomorphism of the closed unit interval [0,1] into a topological space $S$. It is a Jordan curve when the end points of [0,1] are identified. By the interior of an arc we shall mean the open arc obtained by restricting $f$ to (0,1).

An open set $G$ on a surface $S$ is called a $Jordan \ region$ if its closure can be mapped topologically onto a closed disk, in such a way that $G$ corresponds to the open disk. The boundary of a Jordan region is thus a Jordan curve. A $cross$-$cut$ of a Jordan region $G$ is the interior of an arc $\gamma$ in $G$ such that $\gamma \cap \partial G = \partial \gamma$. As a consequence of Jordan-Sch$\ddot{o}$nflies Theorem we have

\begin{lemma}
A cross-cut of a Jordan region divides it into two Jordan regions.
\end{lemma}

\textbf{Definition:} An open covering of a surface $S$ by Jordan regions $\{J_n\}$ is said to be of finite character if \\
(1) each  $J_n$ intersects at most finitely many others. \\
(2) the intersection of any two boundaries consists of at most a finite number of points or arcs.

\begin{lemma}\label{ch2,sec3, lm3}
On a surface S there are sequences $\{V_n\}$ and $\{W_n\}$ of Jordan regions such that  $\overline{V}_n \subset W_n$ and $\{V_n\}$ covers S with the condition that no point belongs to infinitely many $\overline{W}_n$.
\end{lemma}

\begin{proof}
Since $S$ is locally Euclidean and second countable it has a countable basis of Jordan regions $\{U_i\}$. Also, each $U_i$ is a countable union of Jordan regions $\{U_{ij}\}$ such that $\overline{U}_{ij}\subset U_i$. Similarly each $U_{ij}$ is a countable union of Jordan regions $\{U_{ijk}\}$ such that $\overline{U}_{ijk} \subset U_{ij}$. Rearrange the $\{U_{ijk}\}$ in a sequence $\{V_{j_k}\}$ and if $V_{j_k}=U_{ijk}$ set $W_j=U_{ij}$.  Then every open set $F$ is a union of sets $V_{j_k}$  with $\overline{V}_{jk} \subset W_j \subset \overline{W}_j \subset U_i \subset F$.

Take $n_1=1$ and $n_k,(k>1)$ is the least integer such that 
$$\overline{V}_1 \cup \ldots \cup\overline{V}_{n_{k-1}} \subset V_1 \cup \ldots \cup V_{n_k}.$$
Since $\overline{V}_1 \cup...\cup\overline{V}_{n_{k-1}}$ is compact, such an $n_k$ will always exist. Also $n_k \le n_{k-1}$ implies
$$\overline{V}_1 \cup\ldots\cup\overline{V}_{n_{k-1}} \subset V_1 \cup \ldots \cup V_{n_{k-1}}\subset \overline{V}_1 \cup \ldots\cup\overline{V}_{n_{k-1}}.$$
Thus $V_1 \cup \ldots  \cup V_{n_{k-1}}$ is both open and closed and therefore  equal to $S$. This implies $S$ is compact and $V_1 ,\ldots  ,V_{n_{k-1}}$ and $W_1 ,\ldots  ,W_{n_{k-1}}$ are finite sequences with the desired properties. If $n_k > n_{k-1}$ for each $k$ then set $G_k = V_1 \cup \ldots \cup V_{n_k}$. Thus $\overline{G}_{k-1} \subset G_k$. Each $\overline{G}_k$ is compact and $\{G_k\}$ covers $S$. The set $G_{k+2} \setminus \overline{G}_{k-1}$, being open, can be represented as union of sets $V_{j_k} \subset \overline{W}_j \subset G_{k+2} \setminus \overline{G}_{k-1}$. Since $\overline{G}_{k+1} \setminus G_k \subset G_{k+2} \setminus \overline{G}_{k-1}$ is compact it is covered by a finite number of these sets. We denote  the sets in this finite covering by $V_{kl}$, and the corresponding $W_j$ by $W_{kl}$.

Since $\overline{W}_{kl}\cap G_{k-1}=\phi$, for $i \le k-$3 the  set $\overline{W}_{ij}$ does not intersect with $\overline{W}_{kl}$, for $\overline{W}_{ij} \subset G_{i+2} \subset G_{k-1}$. Hence every $\overline{W}_{ij}$ intersects finitely many $\overline{W}_{kl}$. Finally a suitable rearrangement of $\{V_{kl}\}$ and $\{W_{kl}\}$ gives the desired sequences.
\end{proof}

\textbf{Definition:} A set $\Gamma$ of Jordan arcs on a surface $S$ is said to be discrete if every point on $S$ has a neighborhood which meets finitely many arcs in $\Gamma$.

\begin{proposition}
The intersection of the arcs in $\Gamma$ with a region $G \subset S$ will also form a discrete set on G
\end{proposition}

\begin{proof}
For $x \in G$, let $U$ be a neighborhood of $x$ such that $\overline{U}\subset G$ and $U$ meets only finitely many arcs in $\Gamma$. If  infinitely many arcs  in $\Gamma \cap G$ meet $U$ then it follows that there is atleast one arc say $\gamma \in \Gamma$ such that infinitely many components of $\gamma \cap G$ meet $U$. Starting at some point on $\gamma$ and moving along $\gamma$, we can obtain a  sequence of distinct points on $\gamma$ say $x_1,y_1,x_2,y_2, \ldots $ such that $x_i \in G$ and $y_i \notin G$. Since $\gamma$  is compact, the sequence has an accumulation point $z$. But this point is the limit of $\{x_i\}$ it is in $\overline{U}$. Also it being the limit of $\{y_i\}$ it cannot be in open set $G$. This contradict that $\overline{U} \subset G$. Hence $U$ meets only finitely many arcs of $\Gamma \cap G$.
\end{proof}

\begin{lemma}\label{ch2,sec3, lm5}
Let $\Gamma$ be a discrete set of Jordan arcs $\gamma$ in a Jordan region $G$ and suppose that $p_1,p_2\in \overline{G}$ are not on any $\gamma \in \Gamma$. Then $p_1$ and $p_2$ can be joined by a Jordan curve, whose interior lying in $G$, with only a finite number of points on $\bigcup_{\gamma \in \Gamma} \{\gamma\}$.
\end{lemma}

\begin{proof}
Let $\sigma$ be any arc in $G$ that joins $p_1$ to $p_2$. Because of discreteness only finitely many $\gamma_i \in \Gamma$ meet $\sigma$. We denote the set of remaining arcs by $\Gamma '$ and replace $G$ by the component $G \setminus \Gamma '$ that contain $\sigma$. Then it is sufficient to prove the lemma when  $\Gamma$ is a finite set. We shall assume  that $\Gamma= \Gamma _1 \cup \ldots  \cup \Gamma_n$ where the arcs in $\Gamma_i$ are mutually disjoint. Let $P^n$ be the assertion of the lemma for all $\Gamma$ of this form, and let $P^n_m$ be the same assertion with the assumption that $\Gamma_1$ contains at most $m$ arcs.

Then $P^0$ is trivial, and $P_0^n$ is same as $P^{n-1}$. We shall show that $P^{n-1}$ and $P_{m-1}^n$ imply $P_m^n$. Thus $P^{n-1}$ imply $P_m^n$ for all $m$, and $P^n$ follows by the same reduction to the finite  case as above.

Let us assume that $P^{n-1}$ and $P_{m-1}^n$ holds. If some $\gamma ' \in \Gamma_1$ does not separate $p_1$ and $p_2$, then by $P_{m-1}^n$ we get an arc in $G \setminus \gamma '$ that joins $p_1$ to $p_2$. Otherwise consider an arc $\gamma_1 \in \Gamma_1$ that divides $G$ into $G_1$ and $G_2$ with $p_1 \in G_1$ and $p_2 \in G_2$. Since $\Gamma$ is discrete in G, we can choose a point $q \in \gamma_1$ and a neighborhood $U$ in $G$ that intersects only those arcs of $\Gamma$ to which $q$ belongs. By the choice of $q$ there exists a sub arc $\sigma_1$ at $q$ in $U$ that meets no $\gamma$ in $G_1$. By $P^n_{m-1}$ its endpoint in $G_1$ can be joined to $p_1$ by a Jordan arc $\sigma_2$ in the above manner. Then $\sigma_1 \cup \sigma_2$ is  a Jordan arc from $p_1$ to $q$  in $G_1$. By similar construction for $G_2$, we can obtain an arc from $p_1$ to $p_2$. 
\end{proof}

\begin{theorem}\label{ch2,sec3, thm6}
For any surface S, there exists an open covering by Jordan regions of finite character.
\end{theorem}

\begin{proof}
Consider the sequences $\{V_n\}$ and $\{W_n\}$ as in the Lemma \ref{ch2,sec3, lm3}. We shall find a covering of $S$ of finite character by Jordan regions $J_n$ such that $V_n \subset J_n \subset W_n$, for all $n$.

Take $J_1 = V_1$ and assume that $J_1,\ldots,J_{n-1}$ have already constructed. We shall find $J_n$ such that $\gamma_n = \partial J_n$ intersects $\gamma_1 \cup \gamma_2\cup \ldots  \cup \gamma_{n-1}$ at a finite number of points. If $\partial V_n$ has this property then we can choose $J_n = V_n$. Otherwise we represent $\overline{W}_n$ homeomorphically as a closed disk whose center corresponds to a point of $V_n$. Let $\Gamma = \gamma_1 \cup \gamma_2\cup \ldots  \cup \gamma_{n-1}$. Since $\Gamma$ has empty interior, we can find two points $p_1$ and $p_2$ in $W_n \setminus \overline{V_n}$ on distinct radii, which do not lie on $\Gamma$. Let $s_1$ and $s_2$ be the radial segments through $p_1$ and $p_2$ to the boundary of $W_n \setminus V_n$.  Then $s_1,s_2$ and the boundary arcs of $V_n$ between $p_1$ and $p_2$ cut $W_n \setminus \overline{V}_n$ into two Jordan regions $G_1$ and $G_2$. Then by Lemma \ref{ch2,sec3, lm5} join $p_1$ and $p_2$   by two Jordan arcs $\sigma_i$ in $G_i$ which meet $\Gamma$ in a finite number of points. Then $\gamma_n = \sigma_1 \cup \sigma_2$ is a Jordan curve which bounds a Jordan region $J_n \subset W_n$.

\begin{center}
\includegraphics[width=1 \columnwidth]{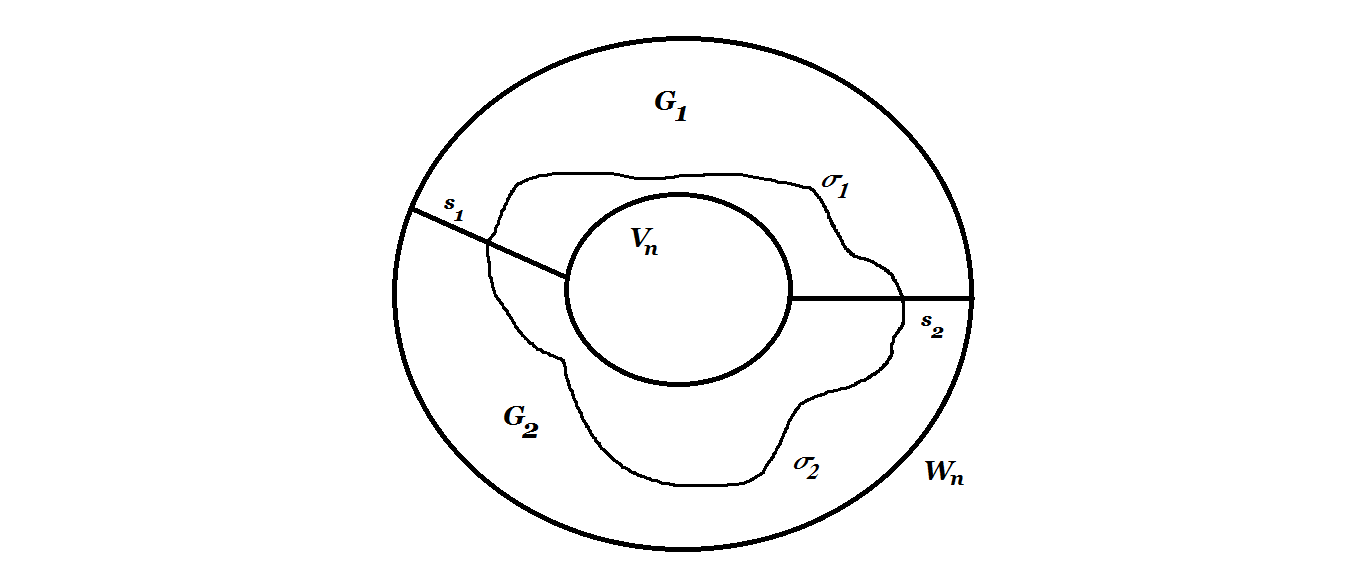}
\end{center}

\begin{center}
Figure 2.6
\end{center}

Since $J_n$ has no boundary point in $V_n$, either $V_n \subset J_n$ or $V_n$ lies in the complement of $J_n$. In the latter case whole $G_1$ would belong to the outside region determined by $\gamma_n$, except $p_1$ and $p_2$ that lie on its boundary. Thus $G_1$ has no boundary point in $J_n$, so that $G_1 \cap \overline{J}_n= \phi$. This is contradiction to the fact that $\sigma_1 \in G_1 \cap \overline{J}_n$. Thus we have $V_n \subset J_n \subset W_n$.
\end{proof}

\begin{theorem}\label{ch2,sec3, thm7}
Every  surface S is triangulable.
\end{theorem}

\begin{proof}
Let $\{J_n\}$ be an open covering of $S$ of finite character. From the covering we discard all those $J_n$, which are contained in $J_m$ for some $m\ne n$. The remaining collection will still form a covering.

If  $\gamma_n \subset J_m$ for some $m \ne n$, it means the Jordan region which $\gamma_n$ encloses in $J_m$ must be complementary to $J_n$, as $J_n$ is not contained in  $J_m$. In that case $J_n \cup J_m$ is both open and closed in $S$. Since $S$ is connected $S= J_n \cup J_m$. Hence  the surface $S$ is a sphere and we know sphere is triangulable.

Now each $\gamma_i$ intersects $\overline{J}_m$ along a finite number of cross-cuts. We begin by considering the cross-cuts on $\gamma_1$, if any. The first cross-cut divides $\overline{J}_m$ into two Jordan regions. One of these subregions is divided in the same way by the second cross-cut, and so on. The $\gamma_1$ divides $\overline{J}_m$ into finite number of Jordan regions. Next, if $\gamma_2$ intersects $\overline{J}_m$ then either these regions does not meet $\gamma_2$ or $\gamma_2$ divides them in finite number of Jordan regions. Since $\overline{J}_m$ meets only finitely many $\gamma_n$, this  process terminates in finite number of steps.

Let $\overline{J}_{mi}$ are the closed subregions of $\overline{J}_m$ obtained above. Then any two $\overline{J}_{mi}$ and $\overline{J}_{nj}$ are either identical or have disjoint interiors.

At the same time the arc $\gamma_m$ is divided by the points on some $\gamma_i$ or the end points of arcs that are common to $\gamma_m$ and $\gamma_i$. Let $\gamma_{mi}$ denote the sub arcs of $\gamma_m$. Any two $\gamma_{mi}$ and $\gamma_{nj}$ are either identical or have at most end points in common. Evidently the boundary of each $\overline{J}_{mi}$ is a union of arcs $\gamma_{nj}$.

\begin{center}
\includegraphics[width=1 \columnwidth]{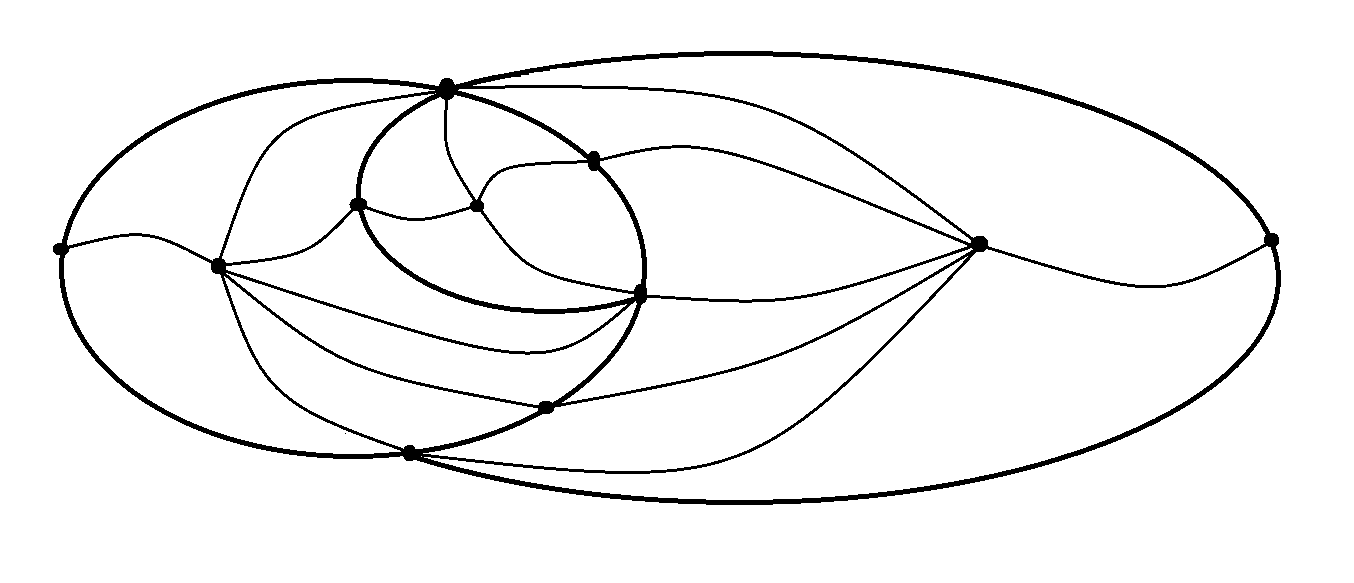}
\end{center}

\begin{center}
Figure 2.7
\end{center}

For the triangulation of the surface $S$, we consider a complex $K$ whose vertices are the end points of the arcs $\gamma_{mi}$ and an interior point of each $\overline{J}_{nj}$ and each $\gamma_{mi}$. The interior point of $\overline{J}_{nj}$ can be joined to the vertices on the boundary of $\overline{J}_{nj}$ by the Jordan arcs. They divide $\overline{J}_{nj}$ into closed triangular regions, see Figure 2.7. Thus we can define the corresponding 1-simplexes and 2-simplexes of $K$ associated with the triangular regions. Thus $S$ is triangulable and $K$ is a triangulation of $S$.
\end{proof}

\begin{theorem}
Let $K_1$ and $K_2$ be triangulations of a surface S. Then $K_1$ is equivalent to $K_2$.
\end{theorem}

\begin{proof}
Since $K_1$ and $K_2$ are two simplicial complexes, their 1-simplexes will form a discrete set. Therefore, as in Theorem \ref{ch2,sec3, thm6}, we can construct a covering of finite character by Jordan regions $\{J_n\}$ with the condition that $\partial J_n$ intersects 1-simplexes of $K_1$ and $K_2$ in a finite number of points.

Now we shall construct a new triangulation of $S$ whose 1-simplexes intersect those of $K_1$ and $K_2$ in a finite number of points. Following the proof of the Theorem \ref{ch2,sec3, thm7}, we consider a complex $K$ where the interior points of $\gamma_{mi}$ and $J_{nj}$ are taken such that they are not on any 1-simplexes of $K_1$ and $K_2$. Then by Lemma \ref{ch2,sec3, lm5}, the interior point of $J_{nj}$ can be joined to the vertices on the boundary of $J_{nj}$ by Jordan arcs intersecting 1-simplexes of $K_1$ and $K_2$ in finite number of points. If some end point $q$ of $\gamma_{mi}$ ,which is a vertex of $K$, lies on 1-simplex of $K_1$ or $K_2$ then  we consider a sub arc near $q$ in the corresponding $J_{nj}$ that intersects no other 1-simplexes of $K_1$ and $K_2$ so that its end point can be joined to the interior point of $J_{nj}$ in the above manner.

To show that $K_1$ and $K_2$ are equivalent it is sufficient  to show that $K$ is equivalent to both $K_1$ and $K_2$.  For this we construct a common subdivision $K_1'$ of $K$ and $K_1$. And in similar manner a common subdivision $K_2'$ of $K$ and $K_2$ can be constructed.

Consider  a covering of $S$ by closed Jordan regions $\{\overline{J}_m\}$  and $\{\overline{J}_n'\}$ which are 2-simplexes of $K$ and $K_1$ respectively. As in proof of Theorem \ref{ch2,sec3, thm7}, we denote by $\overline{J}_{mi}$ (and $\overline{ J}_{nj}')$ the closed subregions of  $\overline{J}_m$(and $\overline{ J}_n')$ obtained by subdivisions of $\overline{J}_m$(and $\overline{ J}_n')$ by the cross-cuts formed by boundaries of $\overline{J}_n'$(and $\overline{ J}_m)$. Any two regions $\overline{J}_{mi}$ and $ \overline{ J}_{nj}'$ are either identical or have disjoint interior.

Also we consider the arcs $\gamma_{mi}$ (and $\gamma_{nj}')$ into which $\gamma_m$ which are 1-simplexes of $K$ (and $\gamma_n'$ which are  1-simplexes of $K_1$) is divided by the points on $\gamma_n' \ (\gamma_n)$ or the end points of arcs that are common to both $\gamma _m$ and $\gamma_n'$.  Any two $\gamma_{mi}$ and $\gamma_{nj}'$ are either identical or have at most end points in common.

We introduce the complex $K_1'$ whose vertices are the end points of the arcs $\gamma_{mi}$ and/or  $\gamma_{nj}'$(it will include all the vertices of $K$ and $K_1$) and an interior point of each $\overline{J}_{mi}$ and/or  $\overline{J}_{nj}'$ and each $\gamma_{mi}$ and/or $\gamma_{nj}'$. The interior point of $\overline{J}_{nj}$ (and/or $\overline{J}_{nj}'$) can be joined to the vertices on the boundary of $\overline{J}_{nj}$ (and/or $\overline{J}_{nj}'$) by the Jordan arcs. Again the subdivision will provide closed triangular regions and we can define the corresponding 1-simplexes and 2-simplexes of $K'_1$ associated with the closed triangular regions.
\end{proof}

\textbf{Remark:} If S is a bordered surface we construct a surface $\widehat{S}$. For $\widehat{S}$ we construct a covering of finite character with the additional property that each $\gamma_n$ intersects the boundary $\partial S$ in a finite number of points. This can be done by including the boundary curves of $S$ in $\Gamma$. Further in the proof of Theorem \ref{ch2,sec3, thm7}, we need to subdivide each $J_m$ by the cross cuts on $\partial S$ also. For the construction of  $\widehat{S}$, we consider $S_1$ a copy of $S$ and a homeomorphism $\phi : S \to S_1$. Then $\widehat{S}$ is the space obtained by identifying each $p \in \partial S$ with its image $\phi(p) \in \partial S_1$. The space $\widehat{S}$ is called $double$ of $S$.





\chapter{A Brief Incursion in Algebraic Topology}\label{ch3}

\section{Homotopy}\label{ch3,sec1}
This chapter provides a brief introduction to algebraic topology. We begin with the notion of homotopy and fundamental group of a topological space $X$. The closed unit interval [0,1] will be denoted by $\I$.

\textbf{Definition:} Let $f$ and $g$ be two continuous maps of a topological space $X$ to a topological space $Y$, and let $H : X \times \I \to Y$ be a continuous map such that
$$H(x,0)=f(x)$$
and
$$H(x,1)=g(x)$$
for each $x \in X$.  Such a map $H$ is called a $homotopy$ between $f$ and $g$ also we say $f$ is $homotopic \ to \ g$.

For each $t \in \I$, the map $h_t : X \to Y$ given by $h_t(x)= H(x,t)$ is continuous. Now $H$ becomes a family of continuous maps $h_t$, also $h_0=f$ and $h_1=g$. Thus a homotopy is simply a family of continuous maps from $X$ to $Y$ which starts from $f$ changes continuously with respect to $t$ and terminates into the map $g$.

\begin{proposition}\label{ch3,sec1,pro1}
Homotopy of maps is an equivalence relation on the space of continuous maps from X to Y.
\end{proposition}

\begin{proof}
Let $f: X \to Y$ be a continuous map. Clearly $f$ is homotopic to itself because the map $H : X \times \I \to Y$ defined by $H(x,t)= f(x)$ is a homotopy.

If $f$ and $g$ are continuous maps from $X$ to $Y$ and $H$ is a homotopy between $f$ and $g$ then the map $G: X \times \I \to Y$ defined by $G(x,t) = H(x,1- t)$ is a homotopy between $g$ and $f$.

Also if $f$, $g$ and $h$ are continuous maps from $X$ to $Y$ such that $H'$ is a homotopy between $f$ and $g$. And $H''$ is a homotopy between $g$ and $h$ then the map $G:X \times \I \to Y$ given by
\begin{equation}\notag
G(x,t)=
\begin{cases}
H'(x,2t), & 0 \le t \le1/2\\
H''(x,2t - 1), & 1/2 \le t \le 1
\end{cases}
\end{equation}
is the homotopy between $f$ and $h$. As the two definitions agree for $t=\frac{1}{2}$ and $G$ is continuous on two closed subsets $X \times [0, \frac{1}{2}]$ and $X \times [ \frac{1}{2} , 1]$, it is continuous on all  of $X \times \I$, by the pasting lemma.

Thus the relation of being $ homotopic \ to$ is an equivalence relation.
\end{proof}

Now we consider the special case in which $f$ is a path in $X$. By a path in $X$ we mean a continuous map $f:\I \to X$. The idea of continuously deforming a path, keeping  its end points fixed, is made precise by the following

\textbf{Definition:} Two paths $f$ and $g$ from $\I$ to $X$ are said to be $path \ homotopic$ if they have the same initial point (that is $f(0)=g(0)$) and the same terminal point (that is $f(1)=g(1)$) and if there is a continuous map $H: \I \times \I \to X$ such that 
$$H(s,0)=f(s)$$
$$H(s,1)=g(s)$$
$$H(0,t)=f(0)=g(0)$$
and
$$H(1,t)=f(1)=g(1)$$
for each $s \in \I$ and each $t \in \I$. The map $H$ is called a $path \ homotopy$ between paths $f$ and $g$.

As in the Proposition \ref{ch3,sec1,pro1}, we can show that being $path \ homotopic \ to$ is an equivalence relation. If $f$ is a path, we shall denote its path homotopy equivalence class by $ [f]$.

Given two paths $f,g : \I \to X$ such that $f(1)=g(0)$ we define the product $f\ast g$ of $f$ and $g$ to be the path by
\begin{equation}\notag
f* g(t)=
\begin{cases}
f(2t), & 0 \le t \le 1/2\\
g(2t- 1), &  1/2 \le t \le 1
\end{cases}
\end{equation}
The product path $f *g$ traverses first $f$ and then $g$, it is a path in $X$ from $f(0)$ to $g(1)$.

\begin{center}
\includegraphics[width=1\columnwidth]{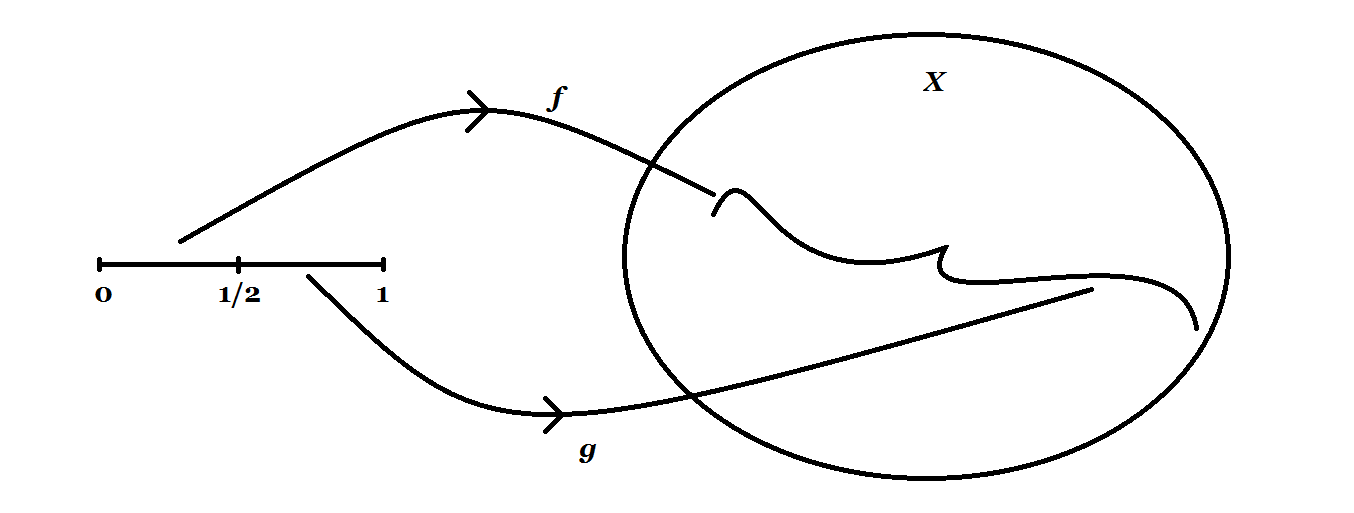}
\end{center}

\begin{center}
Figure 3.1 
\end{center}

It is well defined and continuous by the pasting lemma, which states that if a space $X$ can be written as union of two closed  subsets $A_1$ and $A_2$ of $X$ and $F:X \to Y$ is a map such that $F|_{A_1}$ and $F|_{A_2}$ are continuous then the map $F$ itself is continuous. 

\begin{proposition}
If a path $f$ is path homotopic to $f'$ and a path $g$ is path homotopic to $g'$ in X, and the product path $f*g$ exists then $f'*g'$ exists and is homotopic to $f*g$.
\end{proposition}

\begin{proof}
Since the product $f*g$ is defined, we have $f(1)=g(0)$. As $f$ is path homotopic to $f'$ we get $f(1)=f'(1)$ similarly we have $g(0)=g'(0)$. Therefore the product $f'*g'$ is defined.

Let $H'$ be a path homotopy between $f$ and $f'$ and let $H''$ be a path homotopy between $g$ and $g'$. Then the homotopy between $f*g$ and $f'*g'$ is the map $H: \I \times \I \to X $ defined as 
\begin{equation}\notag
H(s,t)=
\begin{cases}
H'(2s,t), & 0 \le t \le1/2\\
H''(2s -1,t), & 1/2 \le t \le 1 .
\end{cases}
\end{equation}

Clearly $H$  is continuous by the pasting lemma.
\end{proof}

Now we can define the product of homotopy classes of paths $f$ and $g$ as the homotopy class of $f*g$, that is $[f]*[g]=[f*g]$ provided that $f*g$ is defined.

\section{The Fundamental Group}

In this section we shall consider the path $f : \I \to X$ with the same initial and terminal point, that is, $f(0)=x_0=f(1)$. Such paths are called $loops$ and the common initial and terminal point $x_0$ is referred to as the $base\  point$. The set of all homotopy classes $[f]$ of loops $f:\I \to X$ at the base point $x_0$ is denoted by $\pi_1(X,x_0)$.

\begin{theorem}
The set $\pi_1(X,x_0)$ is a group with respect to the binary operation product of paths $[f]*[g]=[f*g]$.
\end{theorem}

\begin{proof}

By considering loops with a fixed base point $x_0 \in X$ we guarantee that the product $f*g$ of any two such loops is defined. Now we shall verify the three axioms for a group.

\textbf{Associative:} Let $[f],[g],[h] \in \pi_1(X,x_0)$. Since 
$$([f]*[g])*[h] = [(f*g)*h]$$
and
$$[f]*([g]*[h]) = [f*(g*h)],$$
we shall show that $(f*g)*h$ is path homotopic to $f*(g*h)$. By definition

\begin{equation}\notag
((f*g)*h)(s)=
\begin{cases}
f(4s), & 0 \le s \le1/4\\
g(4s-1), & 1/4 \le s \le 1/2 \\
h(2s-1), & 1/2 \le s \le 1
\end{cases}
\end{equation}
and
\begin{equation}\notag
(f*(g*h))(s)=
\begin{cases}
f(2s), & 0 \le s \le1/2\\
g(4s-2), & 1/2 \le s \le 3/4 \\
h(4s-3), & 3/4 \le s \le 1
\end{cases}
\end{equation}

Now the map $H:\I \times \I \to X$ given by

\begin{equation}\notag
H(s,t)=
\begin{cases}
f(\frac{4s}{1+t}), & 0 \le s \le(t+1)/4\\
g(4s-1-t), & (t+1)/4 \le s \le (t+2)/4 \\
h(\frac{4s-2-t}{2-t}), & (t+2)/4 \le s \le 1
\end{cases}
\end{equation}
 is the required homotopy.

\begin{center}
\includegraphics[width=1\columnwidth]{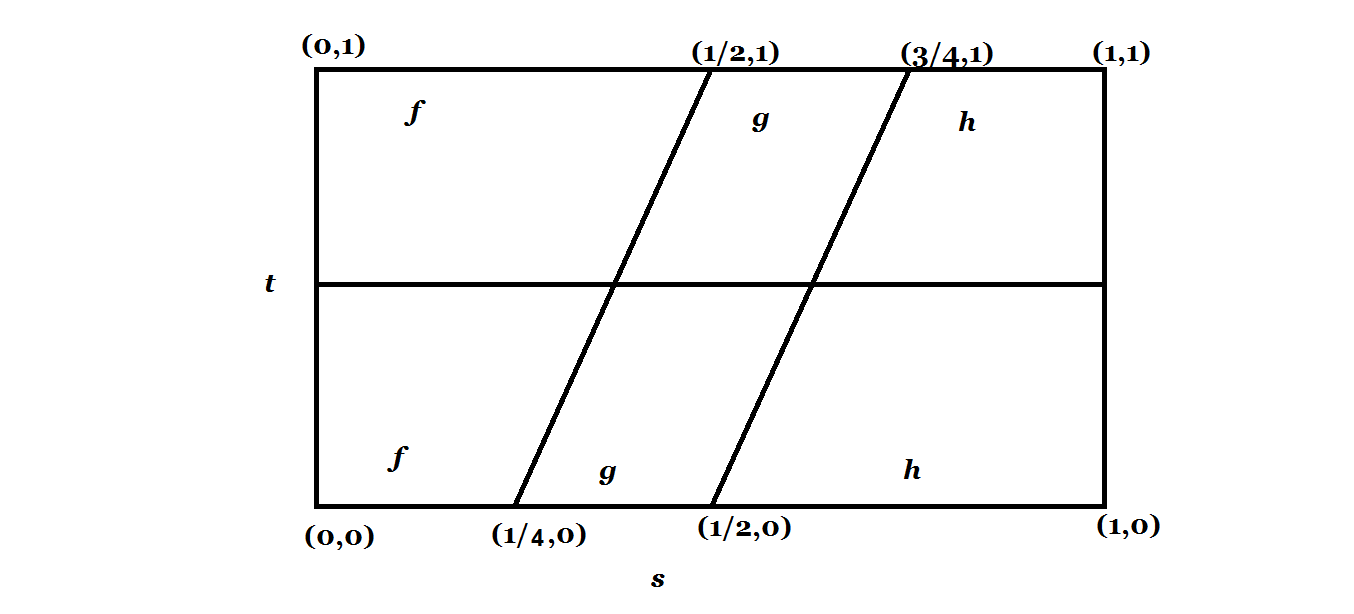}
\end{center}

\begin{center}
$\I \times \I$ 

Figure 3.2
\end{center}

\textbf{Existence of identity:} Consider the constant loop $e_{x_0} : \I \to X$. Then $[e_{x_0}] \in \pi_1(X,x_0)$ is the identity element. For if $[f] \in \pi_1(X,x_0)$ we shall show that\\
(1) $[e_{x_0}]*[f] =[f]$\\
(2) $[f]*[e_{x_0}]= [f]$

Define a map $H:\I \times \I \to X$ by 
\begin{center}
\includegraphics[width=1\columnwidth]{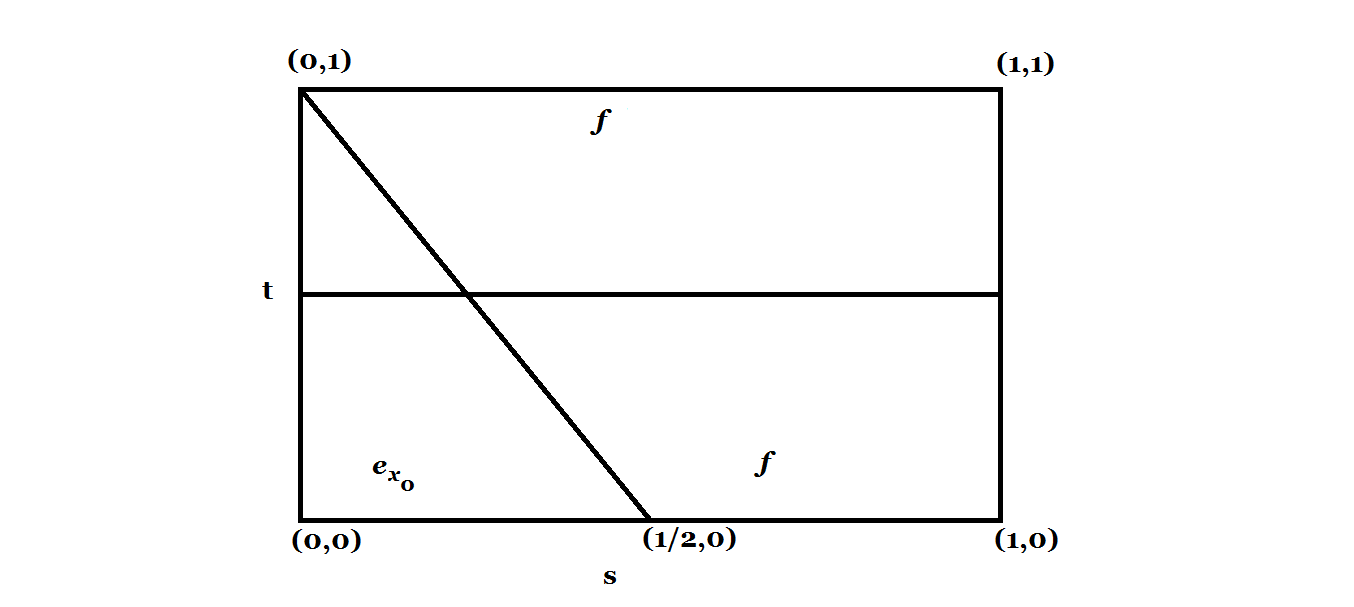}
\end{center}

\begin{center}
$\I \times \I$ 

Figure 3.3
\end{center}

\begin{equation}\notag
H(s,t)=
\begin{cases}
x_0, & 0 \le s \le (1-t)/2\\
f(\frac{2s-1+t}{1+t}), & (1-t)/2 \le s \le 1
\end{cases}
\end{equation}

Then $H$ is a path homotopy between $e_{x_0}*f$ and $f$. Similarly the map $G: \I \times \I \to X$ defined by
\begin{equation}\notag
G(s,t)=
\begin{cases}
f(\frac{2s}{1+t}), & 0 \le s \le (1+t)/2\\
x_0, & (1+t)/2 \le s \le 1
\end{cases}
\end{equation}
is the path homotopy between $f*e_{x_0}$ and $f$.
\begin{center}
\includegraphics[width=1\columnwidth]{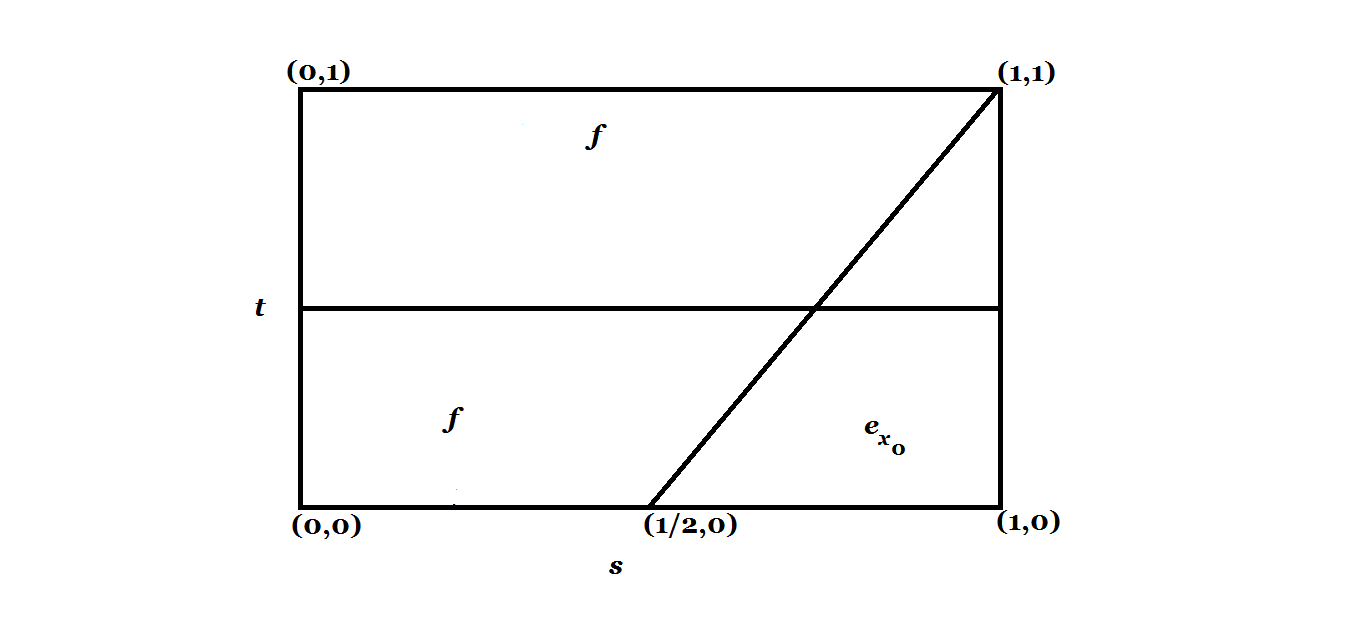}
\end{center}
\begin{center}
$\I \times \I$ 

Figure 3.4
\end{center}

\textbf{Existence of inverse:} Let $[f] \in \pi_1(X,x_0)$. The inverse of $f$ is  a loop $\bar{f}:\I \to X$ defined  by $\bar{f}(t)=f(1-t)$. We shall show that\\
(1) $[f]*[\bar{f}]=[e_{x_0}]$\\
(2) $[\bar{f}]*[f] = [e_{x_0}]$

The map $H: \I \times \I \to X$ given by

\begin{equation}\notag
H(s,t)=
\begin{cases}
f(2s), & 0 \le s \le (1-t)/2 \\
f(1-t), & (1-t)/2 \le s \le (1+t)/2\\
\bar{f}(2s-1), & (1+t)/2 \le s \le 1
\end{cases}
\end{equation}

\begin{center}
\includegraphics[width=1\columnwidth]{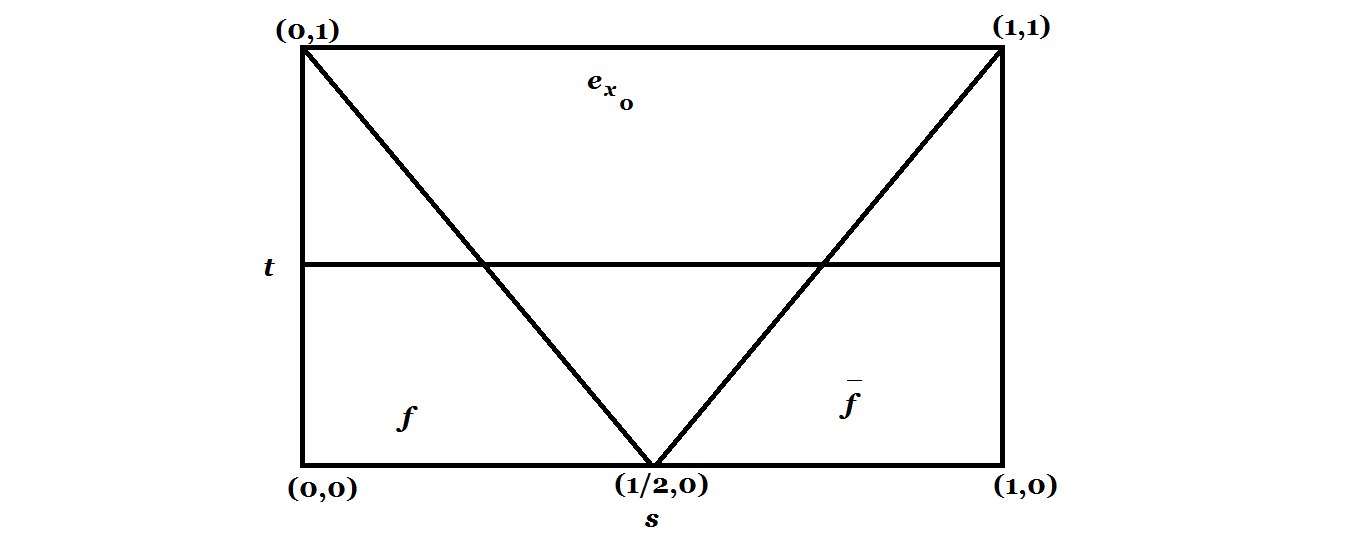}
\end{center}
\begin{center}
$\I \times \I$ 

Figure 3.5
\end{center}
 is the required path homotopy between $f*\bar{f}$ and $e_{x_0}$.

Since $\bar{\bar{f}}=f$, replacing $f$ by $\bar{f}$ gives $\bar{f}*f$ is path homotopic to $e_{x_0}$.
\end{proof}

The group $\pi_1(X,x_0)$ is called the fundamental group of space $X$ based at $x_0$. In general, the fundamental group of a space $X$ depends on the choice of the base point. However, if $X$ is path connected then for any $x_0$ and $x_1$ in $X$, the fundamental groups $\pi_1(X,x_0)$ and $\pi_1(X,x_1)$ are isomorphic. For if $X$ is path connected, there is a path $\alpha : \I \to X$ from $x_0$ to $x_1$. Then the map $\hat{\alpha} : \pi_1(X,x_0) \to \pi_1(X,x_1)$ defined by 
$$\hat{\alpha}([f])=[\bar{\alpha}]*[f]*[\alpha]$$
is a group isomorphism, where $\bar{\al}$ is the inverse path defined by $\bar{\al}(t) = \al (1-t)$.

\textbf{Definition:} A space $X$ is said to be $simply \ connected$ if it is path connected and $\pi_1(X,x_0)$ is the trivial group for some $x_0 \in X$ and hence for every $x_0 \in X$.

\textbf{Example:} The fundamental group of a convex set $X$ in $\R^n$ is trivial, that is $\pi_1(X,x_0)=0$ for $x_0\in X$. Given any two loops $f,g : \I \to X$ based at $x_0$, we can define a homotopy $H : \I \times \I \to X$ by
$$H(s,t)= (1-t)f(s)+tg(s).$$

Such a homotopy is called the $linear \ homotopy$.

Suppose $\Phi : (X,x_0) \to (Y,y_0)$ is a continuous map taking the base point $x_0 \in X$ to the base point $y_0 \in Y$. Then $\Phi$ induces a homomorphism $\Phi_{\#} : \pi_1(X,x_0) \to \pi_1(Y,y_0)$. Let $f : \I \to X$ be a loop in $X$ based at $x_0$. Then the induced homomorphism $\Phi_{\#} : \pi_1(X,x_0) \to \pi_1(Y,y_0)$ is defined by composing loop `$f$ ' based at $x_0$ with $\Phi$, that is
$$\Phi_{\#}([f])=[\Phi \circ f].$$

The map $\Phi_{\#}$ is well defined. For if $f$ is any loop in $X$ based at $x_0$, then $\Phi\circ f$ is a loop in $Y$ based at $y_0$. Also if $f$ is homotopic to $g$ by a homotopy $H$ then $(\Phi \circ f)$ is homotopic to $(\Phi \circ g)$ by the homotopy $\Phi \circ H$. 

\begin{theorem}
If $\Phi : (x,x_0) \to (Y,y_0)$ and $\Psi : (Y,y_0) \to (Z,z_0)$ are continuous then $(\Psi \circ \Phi)_{\#}=\Psi_{\#} \circ \Phi_{\#}$. If $i : (X,x_0) \to (X,x_0)$ is the identity map then $i_{\#}$ is the identity map.
\end{theorem}

\begin{proof}

Let $f: \I \to X$ be a loop in $X$ based at $x_0$. Then by definition we have
$$(\Psi \circ \Phi)_{\#}([f])=[(\Psi \circ \Phi)\circ f] $$
$$\qquad \qquad \qquad \  =[\Psi \circ (\Phi \circ f)] \ $$
$$\qquad \qquad \quad =\Psi_{\#}[\Phi \circ f]$$
$$\qquad \qquad \qquad =\Psi_{\#}[\Phi_{\#}([f])]$$
$$\qquad \quad \qquad \qquad=\Psi_{\#}\circ \Phi _{\#}([f]) \ .$$

Hence $(\Psi \circ \Phi)_{\#}= \Psi_{\#}\circ \Phi_{\#}$.

Similarly $i_{\#}([f]) = [i\circ f]=[f]$.
\end{proof}

\begin{corollary}
If $\Phi : (X,x_0) \to (Y,y_0)$ is a homeomorphism then $\Phi_{\#} : \pi_1(X,x_0) \to \pi_1(Y,y_0)$ is an isomorphism.
\end{corollary}

\begin{proof}
Let $\Psi : (Y,y_0) \to (X,x_0)$ be the inverse of $\Phi$. Then $\Psi_{\#}\circ \Phi_{\#}=(\Psi \circ \Phi)_{\#}=(i_X)_{\#}$ ($i_X$ denotes the identity map of $X$) and $\Phi_{\#}\circ \Psi_{\#}= (\Phi\circ \Psi)_{\#}=(i_Y)_{\#}$ ($i_Y$ denotes the identity map of $Y$). Since $(i_X)_{\#}$ and $(i_Y)_{\#}$ are the identity homomorphisms of  the groups $\pi_1(X,x_0)$ and $\pi_1(Y,y_0)$ respectively, we have $\Psi_{\#}$ is the inverse of $\Phi_{\#}$.
\end{proof}

Thus the fundamental group of a topological space is a topological invariant.

\textbf{Definition:} Let $A\subset X$, a $retraction$ of $X$ onto $A$ is a continuous map $r: X\to A$ such that $r|_A$ is the identity map of $A$. If such a map exists, we say that $A$ is a $retract$ of  $X$. Moreover, $A$ is a \emph{deformation  retract} of $X$ if $A$ is a retract of $X$ and there is a homotopy between $i_X$ and a continuous map that carries $X$ into $A$ such that points of $A$  remains fixed during the homotopy. That is , there is a continuous map $ H: X \times \I \to X $ such that
$$H(x,0)=x$$
$$H(x,1)\in A$$
for each $x\in X$ and
$$H(a,t)=a$$
for each $a\in A$.

\begin{theorem}
Let A be a deformation retract of X and $x_0 \in A$. Then the inclusion map $j:(A,x_0) \hookrightarrow (X,x_0)$ induces an isomorphism of fundamental groups.
\end{theorem}

\begin{proof}
Let $r: X \to A$ be a retraction such that $i_X$ is homotopic to $j \circ r$ via the homotopy $H$. If $f$ is a loop in $X$ based at $x_0$ then the map $G: \I \times \I\to X $ defined by
$$G(x,t)=H(f(x),t)$$
is a homotopy between $f$ and $(j \circ r)\circ f$. Hence $(j \circ r)_{\#}=(i_x)_{\#}$ that is $j_{\#}\circ r_{\#}=(i_X)_{\#}$. Also $r\circ j = i_A$ gives $r_{\#} \circ j_{\#}=(i_A)_{\#}$. Therefore $j_{\#}: \pi_1(A,x_0) \to \pi_1(X,x_0)$ is an isomorphism with $r_{\#}$ as its inverse.
\end{proof}

\section{The Fundamental Group of Spheres}

Consider the unit circle $\s^1=\{(x,y)\ |\ x^2+y^2=1\}$ in $\R^2$ or $\s^1=\{z\ |\ |z|=1\}$ in $\C$. Let $p: \R \to \s^1$ denotes the $exponential \ map$  defined by $p(t)= e^{i2\pi_1 t}$. It is continuous and surjective map which simply wraps the real line onto $\s^1$ infinite number of times.

\textbf{Definition:} Given a continuous map $f:X\to \s^1$, a \emph{lifting}  of $f$ is a continuous map $\tilde{f}:X \to \R$ such that $p\circ \tilde{f}=f$, that is the following diagram commute
$$\qquad \ \R$$
$$\ \ \tilde{f} \  \nearrow  \ \big\downarrow    \   p$$
$$X \longrightarrow \s^1$$
$$f$$

To carry out lifting process we need to examine the properties of the exponential map $p: \R \to \s^1$. Let $U$ be the open set in $\s^1$ given by $U= \s^1\setminus \{-1\}$ and consider the inverse image $p^{-1}(U)$ of $U$ in $\R$. This is precisely the union of all open intervals of the form $(n-1/2,n+1/2), \ n\in \Z$. These intervals are pairwise disjoint and the restriction of $p$ to any one of  them is a homeomorphism of the interval with $U$. Similarly if $V=\s^1\setminus \{1\}$, the inverse image of $V$ breaks up a disjoint union of open intervals $(-n,n),n\in \Z$ in such a way that the restriction of $p$ to any one of the open interval is a homeomorphism. Also $U\cup V=\s^1$.

\begin{theorem}
Let X be a convex compact subspace of $\R^n$, and 0 $\in$ X. If $f:X\to \s^1 $ is a continuous map such that $f(0)=1$, then for each $m \in \Z$, there exist a unique continuous map $\tilde{f}:X \to \R$ such that $p\tilde{f} =f$ and $\tilde{f}(0)=m$.
\end{theorem}

\begin{proof}
Since $X$ is compact, $f$ is uniformly continuous. So there is a $\delta > 0$ such that $|f(x)-f(x')|<2$ whenever $||x-x'||< \delta$. Since $X$ is bounded we can find a positive integer $n$ such that $||x||<n\delta$ for all $x \in X$ . Then for every $x\in X$  and for each $i = 0,1, \ldots , n-1$, we have
$$\big\Vert\frac{i+1}{n}x-\frac{i}{n}x\big \Vert < \delta.$$

So,
$$\big |f(\frac{i+1}{n}x)-f(\frac{i}{n}x) \big | <2$$
implies they are not the end points of a diameter and we have $f(\frac{i+1}{n}x)/f(\frac{i}{n}x) \ne  - 1$  for all $x\in X$ and for each $i = 0,1, \ldots , n-1$. For 0 $\le i< n $ define $g_i: X \to \s^1\setminus \{-1\}$ by $g_i(x)=f(\frac{i+1}{n}x)/f(\frac{i}{n}x)$. Then each $g_i$ is continuous, $g_i(0)=1$ and $f = g_0g_1\ldots g_{n-1}$. Since the mapping $q=p|_{(-1/2,1/2)}$ is a homeomorphism between $(-1/2,1/2)$ and $\s^1\setminus\{-1\}$,  $q^{-1}: \s^1\setminus \{-1\} \to (-1/2,1/2)$  is continuous. Therefore the mapping $\tilde{f}:X \to \R$ given by $\tilde{f}(x)=m+q^{-1}(g_0(x))+\ldots +q^{-1}(g_{n-1}(x))$ is continuous. Also $p\tilde{f}(x)=f(x)$ for every $x\in$ X and $\tilde{f}(0) = m$.

For the uniqueness of $\tilde{f}$, suppose that there is a continuous map $\tilde{g}: X \to \R$ such that $p\tilde{g}=f$ and $\tilde{g}(0)=m$. Then the mapping $\tilde{h}: X \to \R$ defined by $\tilde{h}=\tilde{g} - \tilde{f}$ is continuous and satisfies $p\tilde{h}(x)=p(\tilde{g}(x) -\tilde{f}(x))=1$ for all $x\in$ X and $\tilde{h}(0)=0$. This means that $\tilde{h}(x) \in p^{-1}(1)= \Z$ for every $x\in X$ . Since $X$ is connected, and $\tilde{h}$ is  continuous this implies that $\tilde{h}(x)=0$ for all $x\in X$. Hence $\tilde{f}=\tilde{g}$.
\end{proof}

\begin{corollary}\label{ch3,sec3,cor2}
 If $f: \I \to \s^1$ is a path with $f(0) =1$ , then there exists a unique path $\tilde{f}: \I \to \R$ such that $p\tilde{f} = f$ and $\tilde{f}(0)=0$.
\end{corollary}

\begin{corollary}\label{ch3,sec3,cor3}
 If $F: \I\times \I \to \s^1$ is a homotopy with $F(0,0)=1$, then there is a unique homotopy $\tilde{F}: \I \times \I \to \R$ such that $p\tilde{F}=F$ and $\tilde{F}(0,0)=0$.
\end{corollary}

\begin{theorem}
Let $f,g: \I \to \s^1$ are two loops based at 1 such that $f$ is path homotopic to $g$ and  let $\tilde{f}$ and $\tilde{g}$ be their respective lifting in $\R$ beginning at 0. Then $\tilde{f}$ and $\tilde{g}$ end at the same point.
\end{theorem}

\begin{proof}
Let $F$ be a path homotopy between $f$ and $g$. By the Corollary \ref{ch3,sec3,cor3}, there is a homotopy $\tilde{F}: \I \times \I \to \R$ such that $\tilde{F}(0,0)=0$ and $p\tilde{F}=F$. We have $p\tilde{F}(0,t)=1$ for every  $t\in \I $, this means that $\tilde{F}(0,t)$ is an integer for each $t$. Since $\{0\} \times \I$ is connected and $\tilde{F}(0,0)=0$, we must have $\tilde{F}(0,t)=0$ for every $t\in \I $. Similarly $\tilde{F}(1,t)$ is also a fixed integer $m$ (say).

The restriction $\tilde{F}|_{\I \times \{0\}}$ of $\tilde{F}$ is a path beginning at 0, that is a lifting of $F|_{\I\times \{0\}}$. By uniqueness of path liftings we must have $\tilde{F}(s,0)=\tilde{f}(s)$. Similarly $\tilde{F}|_{\I \times \{1\}}$ of $\tilde{F}$ is a path beginning at 0, that is a lifting of $F|_{\I\times \{1\}}$. By uniqueness of path liftings we get $\tilde{F}(s,1)=\tilde{g}(s)$. Therefore both $\tilde{f}$ and $\tilde{g}$ end at  $m$.
\end{proof}

\textbf{Definition:} Let $f:\I \to \s^1$ be a loop in $\s^1$ at 1 and $\tilde{f} :\I \to R$ be the lifting of $f$ such that $\tilde{f}(0)=0$. The \emph{degree of  f} denoted by \emph{deg(f)} is the integer $\tilde{f}(1)$.

\begin{theorem}
The fundamental group of  $\s^1$ is isomorphic to the additive group of integers.
\end{theorem}

\begin{proof}
We show that the degree map $deg : \pi_1 (\s^1,1)\to \Z$ defined by $deg([f])=$ \emph{deg(f)}  is an isomorphism. Given $[f]$ and $[g]$ in $\pi_1 (\s^1,1)$, let $\tilde{f}$ and $\tilde{g}$ be their respective lifting  in $\R$ beginning at origin. Define a path $\tilde{h} : \I \to \R $  by
\begin{equation}\notag
\tilde{h}(t)=
\begin{cases}
\tilde{f}(2t), & 0 \le t \le 1/2\\
\tilde{f}(1) + \tilde{g}(2t - 1), &  1/2 \le t \le 1 .
\end{cases}
\end{equation}

Then $\tilde{h}$ is a lifting of $f*g$ and $\tilde{h}(0)=0$. And by definition $deg([f]*[g]) = deg([f*g]) = \tilde{h}(1) =\tilde{f}(1) + \tilde{g}(1) = deg([f]) + deg([g])$.

Now to show that $deg$ is one to one, suppose that $deg([f])=deg([g])$ for $[f]$ and $[g]$ in $\pi_1(\s^1,1)$. If $\tilde{f}$ and $\tilde{g}$ be respective liftings of $f$ and $g$ beginning at origin then $\tilde{f}(1)=\tilde{g}(1)$. Consider the linear homotopy $H : \I\times \I \to \R$ defined by
$$H(s,t)=(1-t)\tilde{f}(s) +t\tilde{g}(s).$$

Then $\tilde{f}$ is homotopic to $\tilde{g}$ . It follows that $p\circ H$ is the homotopy between $f$ and $g$, that is $[f]=[g]$.

To show that $deg$ is onto, let $n$ be any integer and consider the path $\tilde{f}:\I \to \R$ defined by $\tilde{f}(t)=nt$. Then $f=p\tilde{f}$ is a loop in $\s^1$ based at 1 with $deg([f])= \tilde{f}(1)=n$. Thus $deg$ is onto.
\end{proof}

Now we state the famous theorem of algebraic topology called the Seifert-Van Kampen Theorem. It expresses the fundamental group of the space $X$, which is decomposed as the union of a collection of path connected open subsets $A_{\alpha}$ (each of which contains the base point $x_0\in X$), in terms of the free product of the fundamental groups of $A_{\alpha}$, $*_{\alpha}\pi_1(A_{\alpha},x_0)$. Consider homomorphisms $j_{\alpha}: \pi_1(A_{\alpha},x_0)\to \pi_1(X,x_0)$ induced by the inclusions from $A_{\alpha}$ to $X$ and the homomorphisms $i_{\alpha \beta} : \pi_1(A_{\alpha} \cap A_{\beta},x_0) \to \pi_1 (A_{\alpha},x_0)$ induced by the inclusions from $A_{\alpha}\cap A_{\beta}$ to $A_{\alpha}$.

\begin{theorem}
$($see ~\cite{hatcher}.$)$  If X is the union of path connected open sets $A_{\alpha}$ each containing the base point $x_0\in X$ and if each intersection $A_{\alpha} \cap A_{\beta}$ is path connected, then the homomorphism $\Phi : *_{\alpha}\pi_1(A_{\alpha},x_0) \to \pi_1(X,x_0)$ is surjective. If in addition each intersection $A_{\alpha}\cap A_{\beta} \cap A_{\gamma}$ is path connected then the kernel of $\Phi$ is the normal subgroup N generated by all elements of the form $i_{\alpha \beta}([f])i_{\beta \alpha}([f])^{-1}$, and so $\Phi$ induces an isomorphism between $\pi_1(X,x_0)$ and $*_{\alpha}\pi_1(A_{\alpha},x_0)/$N.
\end{theorem}

\textbf{Examples:}
\\
(1) The fundamental group of $\s^n$, for $n\ge$ 2 is trivial. Let $p=(0,0,\ldots ,0,1) \in \R^{n+1}$ and $q=(0,0,\ldots ,0,-1) \in \R^{n+1}$ be the north and south pole of $\s^n$ respectively. Take $A_1=\s^n \setminus \{p\}$ and 
$A_2=\s^n \setminus \{q\}$. Each of $A_1$ and $A_2$ are homeomorphic to $\R^n$ and so is simply connected. Hence the fundamental groups of $A_1$ and $A_2$ are trivial. Also $A_1 \cap A_2$ is path connected. Hence it follows from the above theorem that the fundamental group of $\s^n$ is trivial for $n\ge$ 2.\\
(2) For $x\in \R^{n+1}$, the space  $\R^{n+1}\setminus \{x\}$ is  homeomorphic to $X=\R^{n+1}\setminus \{0\}$. Infact $X$ is a deformation retract of $\s^n$.  The homotopy $H:X\times \I \to X$ defined by
$$H(x,t) = (1-\ t)x +t\frac{x}{\Vert x \Vert}$$
is a deformation retraction. Hence the fundamental group of $\R^{n+1}\setminus \{x\}$ is $\Z$ for $n=1$ and is trivial for $n > 1$.\\
(3) Consider  the wedge sum $\bigvee_{\alpha}X_{\alpha}$ of a collection of path connected spaces $X_{\alpha}$ with base points $x_{\alpha}\in X_{\alpha}$ to be the quotient space of the disjoint union $\bigsqcup_{\alpha}X_{\alpha}$ in which all the base points $x_{\alpha}$ are identified to a single point. If each $x_{\alpha}$ is a deformation retract of an open neighborhood $U_{\alpha}$ in $X_{\alpha}$ containing $x_{\alpha}$, then $X_{\alpha}$ is a deformation retract of its open neighborhood $A_{\alpha}=X_{\alpha}\bigvee_{\beta \ne \alpha}U_{\beta}$. The intersection of two or more distinct $A_{\alpha}'$s is $\bigvee_{\alpha}U_{\alpha}$, which deformation retracts to a point. Then by above theorem the fundamental group of $\bigvee_{\alpha}X_{\alpha}$ is isomorphic to free product of the fundamental groups of $X_{\alpha}$.


\section{Simplicial Homology}

For a finite simplicial complex $K$, its $simplicial \ homology \ group\ H_*(K)$ is defined as follows. We consider an ordering of the vertices $v_0,v_1,\ldots ,v_n$ of each $n$-simplex $\sigma^n = \ <v_0,v_1,\ldots ,v_n>$. Two orderings are said to be equivalent if they can be transformed into each other by even permutations. An equivalence class of ordering of vertices is called an $orientation$ of the simplex. If $n \ge$ 1, there are exactly two orientation in each $n$-simplex. When we consider the simplex $\sigma^n$ together with the equivalence class of all even permutations of its vertices, we say that $\sigma^n$ is $positively \ oriented$ and write the pair as $+\sigma^n$. On the other hand, when $\sigma^n$ is $negatively \ oriented$  we write this as $-\sigma^n$.

\begin{center}
\includegraphics[width=1\columnwidth]{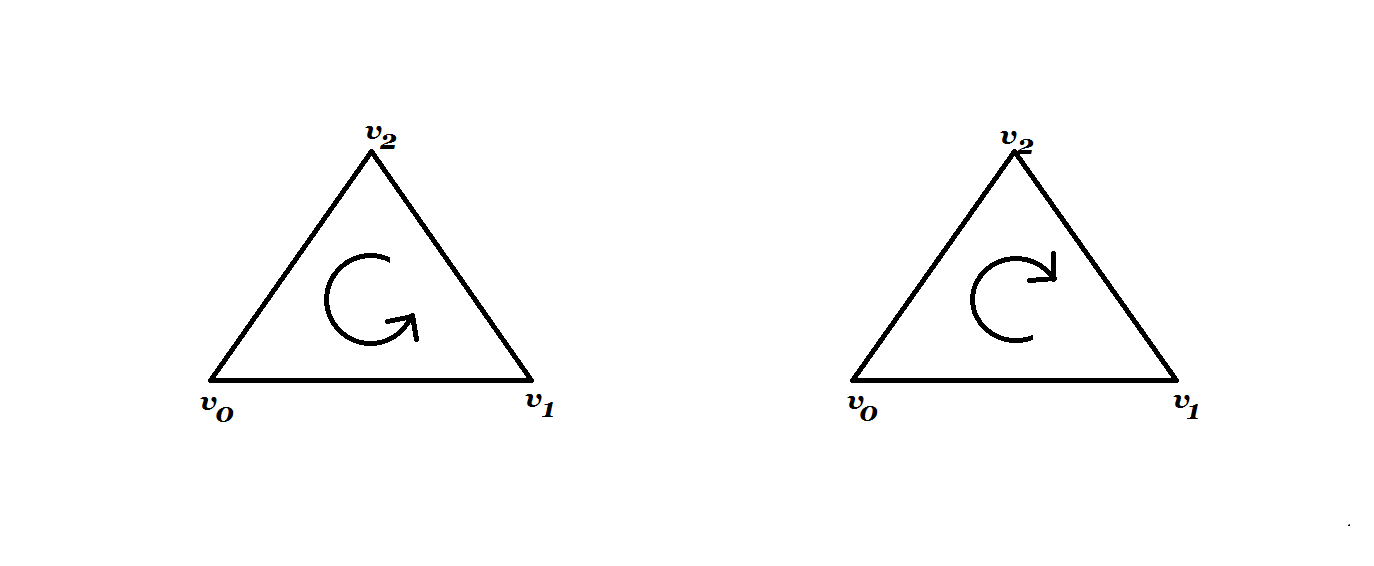}
\end{center}

\begin{center}
Figure 3.6 
\end{center}

Consider the 2-simplex $\sigma^2= <v_0,v_1,v_2>$. If we order the vertices as $v_0,v_1,v_2$, then $+\sigma^2=<v_0,v_1,v_2>=<v_1,v_2,v_0>=<v_2,v_0,v_1>$ and $-\sigma^2=<v_1,v_0,v_2>= <v_0,v_2,v_1>\ = <v_2,v_1,v_0>$. In fact, orienting a simplex means fixing the positive direction of its vertices, and then the negative direction is automatically fixed.

\textbf{Definition:} A simplex with a specified orientation is called an oriented simplex. And a simplicial complex $K$ is said to be oriented if each of its simplexes is assigned an orientation.

Note that if an $n$-simplex $\sigma^n=\ <v_0,v_1,\ldots ,v_n>$ is oriented by the ordering $v_0,v_1,\ldots ,$ $v_n$ then all of its faces are automatically oriented by the ordering induced by the above ordering. A 0-simplex $<v_0>$ is always taken to be positively oriented.

We now assign an orientation to each $n$-simplex $\sigma^n_{\alpha}$ of a simplicial complex $K$. We denote the free abelian group generated by each positively oriented $n$-simplex $\sigma^n_{\alpha}$ by $C_n(K)$. Elements of $C_n(K)$, called \emph{n-chains}, can be written as finite sums  $c=\sum_{\alpha}n_{\alpha}\sigma^n_{\alpha}$ with coefficients $n_{\alpha} \in \Z$. A homomorphism $\partial_n:C_n(K)\to C_{n-1}(K)$ called $the \ boundary \ operator$ is defined as follows: on generator $\sigma^n=\ <v_0,v_1,\ldots ,v_n>$ of $C_n(K)$, we define
$$\partial_n(\sigma^n)=\sum_{i=1}^{n}(-1)^i<v_0,\ldots ,\hat{v_i},\ldots ,v_n>.$$
\begin{center}
\includegraphics[width=1\columnwidth]{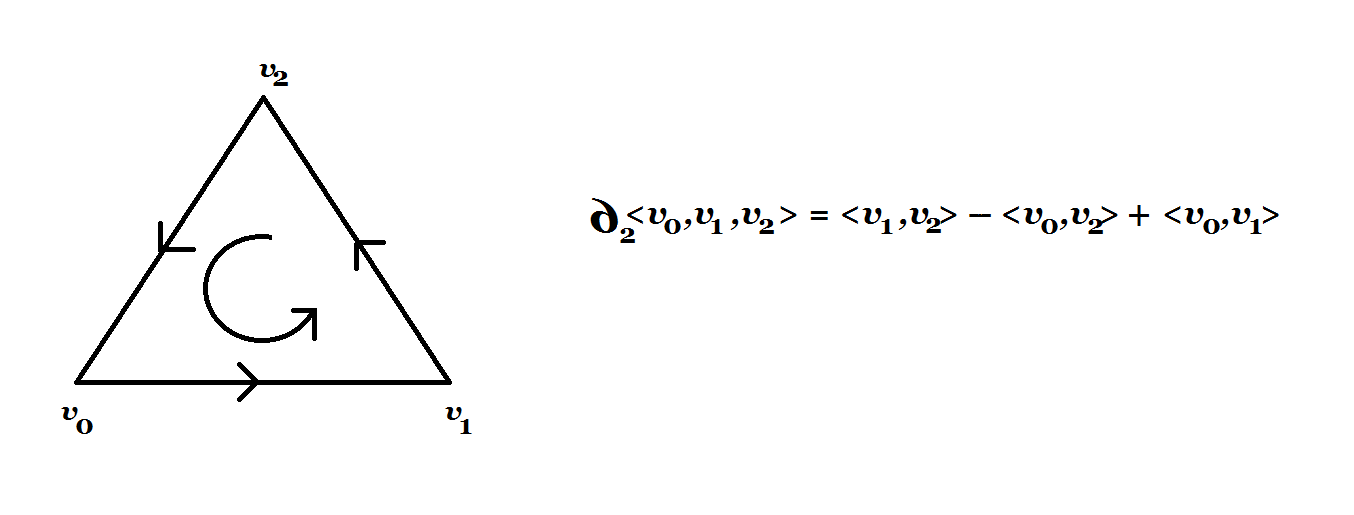}
\end{center}

\begin{center}
Figure 3.7
\end{center}

Here, $\hat{v_i}$ means that $v_i$ is omitted. Then we linearly extend it over $C_n(K)$ , that is,
$$\partial_n(\sum_{\alpha}n_{\alpha}\sigma^n_{\alpha})=\sum_{\alpha}n_{\alpha}\partial_n(\sigma^n_{\alpha}).$$

The important thing here is that if we apply the boundary operator twice, the value is always a 0, that is, $\partial_{n-1} \circ \partial_n =0$. By virtue of this fact, if we let $Z_n(K)=\{ c \in C_n(K) : \partial_n ( c) =0\}$ and $B_n(K)=\{ \partial_{n+1} (c) : c \in C_{n+1}(K)\}$ then $B_n(K)\subset Z_n(K)$. We denote the quotient group $Z_n(K)/B_n(K)$ by $H_n(K)$ and call it the \emph{n-dimensional  homology  group} of $K$. The homology group $H_*(K)$ refers to the integral homology group.  Any element of $Z_n(K)$ is called an \emph{n-dimensional  cycle}, and any element of $B_n(K)$ is called  a $boundary \ cycle$ of $K$. The homology class represented by a cycle $z\in Z_n(K)$ is usually denoted by $[z] \in H_n(K)$. Also any two cycles $z,z' \in Z_n(K)$ are called $homologous$ if they represent the same homology class, that is, there exists a chain $c\in C_{n+1}(K)$ such that $z'- z=\partial(c)$.

Observe that for  $n<$ 0 or  $n>$ dim$K$, all $C_n(K)$ are evidently zero groups and therefore $H_n(K)=0$ for all such $n$. Possibly the  non trivial homology group $H_n(K)$ of $K$ can occur only when 0 $\le n \le$ dim $K$. Moreover, for $m=$ dim$K$,  $B_m(K)=0$ and $Z_0(K)=C_0(K)$.

Note that an $n$-simplex $\sigma^n$ has exactly two orientations  and   a simplicial complex $K$ is oriented by assigning an arbitrary orientation to each of the simplexes of $K$. Therefore, simplicial complex can be oriented in several different ways. Suppose $K_1$ and $K_2$ denote the same simplicial complex $K$ equipped with different orientations. Then $H_n(K_1)=H_n(K_2)$ for each  $n \ge$ 0 (see~\cite{deo}).

\section{Singular Homology}

For any integer $n \ge 0$, let $\tri n$ denote the Euclidean simplex $<e_0,e_1,\ldots ,e_n>$, where $e_0=0$ and $e_i=(0,\ldots ,1,\ldots ,0)$, for $1\le i \le n$, is the vector with a `1' at the $ith$ place and zero elsewhere. This is called the standard $n$-simplex. Let $X$ be a topological space. A continuous map $\sg : \tri n \to X$ is called a \emph{singular n-simplex} in $X$. The word `singular' is used here to express the idea that $\sg$ need not be a nice embedding, so its image does not look at all like a simplex.

\begin{center}
\includegraphics[width=1\columnwidth]{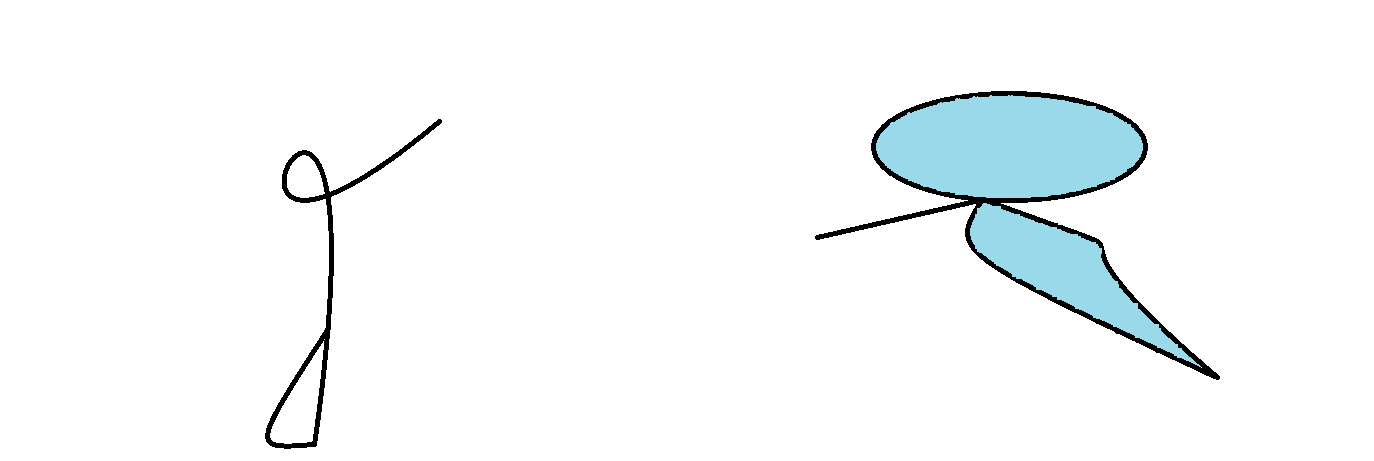}
\end{center}

\begin{center}
Figure 3.8
\end{center}

Note that since $\tri n$ is compact connected space, the image set $\sg(\tri n)$ in $X$ must be compact and connected. Thus a singular 0-simplex in $X$ is just a point in $X$ whereas a singular 1-simplex in $X$ is simply a path in $X$. A singular 2-simplex in $X$ will be a curved triangle or region with its interior.

Let $C_n(X)$ be the free abelian group generated by the set of all singular $n$-simplexes in $X$. Elements of $C_n(X)$, called the \emph{singular n-chains} can be written as finite sums $c=\sum_{\alpha}n_{\alpha}\sigma^n_{\alpha}$ with coefficients $n_{\alpha} \in \Z$ and $\sg_{\alpha}^n: \tri n \to X$.

Now we define the  boundary homomorphism from $C_n(X)$ to $C_{n-1}(X)$ for all $n\ge 1$. For each $0\le i \le n$, let $F_n ^i : \tri{n-1} \to \tri n$ be the map such that 
\begin{equation}\notag
F_n ^i(e_j)=
\begin{cases}
e_j, & j<i\\
e_{j+1}, & j\ge i
\end{cases}
\end{equation}
that is the map $F_n ^i$ map the vertices $(e_0,\ldots ,e_{n-1})$ to the set of points $(e_0,\ldots ,\hat{e}_i,\ldots ,e_n)$, where $\hat{e_i}$ means the point $e_i$ is omitted. The map $F_n ^i$ is called \emph{ith face map} of $n$-simplex $\tri n$ and it map $\tri {n-1}$ to the boundary face of $\tri n$ opposite to the vertex $e_i$. First we note that the face maps $F_n ^i ,\ F_{n-1 } ^j , \ 0 \le j< i \le n$ satisfy the commutative relation
\begin{equation}\label{ch3,sec5,eq1}
F_n ^i \circ F_{n-1} ^j = F_n ^j \circ F_{n-1} ^{i-1}.
\end{equation}

For any singular $n$-simplex $\sg : \tri n \to X$, the boundary $\partial_n(\sg)$ of $\sg$ is given by 
$$\partial_n(\sg)=\sum_{i=0}^{n} (-1)^i \sg \circ F_n ^i.$$
Since $C_n(X)$ is free abelian group generated by the set of all singular $n$-simplexes in $X$, we can extend $\partial_n$ linearly to a homomorphism $\partial_n : C_n(X) \to C_{n-1}(X)$,  that is for any singular chain $c=\sum_{\alpha}n_{\alpha}\sigma^n_{\alpha}$ we define 
$$\partial_n(c) = \sum_{\alpha}n_{\alpha}\partial_n(\sg^n_{\alpha}).$$
This is called $the \ boundary \ operator$. The boundary of any 0-chain is defined to be zero. By using the commutative relation \ref{ch3,sec5,eq1} of face maps here also we can see that $\partial_{n-1}\circ \partial_n=0$ for all $n \ge 1$.

As in previous section, we let $Z_n(X)=\{ c \in C_n(X) : \partial_n ( c) =0\}$ and $B_n(X)=\{ \partial_{n+1} (c) : c \in C_{n+1}(X)\}$, then $B_n(X)\subset Z_n(X)$. Again element of $Z_n(X)$ is called an $n$ cycle and an element of $B_n(X)$ is called a boundary. The $n$-dimensional singular homology group of $X$ is defined to be the quotient group $H_n(X)=Z_n(X)/B_n(X)$. Also if two  $n$-cycles determine the same homology class, they are said to be homologous.

Next we compute the zero dimensional homology group of a space $X$.
\begin{theorem}
If X is path connected, then $H_0(X)$ is isomorphic to $\Z$.
\end{theorem}

\begin{proof}

Since the boundary operator is the zero map in dimension 0, every 0-chain is a cycle and we have $Z_0(X)=C_0(X)$.  Also a singular 0-chain is a linear combination of points in $X$ with integer coefficients : $c=\sum_{\alpha}n_{\alpha}x_{\alpha}$. Define a map $f: C_0(X) \to \Z$ by
$$f(\sum_{\alpha}n_{\alpha}x_{\alpha})=\sum_{\alpha}n_{\alpha}.$$
Clearly $f$ is a surjective homomorphism.

Next we show that $Ker f=B_0(X)$, that gives $H_0(X)=C_0(X)/B_0(X)$ is isomorphic to $\Z$. Now if $\sg$ is a singular 1-simplex we have, by definition, $\partial(\sg) = \sg(1)-\sg(0)$, so that $f(\partial(\sg))=1-1=0$. Therefore $B_0(X) \subset Ker f$.

Also for $c=\sum_{\alpha}n_{\alpha}x_{\alpha}$ in $Ker f$, that is $f(\sum_{\alpha}n_{\alpha}x_{\alpha})=\sum_{\alpha}n_{\alpha}=0$, take any point $x_0 \in X$ and for each $x\in X$, let $\alpha_x : \I \to X$ be a path from $x_0$ to $x$. This is a singular 1-simplex and $\partial(\alpha_x)=\alpha_x(1)-\alpha_x(0) = x- x_0$. Thus
$$\partial(\sum_{\alpha}n_{\alpha}\alpha_{x_{\alpha}})= \sum n_{\alpha}x_{\alpha}- \sum n_{\alpha} x_0$$
$$\qquad \quad=c-\ x_0 \sum n_{\alpha}$$
$$=c\qquad$$
then $c\in B_0(X)$, which shows that $Ker f\subset B_0(X)$.
\end{proof}

Infact one can show in general, that for any topological space $X$, $H_0(X)$ is the direct sum of $\Z'$s  one for each path component of $X$.

The significance of the homology groups derives from the fact that they are topological invariant. The proof is very easy consequence of the fact that continuous maps induces homology homomorphisms.

Let $f: X\to Y$ be a continuous map. If $\sg : \tri n \to X$ is a singular $n$-simplex in $X$ then clearly $f \circ \sg : \tri n \to Y$ is a singular $n$-simplex in $Y$. Thus, for every $n\ge 0$, $f$ induces a homomorphism $f_{\#}: C_n(X) \to C_n(Y)$ defined by composing each singular $n$-simplex $\sg$ in $X$ with $f$ to get a singular $n$-simplex in $Y$, then extending $f_{\#}$ linearly to $C_n(X)$, that is
$$f_{\#}(\sum_{\alpha}n_{\alpha}\sg_{\alpha})=\sum_{\alpha}n_{\alpha}f\circ \sg_{\alpha}$$

The key fact is that $f_{\#}$ commutes with boundary operator $\partial$, since
$$f_{\#}(\partial(\sg))=f_{\#}(\sum(-1)^i \sg \circ F_n^i)$$
$$\qquad \qquad \ \ =\sum(-1)^if\circ \sg \circ F_n^i \ \ $$
$$=\partial(f\circ \sg)\ \ $$
$$=\partial(f_{\#}(\sg)) .$$

This implies that $f_{\#}$ maps $Z_n(X)$ to $Z_n(Y)$ as $\partial c=0$ implies $\partial(f_{\#}(c))=f_{\#}(\partial c)=0$. Also $f_{\#}$ maps $B_n(X)$ to $B_n(Y)$ since $f_{\#}(\partial c)=\partial (f_{\#} c)$. Hence $f_{\#}$ induces a homomorphism $f_{*} : H_n(X) \to H_n(Y)$, for every $n\ge 0$, defined by 
$$f_{*}[z]=[f_{\#}z].$$
where $z\in C_n(X)$ is a cycle. Two important properties of induced homomorphism are: \\
(1) If $i_X:X\to X$ is the identity map, then the induced homomorphism $(i_X)_* :  H_n(X)\to H_n(X)$  is also identity map, for  every  $n \ge 0$.\\
(2) If $f : X\to Y$and  $ g: Y \to Z$ are continuous maps then, for every $n\ge 0$, the induced homomorphisms satisfy
$$(g\circ f)_* = g_*\circ f_* .$$

\begin{theorem}
If $f: X \to Y$ is a homeomorphism then for every $ n \ge 0$,  $f_* : H_n(X) \to H_n(Y)$ is an isomorphism.
\end{theorem}

\begin{proof}
Let $f^{-1}$ denote the inverse of the homeomorphism $f$. Then the induced homomorphism $f_*^{-1}: H_n(Y) \to H_n(X)$ satisfies $f_* \circ f_* ^{-1}=(f \circ f^{-1})_*= (i_Y)_*$ and $f_*^{-1}\circ f_*=(f^{-1}\circ f)_*=(i_X)_*$. This implies $f_*^{-1}$ is the inverse of $f_*$ and thus $f_*$ is  an isomorphism.
\end{proof}

\section{Homology and Fundamental Group}

Now we show that there is a nice relationship between the first homology group of a path connected space and its fundamental group. The former is just the abelianization of the latter.  This will enable us to compute the first homology group of all the path connected spaces whose fundamental groups are known. 

Given a group $G$, the $commutator \ subgroup$ of $G$, denoted by $G'$, is the subgroup generated by elements of the form $\alpha \beta \alpha ^{\emph{-}1}\beta ^{\emph{-}1}$ for $\alpha ,\beta \in G$. The quotient group $G/G'$ is always abelian, and the commutator subgroup is trivial if and only if $G$ itself is abelian. This quotient group is denoted by $Ab(G)$ and called the $abelianization$ of $G$. It is clear that isomorphic groups have isomorphic abelianization.

Let $X$ be a topological space and $x_0\in X$. A map $f: \I \to X$ can be viewed as either a path or a singular 1-simplex. If $f$ is a loop, then this singular 1-simplex is a cycle since $\partial f= f(1)-f(0)=0$. Also if $f$ is a constant path at $x_0$ then $f$ is not only a cycle being a loop but it is the boundary of a constant singular 2-simplex. Consider a singular 2-simplex $\sg : \tri 2 \to X$ which maps  whole of $\tri 2$ to $x_0$. Then by definition $\partial (\sg)=f -f+f=f.$

\begin{lemma}\label{ch3,sec6,lm1}
Suppose f , $g: \I \to X$ are paths in X from a to b such that f is path homotopic to g. Then the singular 1-chain $f - g$  is a boundary.
\end{lemma}

\begin{proof}
Let $H: \I \times \I \to X$ be a path homotopy between $f$ and $g$. Since $H$ is constant on $\{0\} \times \I$,  let $q: \I \times \I \to \tri 2$ be  the continuous  map  defined as
$$q(x,y)=(x-x y , x y)$$
which maps $\{0\} \times \I$ to the vertex $e_0$ and each horizontal line segment linearly to a radial line segment. Then $q$ is a quotient map and there is a unique continuous map $\sg : \tri 2 \to X$ such that $H=\sg \circ q$ :
$$\ \I \times \I\qquad \quad \ \ $$
$$\ \ q\  \big\downarrow \quad \searrow H\qquad $$
$$\tri 2 \ \longrightarrow\  X \ $$
$$\sg$$

By the definition of boundary operator we have 
$$\partial (\sg)= e_b - g +f$$
 where $e_b$ is the constant path at $b$. Since constant path $e_b$ is already a boundary, it follows that $f -  g $ is a boundary.
\end{proof}

Since a loop in $X$ based at any point $x_0$ corresponds to a cycle and Lemma \ref{ch3,sec6,lm1}  shows that the homology class depends only on the path homotopy class of  $f$, let us define a map 
$$\Phi : \pi_1 (X,x_0) \to H_1(X)$$
$$\Phi([f]_{\pi_1})=[f]_H$$
where $[f]_{\pi_1}$ denotes the homotopy class of a loop and $[f]_H$ denotes the homology class of 1-cycle $f$. Infact the next  lemma shows that $\Phi$ is a homomorphism.

\begin{lemma}
If $f,g : \I \to X$ are paths in X such that $f(1)=g(0)$, then the 1-chain $f*g$ is homologous to f + g, that is $f + g - f *$ g  is a boundary in X.
\end{lemma}

\begin{proof}
Given $f,g : \I \to X$ such that $f(1)= g(0)$, define a singular 2-simplex $\sg : \tri 2 \to X$ by 
\begin{equation}\notag
\sg(x,y)=
\begin{cases}
f(x + 2y ), & x\le 1-2y\\
g(x + 2y -  1), & x \ge 1-2y
\end{cases}
\end{equation}

\begin{center}
\includegraphics[width=1\columnwidth]{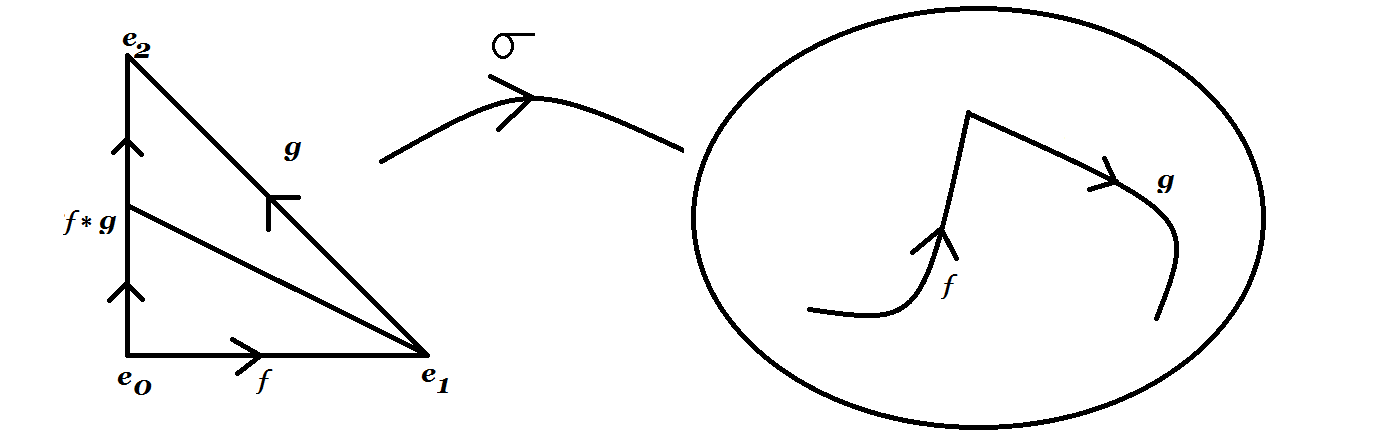}
\end{center}

\begin{center}
Figure 3.9
\end{center}

This is constant on lines parallel to the line joining $e_1$ with mid point of the edge $[e_0,e_2]$, and is continuous by the pasting lemma. Also
$$\partial (\sg )=g-  f*g+f.$$
\end{proof}

\begin{theorem}\label{ch3,sec6,th3}
Let X be a path connected space and $x_0 \in X$. Then $\Phi : \pi_1 (X,x_0) \to H_1(X)$ is surjective whose kernel is the commutator subgroup of $\pi_1(X,x_0)$. So $\Phi$ induces an isomorphism from the abelianization of $\pi_1(X,x_0)$ onto $H_1(X)$.
\end{theorem}

\begin{proof}
Since $X$ is path connected space, for each $x \in X$ let $\alpha_x : \I \to X$ be a specific path from $x_0$ to $x$ and $\alpha _{x_0}$ is the constant path. Since each path $\alpha _x$ is a 1-chain the map $x$ goes to $\alpha_x$ extends uniquely to a group homomorphism $\alpha : C_0(X) \to C_1(X)$. For any path $\sg$ in $X$, define a loop $\tilde{\sg}$ based at $x_0$ by
$$\tilde{\sg}=\alpha_{\sg(0)}*\sg*\bar{\alpha}_{\sg(1)}.$$
where $\bar{\alpha}$ denote the inverse path of $\alpha$.

Then
$$\Phi([\tilde{\sg}]_{\pi_1})=[\alpha_{\sg(0)}*\sg*\bar{\alpha}_{\sg(1)}]_H$$
$$\qquad \qquad \qquad \ =[\alpha_{\sg(0)}]_H+[\sg]_H+[\bar{\alpha}_{\sg(1)}]_H$$
$$\qquad \qquad \qquad \ =[\alpha_{\sg(0)}]_H+[\sg]_H - [\alpha_{\sg(1)}]_H$$
$$\qquad \quad =[\sg]_H -  [\alpha(\partial \sg)]_H.\ $$

Suppose $c=\sum_{i=1}^{k}n_i\sg_i$ is an arbitrary 1-chain and let $f$ be the loop given by
$$f=(\tilde{\sg}_1)^{n_1}*\ldots *(\tilde{\sg}_k)^{n_k}.$$

Since $\Phi$ is a homomorphism, we have
$$\Phi([f]_{\pi_1})=\sum_{i=1}^k n_i([\sg_ i]_H - [\alpha(\partial \sg_i)]_H)$$
$$=[c]_H - [\alpha(\partial c)]_H . \ $$
Now if $c$ is a cycle, then $\Phi([f]_{\pi_1}) = [c]_H$, hence $\Phi$ is surjective.

Since $H_1(X)$ is abelian, $Ker \Phi$ will contain the commutator subgroup of  $\pi_1(X,x_0)$.

Let $G = Ab(\pi_1(X,x_0)$) be the abelianization of $\pi_1(X,x_0)$ and for any loop $f$ based at $x_0$, let $[f]_G$ denote the equivalence class of $[f]_{\pi_1}$ in $G$. Because the product in $G$ is induced by path multiplication, we will indicate it with a ` $*$ '. For any singular 1-simplex $\sg$, let $\beta(\sg)=[ \tilde{\sg}]_G \in G$. Because $G$ is abelian this extends uniquely to a homomorphism $\beta : C_1(X) \to G$. We will show that $\beta$ takes all 1-boundaries to the identity element of $G$.

 Since $B_1(X)$ is generated by all elements of type $\partial(\sg)$ where $\sg : \tri 2 \to X$ is a singular 2-simplex. Let $\sg$ be be an arbitrary singular 2-simplex and $v_i=\sg(e_i)$. Put $\sg ^{(i)}= \sg \circ F_2^i$, so that $\partial \sg =\sg^{(0)} - \sg^{(1)}+\sg^{(2)}$. Then

\begin{center}
\includegraphics[width=1\columnwidth]{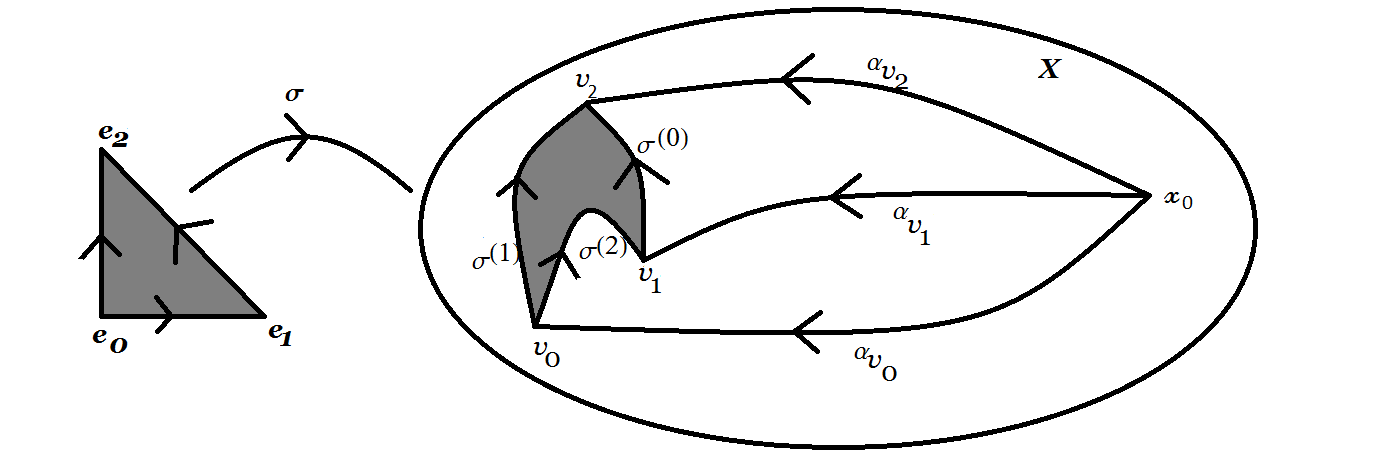}
\end{center}

\begin{center}
Figure 3.10
\end{center}

\begin{equation} \notag
\begin{split}
\beta (\partial \sg)& = [ \tilde{\sg}^{(0)}]_G * ([\tilde{\sg}^{(1)} ]_G)^{-1} * [\tilde{\sg}^{(2)} ]_G \\
&= [ \tilde{\sg}^{(0)}* (\tilde{\sg}^{(1)} )^{-1} * \tilde{\sg}^{(2)} ]_G \\
&=[\alpha_{v_1}*\sg^{(0)}*\bar{\alpha}_{v_2}*\alpha_{v_2}*\bar{\sg}^{(1)}*\bar{\alpha}_{v_0}*\alpha_{v_0}*\sg^{(2)}*\bar{\alpha}_{v_1}]_G \\
&=[\alpha_{v_1}*\sg^{(0)}*\bar{\sg}^{(1)}*\sg^{(2)}*\bar{\alpha}_{v_1} ]_G \\
&=[ \alpha_{v_1}*e_{v_1}*\bar{\alpha}_{v_1}]_G \\
&=[ e_{x_0}]_G \ \ . 
\end{split}
\end{equation}
Hence $B_1(X) \subset Ker \beta$.

Also for $[f]_{\pi_1} \in Ker \Phi ,\ [f]_H =0$ that is $f$ is a boundary. Since $\beta$ takes boundary to identity of $G$, it follows that $[f]_G=\beta(f)=1$. Thus $[f]_{\pi_1}$ is in the commutator subgroup.
\end{proof}


\chapter{Geometry of Surfaces}\label{ch4}

\section{Polygonal Presentation}
In this section we shall construct a number of compact surfaces as the quotient space obtained from a polygonal region in the plane by identifying its edges together. A subset $P$ of the plane is a polygonal region if it is a compact connected subset whose boundary is a finite 1-dimensional simplicial complex satisfying the following conditions: \\
(1) each point $q$ of an edge other than a vertex has a neighborhood $U$ in $\R^2$ such that $P \cap U$ is equal to the intersection with $U$ of some closed half plane.\\
(2) each vertex $v$ has a neighborhood $V$ in $\R^2$ such that $P \cap V$ is homeomorphic to the intersection of $V$ with 2 closed half planes whose boundaries intersect only at $v$.

Let $P$ be a $2n$-sided polygonal region. A labelling  of the edges of $P$ is a map from the set of edges of $P$ to a set $L$ of labels and giving each edge an arrow pointing  towards one of its  vertex in such a way that edges with the same label are to be identified with the arrows indicating which way the vertices match up. With each such labelling of a polygon we associate a sequence of symbols obtained by reading off the boundary labels counterclockwise from the top. Label $a_i$ in the sequence if the arrow points counterclockwise and $a_i^{-1}$ if it points clockwise. The quotient space obtained by pasting the edges of $P$ together according to the given orientation and labelling determines a connected topological space, being quotient of a single connected polygon.

\textbf{Example:}
\\
(1) Consider the polygon $P= \I \times \I$ with the \ori \ and labelling of the edges as specified in the Figure 4.1. The quotient space obtained by identifying the corresponding edges according to the equivalence relation  given by $ (s,0) \thicksim (s,1)$ and $(0,t) \thicksim (1,t)$ for all $s,t \in \I$. The resulting space is homeomorphic to $\s^1 \times \s^1$, the torus.
\begin{center}
\includegraphics[width=1\columnwidth]{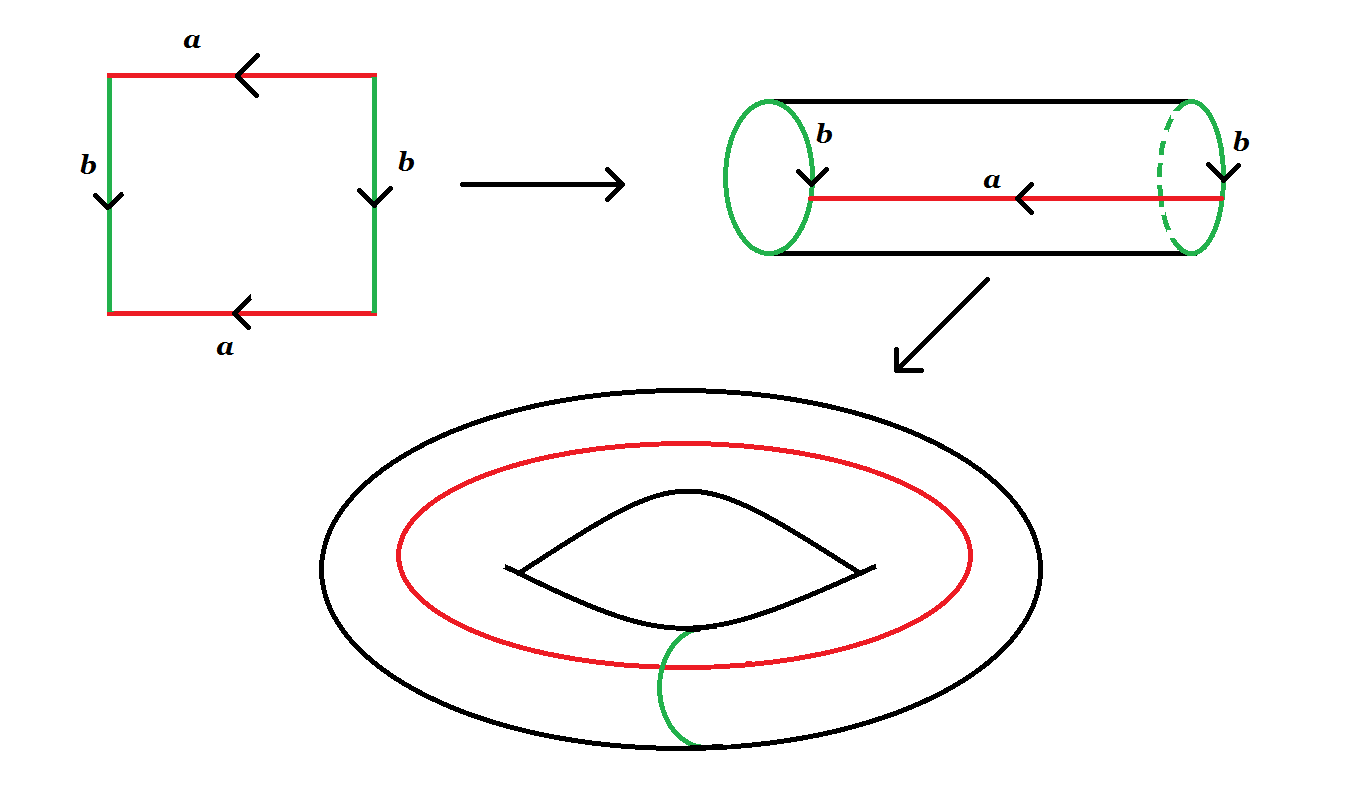}
\end{center}

\begin{center}
Figure 4.1\\
\end{center}

(2) Also the quotient space of a polygon $P= \I \times \I$ obtained by means of the  \ori \ and labelling $abab^{-1}$ as indicated  in the Figure 4.2 below is called a Klein bottle. This space can not be embedded in $\R^3$. However Figure 4.2 depicts an immersion of this space in $\R^3$ with self intersection.\\

\begin{center}

\includegraphics[width=1\columnwidth]{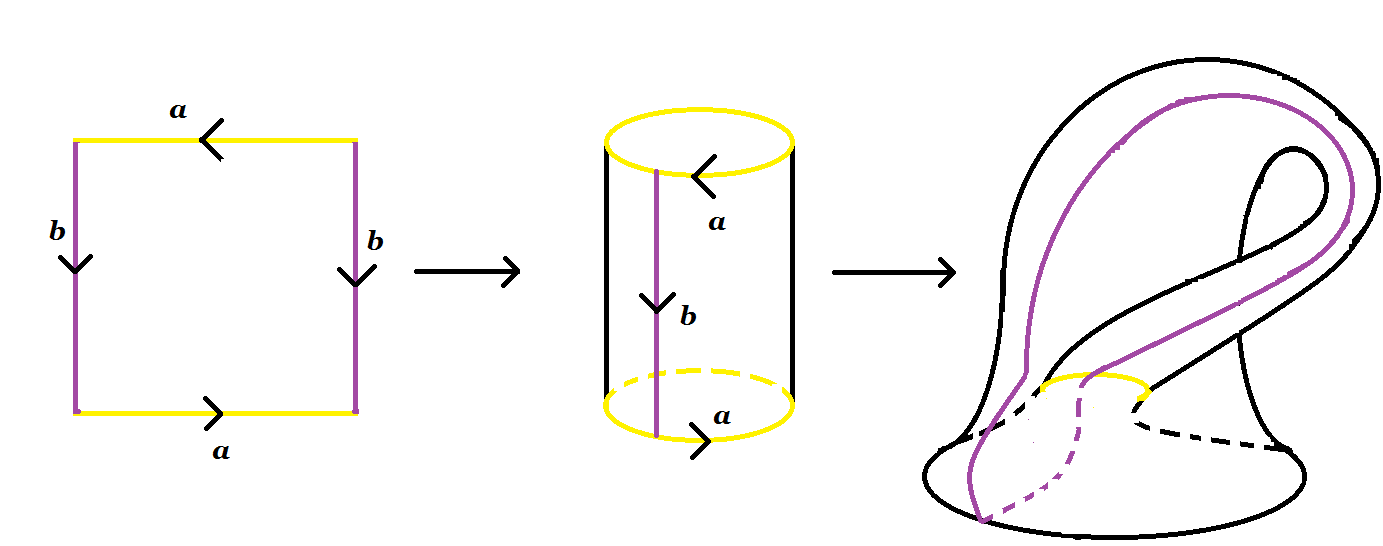}
\end{center}

\begin{center}
Figure 4.2
\end{center}

\section{Connected Sum}

Let $M_1$ and $ M_2$ be two surfaces. The connected sum of $M_1$ and $ M_2$  is formed  by cutting out a small open disk from each surface and then gluing the resulting spaces together along their  boundary. To be precise, let $B_1 \subset M_1$ and  $B_2\subset M_2$ are open disks and $\alpha \colon \partial B_1 \to \partial B_2$ be a homeomorphism (both being homeomorphic to $\s^1$). Let $M_i'=M_i \setminus B_i$ and define a quotient space of disjoint union of $M'_1$ and $ M'_2$ by identifying each $x\in \partial B_1$ with $\alpha(x)\in \partial B_2$. The resulting quotient space is called connected sum of $M_1$ and $M_2$ and is denoted by $M_1 \# M_2$. \\

\begin{center}
\includegraphics[width=1.2\columnwidth]{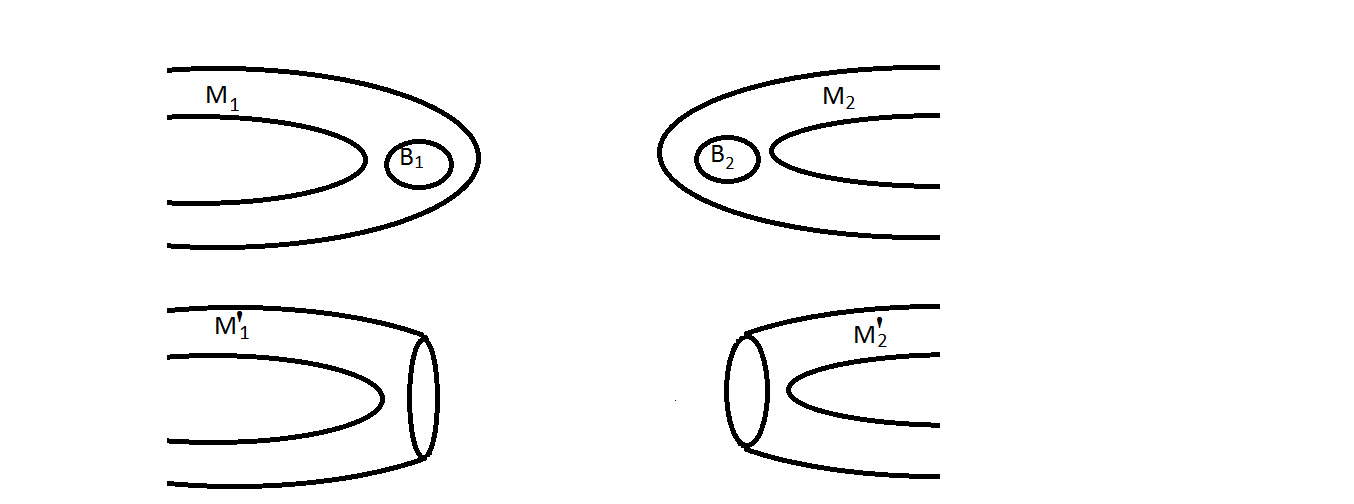}
\includegraphics[width=1.2\columnwidth]{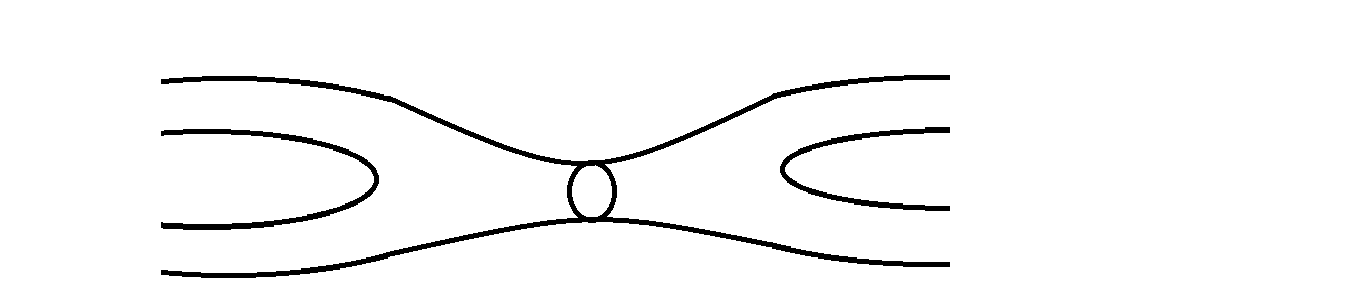}
\end{center}

\begin{center}
Figure 4.3
\end{center}

\begin{theorem}
Connected sum of two connected \emph{2}-manifolds is a \emph{2}-manifold.
\end{theorem}

\begin{proof}
 Let $M_1$ and $ M_2$  be 2-manifolds and let $B_i \subset M_i,\  i=1,2$ be open disks. Then  $M_i '=M_i \setminus B_i$ are connected 2-manifolds with boundary. Let $\pi \colon M' \to M$ be the quotient map, where $M'=M_1'\sqcup M_2'$ and $M=M_1 \# M_2$, and let $ D = \partial M_1' \sqcup \partial M_2' $. Since $\pi$ is a quotient map therefore $ \pi \vert _{M' \setminus D}$ is a homeomorphism. Thus, for each  $p \in M$$\setminus$$\pi (D) $ there exist an open neighborhood homeomorphic to an open subset of $\R^2$.

Also  any $p \in \pi(D) $ has exactly two inverse images $ p_i \in \partial M_i$. Since each $M_i'$ is a  2-manifold with boundary, there exist disjoint neighborhoods $U_i$ of $p_i$ and homeomorphisms $\al _i $ taking $U_i$ to a half disk in the upper half plane and maps $p_i$ to the origin. But we can consider $\al _2(U_2)$ be a half disk in the lower half plane with $p_2$ goes to the origin and identifying the boundary of $\al _i(U_i)$. Define $\al \colon U_1 \cup U_2 \to \mathbb{B} ^2$ by setting $\al = \al _i$ on $U_i$. Now $\al$ is continuous and closed map. Shrinking $U_i$ if necessary, we can ensure that $U_1 \cup U_2=U$ is a saturated open set in $ M'$ . Thus we have the map $\tilde{\al}$, so that the following diagram commutes:
$$ \	 U $$
$$\qquad \pi \ \ \big\downarrow \	  \searrow\  \al$$
$$\qquad \tilde{\al} \colon\pi(U) \longrightarrow \mathbb{B} ^2$$

Since $U$ is a saturated open set in $M'$, $\pi(U)$ is open and $\tilde{\al}$ is the desired homeomorphism.

To show that $M$ is Hausdorff, we shall consider the following cases:\\
(1) For $p \in M\setminus \pi(D)$ and $q \in \pi(D)$, let  $\pi ^{-1}(q) = q_i \in M_i$ for $i=1,2$. Since each $M_i$ is Hausdorff, for $\pi ^{-1}(p) = p' \in M_1$ (or $M_2)$ there exist disjoint  neighborhoods $U_{p'}$ and  $U_{q_1}$ of $p'$  and $q_1$ respectively. Let $U_{q_2}$ be a neighborhood of $q_2$  in  $M_2$. Shrinking $U_{q_1}$ and $U_{q_2}$ if necessary, we  can assume $U_{q_1} \cup U_{q_2}$ is a saturated open set. Then $\pi (U_{p'})$  and $\pi (U_{q_1 }\cup U_{q_2})$ are the desired disjoint neighborhoods of $p$ and $q$.\\
(2) For $p,q \in \pi (D) $, let  $\pi ^{-1}(p) = p_i \in M_i$ and $\pi ^{-1}(q) = q_i \in M_i$ for $i=1,2$. As in the previous case, we will get disjoint saturated open sets $U_{p_1} \cup U_{p_2 }$ and $U_{q_1} \cup U_{q_2}$ so that $\pi (U_{p_1}\cup U_{p_2})$ and $\pi (U_{q_1}\cup U_{q_2})$ are the disjoint neighborhoods of $p$ and $q$.\\
(3) For $p,q\in M\setminus \pi (D)$, since $\pi \vert_{M'\setminus D}$ is a homeomorphism the result holds.

To show that $M$ is second countable, consider a covering $\mathfrak{U} $ of $M$ by open Euclidean disks. The collection $\{\pi^{-1}(U) \colon U \in \mathfrak{U}\}$ is an open cover of $M'$, which has a countable subcover. Let $\mathfrak{U}'$ be a countable subset of $\mathfrak{U}$ such that $\{\pi^{-1}(U) \colon U \in \mathfrak{U}'\}$ covers $M'$, then $\mathfrak{U}'$ is a countable cover of $M$ by Euclidean disks. Each such disk has a countable basis and the union of all these bases is a countable basis  for $M$.
\end{proof}

The notion easily gets extended to the case when manifolds are not connected, since only one component from each of them will be involved in the entire operation.

\section{Genus of a Surface}

A $handle$ $H$ is a space obtained from torus by deleting an open disk. A handle is shown in Figure 4.4 below. Given a surface $M$ we can attach a handle to $M$ by cutting a hole in the surface and gluing the $\partial H$ to the hole.
\begin{center}
\includegraphics[width=1\columnwidth]{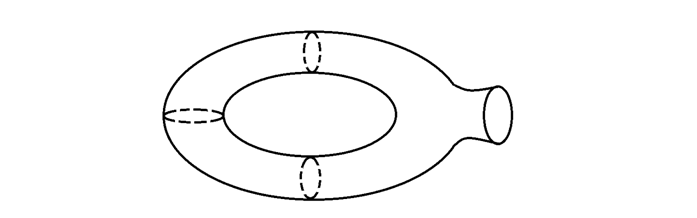}
\end{center}

\begin{center}
Figure 4.4
\end{center}

If we consider attaching a handle to a sphere, we could appeal to the Jordan curve theorem, which states that any homeomorphic image of a circle on the sphere separates it into two disjoint regions each homeomorphic to an open disk. We then remove one of  these disk and glue the torus with a hole along the boundary circle. A sphere with $k$ holes is a space obtained by deleting $k$ disjoint open disks from a sphere. If a handle is attached to the boundary of each of the holes, the resulting space is a sphere with $k$ handles. Evidently such a space will be homeomorphic to connected sum of $k$ tori. 
\begin{center}
\includegraphics[width=1\columnwidth]{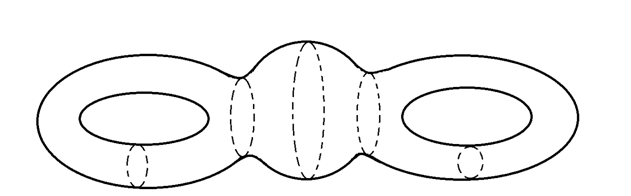}
\end{center}
\begin{center}
Figure 4.5
\end{center}

Attaching a M$\ddot{o}$bius band to the sphere is carried out by removing an open disk from the sphere and then gluing its boundary circle to the boundary circle of the  M$\ddot{o}$bius band. Attaching a M$\ddot{o}$bius band to the sphere results in the projective plane. One can also obtain this by removing a disk from the sphere and then identifying antipodal points on the boundary. The space obtained by attaching a M$\ddot{o}$bius band to each boundary component of a sphere with  $k$ holes  is  called a sphere with $k$ cross-caps. The resulting space will be homeomorphic to a connected sum of $k$ projective planes. In particular connected sum of two projective planes is homeomorphic to the Klein bottle.

\textbf{Definition:} The $genus$ of a surface is an integer representing the maximum number of cutting along non intersecting consecutive simple closed curves without rendering the resultant manifold disconnected.

For example, the genus of a sphere is 0. Since any closed circular cut separates the sphere into two components. A surface in which every Jordan curve seprates it, is said to be planar.  On the other hand, the genus of a torus is 1 as there is only one circular cut possible on the torus that will not separate it into two parts. Evidently any two consecutive circular cuts would disconnect the torus see Figure 4.6 below. Alternatively genus of an orientable surface is the number of handles on it. Also the genus of a nonorientable 

\begin{center}
\includegraphics[width=2.2\columnwidth]{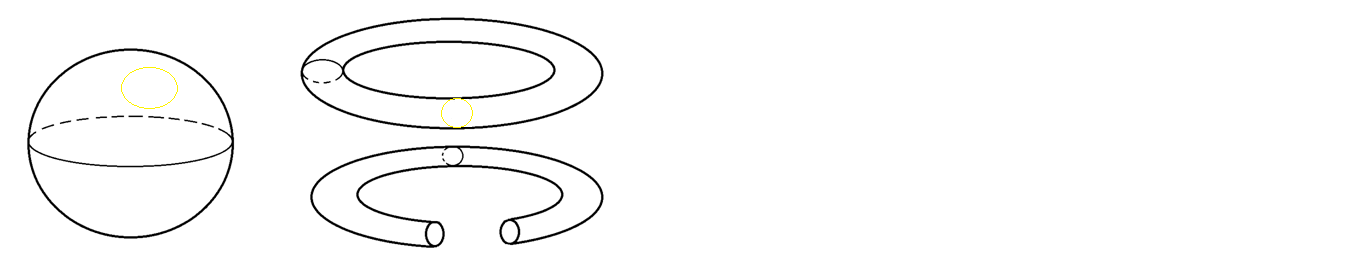}
\end{center}

\begin{center}
Figure 4.6
\end{center}

\begin{center}
\includegraphics[width=1\columnwidth]{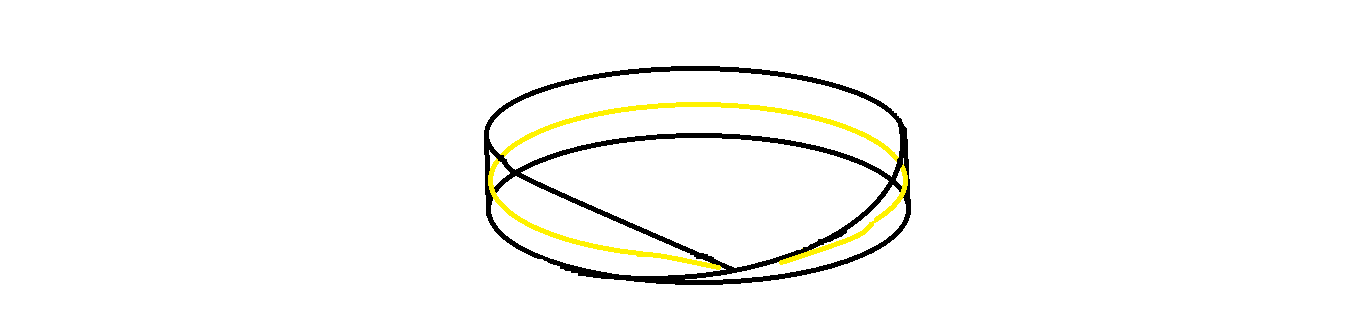}
\end{center}

\begin{center}
Figure 4.7
\end{center}
surface is the number of M$\ddot{o}$bius bands attached to a sphere. The genus of a surface with boundary is defined to be the genus of the  surface obtained by capping off a disk to each of the boundary component.

We denote the orientable surface obtained by taking the connected sum of $g$ copies of a torus by $\Sigma_g \  , \ g \ge0$, that is $\Sigma_g$ is a surface obtained by attaching $g$ handles to the sphere and $\Sigma_0$ is simply a sphere. Also we denote an orientable surface of genus $g$ with $k$ boundary components obtained by making $k$ holes in $\Sigma_g$ by $\Sigma_{g,k}$.

And we denote the nonorientable surface obtained by taking connected sum of $g$ copies of a projective plane by $U_g \ , \ g \ge 1$, that is  $U_g$ is a surface obtained by attaching $g$ M$\ddot{o}$bius bands to the sphere. Similarly we denote  a nonorientable surface of genus $g$ with $k$ boundary components by $U_{g,k}$.

\section{Euler Characteristic}

Let $K$ be a finite complex of dimension less than or equals to two. Then $\chi (K) = |V| - |E| - |F|$ is called the $Euler \  characteristic$  of $K$, where $|V|,\  |E|$  and $|F|$ are the number of vertices, edges and 2-faces of $K$ respectively. In general the Euler characteristic for a complex with finitely many simplexes is defined to be the sum of number of even dimensional simplexes minus the number of odd dimensional simplexes.

\textbf{Examples:}\\
The Euler characteristic of some of the surfaces is given below:\\
(1) A triangulation of a 2-sphere $\s^2$  is given in the Figure 4.8 below. It has 4 vertices, 6 edges and 4 faces. Hence the Euler characteristic of a sphere is given by
\begin{equation} \notag
\begin{split}
\chi&=4-6+4\\
&=2
\end{split}
\end{equation}

\begin{center}
\includegraphics[width=1.5\columnwidth]{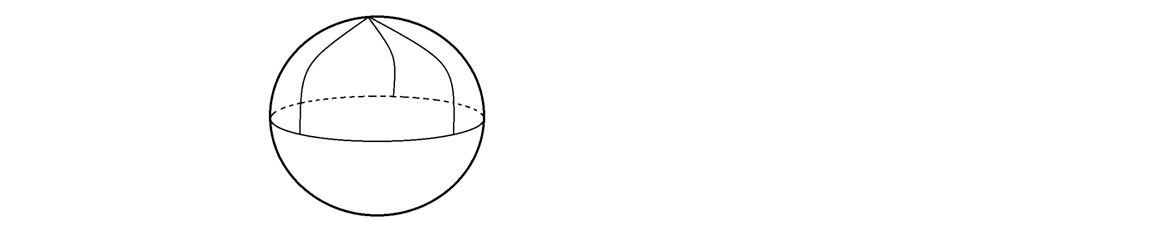}
\end{center}

\begin{center}
Figure 4.8
\end{center}
(2) Another example is of a torus. One of its triangulation given in Figure 4.9 has 9 vertices, 27 edges and 18 faces. And the Euler characteristic is
\begin{equation} \notag
\begin{split}
\chi&=9-27+18\\
&=0
\end{split}
\end{equation}

\begin{center}
\includegraphics[width=1.5\columnwidth]{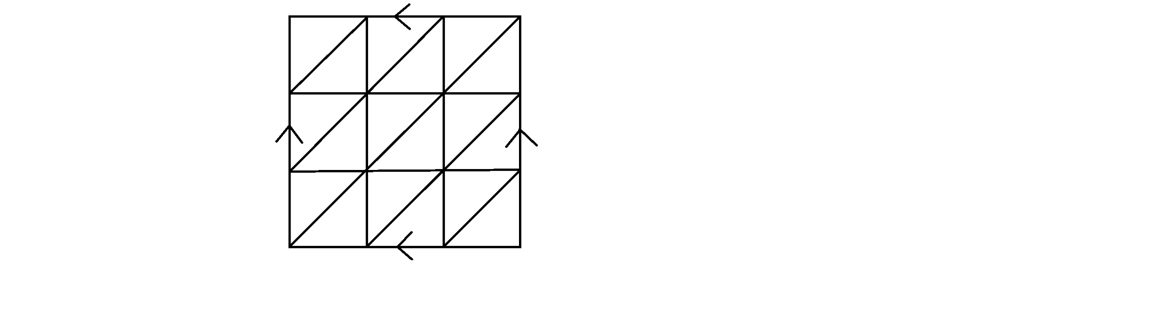}
\end{center}

\begin{center}
Figure 4.9
\end{center}
(3) A triangulation of a projective plane has 7 vertices, 18 edges and 12 faces, its Euler characteristic is 
\begin{equation} \notag
\begin{split}
\chi&=6-15+10\\
&=1
\end{split}
\end{equation}
\begin{center}
\includegraphics[width=1.1\columnwidth]{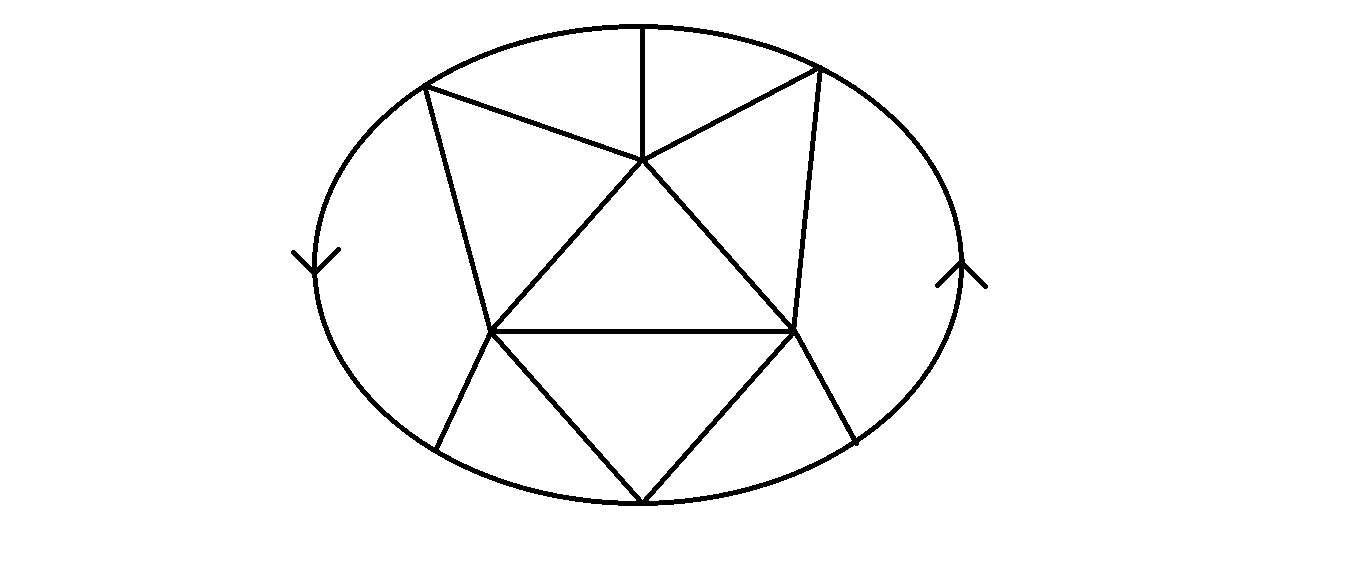}
\end{center}
\begin{center}
Figure 4.10
\end{center}

The Euler characteristic of a compact bordered surface can be obtained in the same way. Infact if the components of $\partial$M are $J_1,J_2,\ldots ,J_n$ and for each $i$  let $B_i$ be a disk with $\partial B_i= J_i$  such that
 $$ M'=M \cup (\bigcup_i^n B_i)$$
 is a surface. Then
\begin{equation}\label{ch4,sec4,eq1}
\chi(M)= \chi(M')-n .
\end{equation}

\begin{theorem}
Let $M_1$ and $M_2$  be two compact surfaces. The Euler characteristic of  $M_1 \# M_2$ is given by
\begin{equation}\label{ch4,sec4,eq2}
\chi(M_1 \# M_2) =\chi(M_1)+ \chi(M_2) - 2.
\end{equation}

\end{theorem}
\begin{proof}
Let $M_1$ and $M_2$ be two compact surfaces. Then $M_1 \# M_2$ is formed by removing from each surface the interior of a disk and then identifying the boundaries of the removed disks. And the result follows by the formula.
\end{proof}

From Equation \ref{ch4,sec4,eq1} and Equation \ref{ch4,sec4,eq2} we can express the Euler characteristic of a 2-manifold with or without boundary of  genus $g$ with $k$  boundary components as follows
$$\chi(\Sigma _{g,k})=2-2g-k $$
and
$$\chi(U_{g,k})=2-g-k.$$
We shall discuss an alternative proof of this later in next chapter.

Recall that, if dim$K$ = 2, then there are no $n$-simplexes for $n > 2$. There are only finite number of $n$-simplexes in $K$ and consequently  $C_n(K)$ is a finitely generated free abelian group of rank equal to the number of $n$-simplexes in $K$. Because $Z_n(K)$ and $B_n(K)$ are subgroups of $C_n(K)$ and any subgroup of a free abelian group is free abelian of rank less than or equal to that of the group. The factor group $H_n(K) = Z_n(K) / B_n(K)$ must be finitely generated abelian group which may or may not be free. Also for any finite complex $K$, the simplicial homology group $H_n(K) $ is isomorphic to the singular homology group $H_n(|K|)$ for all $n$. Then from the following theorem we can conclude that the Euler characteristic is  a topological invariant of $|K|$.

\begin{theorem}
The Euler characteristic of a finite complex K of dimension \emph{2}  is given by 
$$\chi(K) = rank H_0(|K|) - rank H_1(|K|) + rank H_2(|K|).$$
\end{theorem}

\begin{proof}
Let $c_n$ denote the number of $n$-simplexes in $K$, then
$$\chi(K) = c_0 - c_1 + c_2.\qquad$$

Also rank of $C_n(K)=c_n$. Since $B_{n-1}(K)$ is isomorphic to $C_n(K) / Z_n(K)$ we have
$$rank B_{n-1}(K)=rank C_n(K) - rank Z_n(K)$$
also
$$\ \ \ rank H_n(K)= rank Z_n(K) - rank B_n(K)$$
we get 
$$\qquad\qquad \qquad \qquad \ \ \ rank C_n(K) = rank B_{n-1}(K) + rank H_n(K) + rank B_n(K).$$

Since $B_2(K)=0= B_{-1}(K)$. Therefore
$$\chi(K) = c_0 -c_1 +c_2 \qquad\ $$
$$\qquad \qquad \qquad \qquad \qquad \qquad  \ \ =rank C_0(K) -rank C_1(K) + rank C_2(K)$$
$$\qquad \qquad \qquad \qquad \qquad \qquad\ \ = rank H_0(K) - rank H_1(K) + rank H_2(K) \ .$$
Because $rank H_n(|K|) = rank H_n(K)$, this completes the proof.
\end{proof}

\section{Orientation}

We now compare a M$\ddot{o}$bius band which is a famous bordered surface where we can not distinguish the two faces, with a cylinder where we can distinguish them. The  M$\ddot{o}$bius band is the topological space obtained by identifying two edges of the square $\I\times \I$ according to the relation $(0,t)\thicksim(1,1-t)$. It is a manifold with boundary. It has the curious property that it is impossible to consistently pick out which is the front side and which is the back side.

\begin{center}
\includegraphics[width=1\columnwidth]{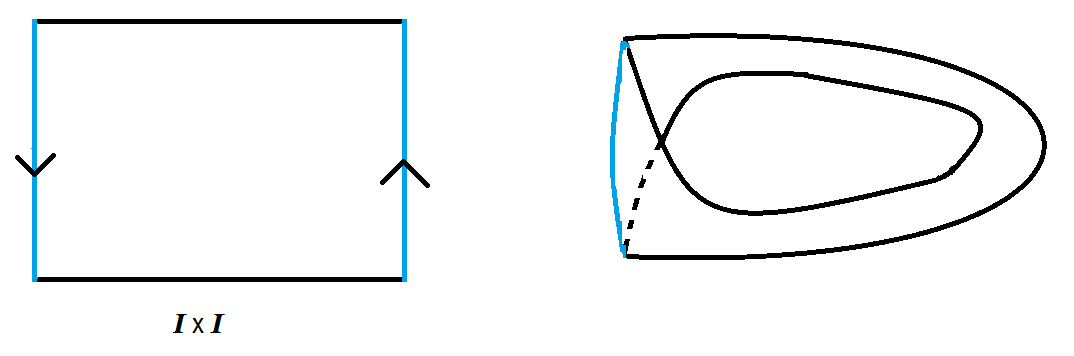}
\end{center}

\begin{center}
Figure 4.11
\end{center}

The essential point is whether, when a surface is constructed by gluing, we can distinguish the two faces. A surface on which this can be done is called an orientable surface. We will make this notion precise and define \ori \ at each point on the surface and there are exactly two. We shall present them by two kinds of arrows clockwise and anticlockwise, we call them opposite orientations. When an orientation is specified at a point, the same \ori \ is specified at an arbitrary point in a neighborhood of the point. This is called the $coherent\   \ori$. We specify an \ori \ at a point on the surface and choose the coherent \ori \ at each point on a curve starting from the point. If the curve goes back to the starting point then the original \ori \ at that point may or may not coincide with the \ori \ propagated along the curve. A surface is called orientable if the \ori \  propagated along any curve always comes back to the starting \ori . In this case we can assign an \ori \ to all points on the surface in such a way that near points have mutually coherent \ori s.

However we define the M$\ddot{o}$bius band, the center line of the rectangular strip becomes a circle after the gluing. Suppose that $P$ is a point on it. From a purely local point of view there is a corresponding point $P'$ on the  surface. But since M$\ddot{o}$bius  band is a one sided surface, it is possible to draw a continuous path from $P$ to $P'$ without crossing the boundary curve. Such a path is depicted in the Figure 4.12  below. If the small oriented closed curve drawn around $P$ is now slide along the path $PP'$ and when it eventually arrives at $P'$ its \ori \   reversed.

\begin{center}
\includegraphics[width=1\columnwidth]{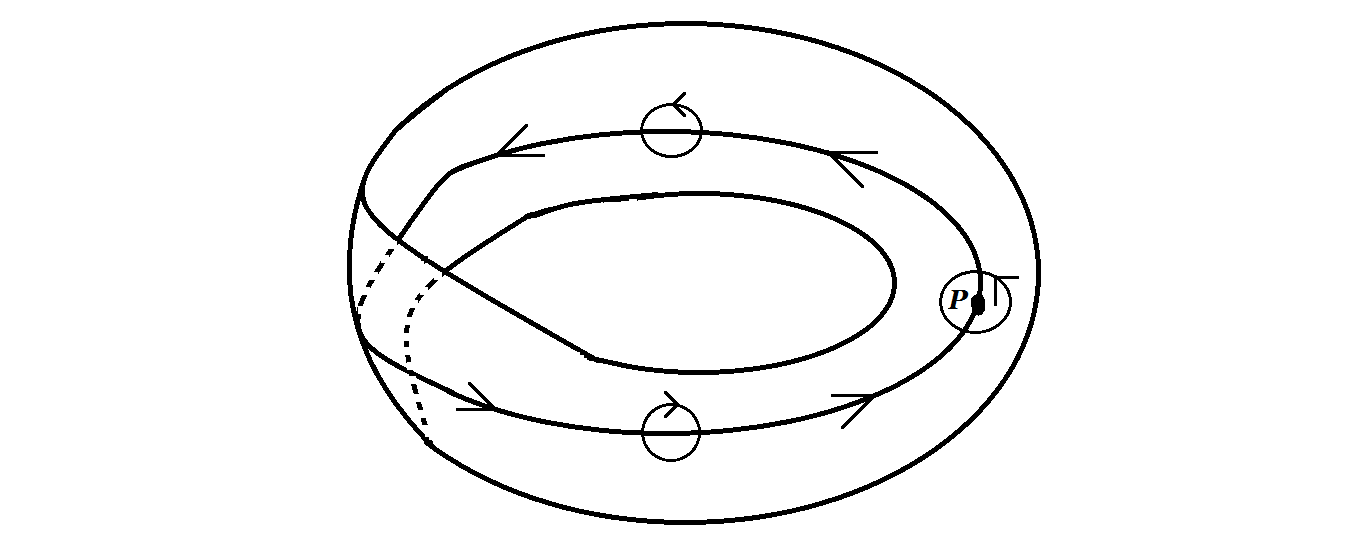}
\end{center}

\begin{center}
Figure 4.12
\end{center}

Now the projective plane is easily seen to be nonorientable, infact it contains a subset homeomorphic to a M$\ddot{o}$bius band. Projective plane is obtained from the closed upper hemisphere of $\s^2$ as quotient space obtained by identifying diametrically opposite points on the boundary. Consider a thin strip `$S$' made up of open segments of meridian whose center lay on half an equator. Under identification of antipodal points, clearly `$S$'  becomes an open M$\ddot{o}$bius band in projective plane. Thus projective plane is nonorientable.\\

\begin{center}
\includegraphics[width=1\columnwidth]{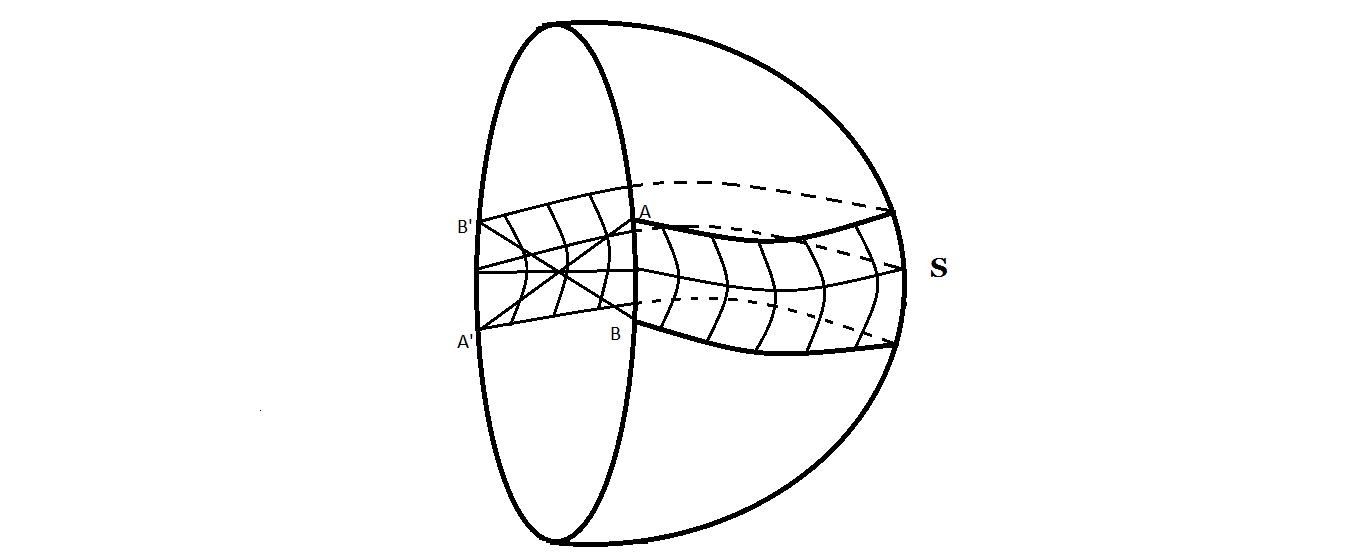}
\end{center}
\begin{center}
Figure 4.13
\end{center}

By a similar argument it can be shown that the Klein bottle is also nonorientable. 

Moreover, as every surface is triangulable, we say that a surface $S$ with a triangulation $K$ is orientable  if it is possible to orient all the triangles of $K$ in a compatible manner, that is in such a way that any two adjacent triangles always induce opposite \ori \ on their common edge as depicted in Figure 4.14 below. In the next chapter we shall show that the condition of orientability is independent of the choice of triangulation. That is, if $S$ is a surface triangulated by $K$ in compatible manner, then any other triangulation of $S$ must also be compatible.\\

\begin{center}
\includegraphics[width=1\columnwidth]{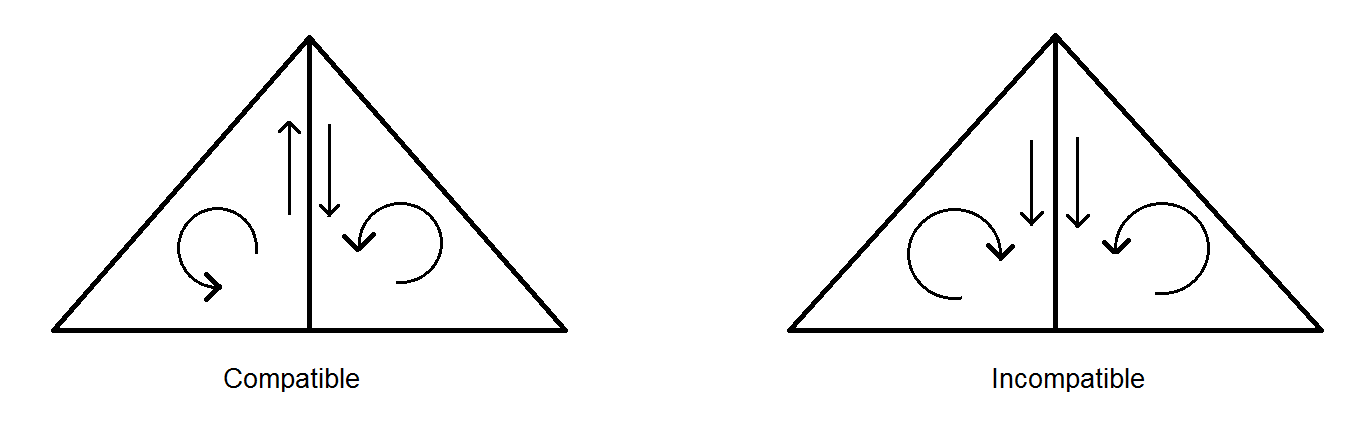}
\end{center}

\begin{center}
Figure 4.14
\end{center}


\chapter{Classification of Compact 2-Manifolds}\label{ch5}

\section{Decomposition of Surfaces}

Two manifolds are homeomorphic if and only if there exist a one to one correspondence between their components such that the corresponding components are homeomorphic. Thus, for classification of 2-manifolds it suffices to classify only  surfaces. In this chapter we shall prove the following classification 
\begin{theorem}
 Let $ M_1$ and $M_2 $ be compact surfaces with boundary. Then $ M_1$ and $M_2 $ are homeomorphic if and only if\\
\emph{(1)} they are both orientable or nonorientable; \\
\emph{(2)} they have same Euler characteristic or genus; \\
\emph{(3)} they have the same number of boundary components.
\end{theorem}

Let $K$ be a triangulation of a compact surface $M$. Since $M$ is compact and connected by Theorem \ref{ch5,th2} there is an open set $F$ given by 
$$F= \bigcup_{i=1}^n int (\sigma _i) \  \cup \ \bigcup_{i=1}^{n-1} int (e_i) $$
where $  M= \bigcup_{i=1}^n  \sigma _i$ and  $e_i$ is  a  common  edge  of $\sigma _i$ and $\sigma_{i+1} $. Put 
$$D' = \bigcup_{i=1}^n N(b_{\sigma_i}) \cup \bigcup_{i=1}^{n-1} N(b_{e_i})$$ 
which is homeomorphic to a disk.

Let $G$ be the 1-dimensional  subcomplex of $K$ such that $\vert G \vert = \partial F$. Since $\partial F  $ is connected, $|G|$ is a connected simple graph. A graph which does not contain any loops is called a tree. If $|G|$ is not a tree then it contains an edge whose removal  produces a graph $|G'|$  which is still connected. Continue like this since $|G|$ is finite the process terminates after removing finitely many edges $e_1,e_2,\ldots ,e_s, \ s\ge 0$, resulting a tree $|T|$ that contains every vertex of $G$ and has fewer edges than $G$. 

\begin{theorem}
Let T be a tree which is a subcomplex of K. Then SB-neighborhood N(T) of T is homeomorphic to a closed disk .
\end{theorem}
\begin{proof}
If $T$ has a single vertex with no edge, then the result is clear. Suppose that every tree in $K$ with $k$ vertices has a SB-neighborhood in $K$ homeomorphic to a disk. Let a tree $T$ in $K$ has $k+1$ vertices and let $v$ be a leaf and  $w$ be the parent of $v$. Removing from $T$ the vertex $v$ and the edge $e$ connecting $v$ to $w$, to produce a tree $T'$ in $K$ with $k$ vertices. Then by induction hypothesis $N(T')$ is homeomorphic to a disk. And
$$  N(T)=N(T')  \cup N(b_e) \cup N(v). $$

Thus $N(T)$ is homeomorphic to a disk.
\end{proof}

Let  $D= N(T)$ so that $D$ is homeomorphic to a closed disk and let  $S_i=N(b_{e_i})$. Then 
$$ M'=D \cup \bigcup_{i=1}^s S_i \  = N(G) $$ 
is a bordered surface  having one boundary component. The sets $S_i$ will be called strips and $M'$ will be called a disk with strips attached. Thus 
\begin{equation}\notag
\begin{split}
M &= D' \cup N(G) \\
& =D' \cup D \cup \bigcup_{i=1}^s S_i 
\end{split}
\end{equation}
with disjoint interiors such that \\
(1) $S_i \cap D$ and $S_i \cap D' $ is the union of 2 disjoint arcs for each $ i=1,2,\ldots ,s.$ \\
(2) $ D \cap D' $ is the union of $2s$ disjoint arcs.

The above discussion gives a new decomposition $K'$ of $M$ such that $K'$ has exactly one face $F$ and has fewer edges than $K$. Suppose $K'$ has $n$ vertices, infact these are the vertices of the tree $T$ so that $K'$ has ($n-$1) edges of the tree plus the edges $  e_1,e_2,\ldots ,e_s, \ s \ge 0  $ which were removed from $G$. Hence the total number of edges are ($n-1+s$). Thus 
\begin{equation}\begin{split}\notag
\chi(M) &= \chi(K') \\
& =\vert V \vert -\vert E \vert + \vert F \vert \\
& =n-(n-1+s)+1  \\
&    =2-s
\end{split}
\end{equation}
\begin{equation}
\begin{split}\label{ch5,sec0,eq1}
\qquad \qquad  \quad  \quad= 2 - {\rm{no.\  of \  strips \  attached. }} \	
\end{split}
\end{equation}

Thus we have the following
\begin{theorem}
The Euler characteristic of a compact surface  M is not greater than $2$. If $ \chi(M)$ =$2$, then M is homeomorphic to $\s^2$.
\end{theorem}

\section{Representation of Surfaces}

In the earlier chapter we have presented the polygonal representation of surfaces. Now from the above discussion, we have got another representation of a compact  surface $M$. By removing a disk $ D'$ from $M$  it reduces to $ M' $ a surface with a boundary component that can be expressed as a union 
$$ M'=D \cup \bigcup_{i=1}^s S_i  .$$ 

Since each strip $S_i$ is homeomorphic to a disk, there exists an ordinary rectangular strip $R_i$ in $\R^3$ and a homeomorphism of $S_i$ onto $R_i$. We can assume that the strips $R_1,R_2,\ldots ,R_s$ are pairwise disjoint otherwise we can translate some of them. Also we have a disk $\widetilde{D}$  in $ \R^3$ homeomorphic to $D$.

Since $D \cap S_i$ is the union of 2 disjoint arcs, we will form a quotient space of $\widetilde{D} \cup R_i $ by identifying the corresponding  pair of opposite edges of $R_i$ with two disjoint arcs of $\widetilde{D}$. This can be done in two ways depending upon the orientability of surface:

(1) \textbf{Attaching oriented strip:} If $M$ is orientable then the 2-simplexes in the triangulation of $M$ having common face contained in $D \cap S_i$ should have compatible  orientation. 
\begin{center}
\includegraphics[width=1.2\columnwidth]{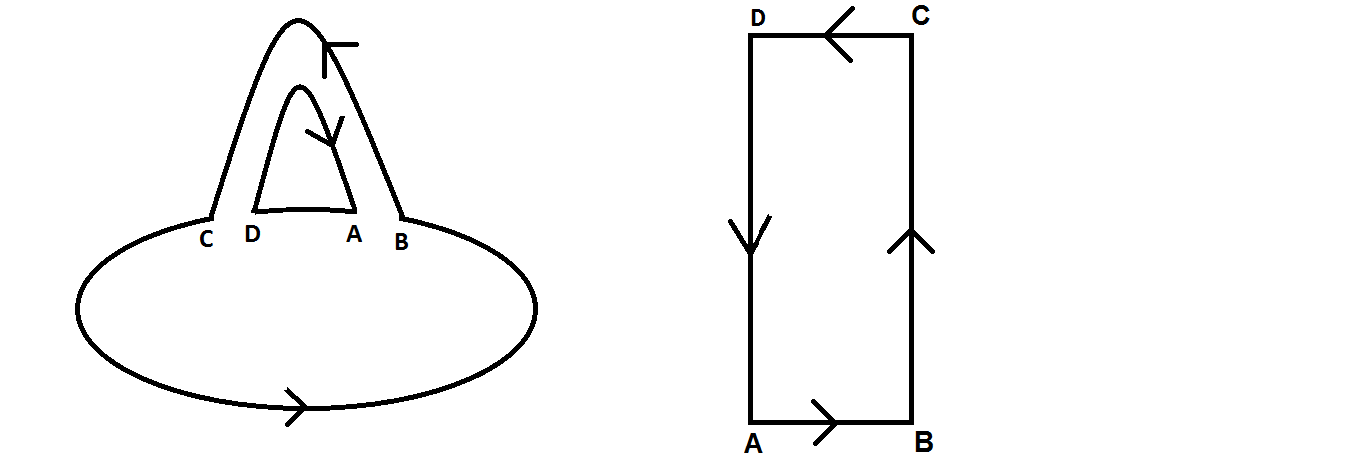}
\end{center}

\begin{center}
Figure 5.1 
\end{center}

\begin{center}
\includegraphics[width=1\columnwidth]{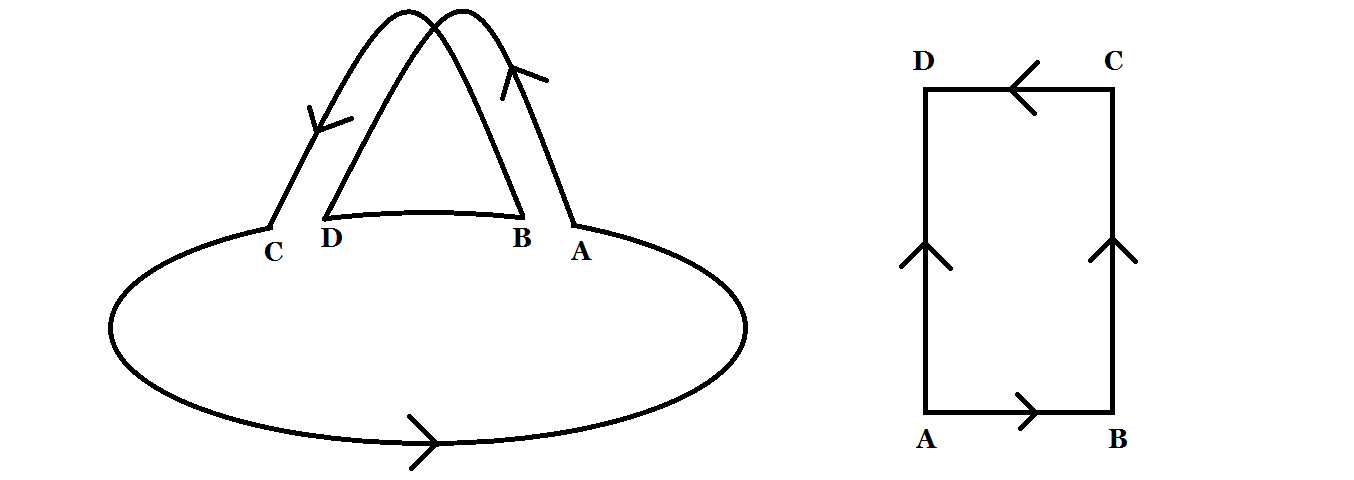}
\end{center}
\begin{center}
Figure 5.2
\end{center}
Attach the rectangle $R_i$ to $\widetilde{D}$ so that  opposite edges are identified with disjoint segments of $\widetilde{D}$ both in orientation reversing way. In that case $S_i \cup D$ or $ R_i \cup \widetilde{D}$ is an annulus and $S_i$ will be called annular strip see Figure 5.1.

(2) \textbf{Attaching unoriented strip:} If $M$ is nonorientable then there are 2-simplexes in the triangulation of $M$ that do not have compatible orientation and their common face lie in one of the component of $D \cap S_j$. Attach the rectangle $R_j \ to \ \widetilde{D}$ so that one pair of opposite edges are identified with disjoint segments of $\widetilde{D}$, one in orientation reversing way and the other in orientation preserving way. In that case $S_j \cup D$ or  $R_j \cup \widetilde{D}$ is a M$\ddot{o}$bius  band and $S_j$ will be called twisted strip or a cross-cap see Figure 5.2.

\textbf{Note:}

(1) Since boundary of $M'$ is connected and homeomorphic to $\s^1$, if there is an annular strip $S_i$, then there must be another annular strip (or M$\ddot{o}$bius  band) $S_j$, which is intertwined with $S_i$ as in Figure 5.3 below. One can easily see that, in Figure 5.3 (a) 
\begin{center}
\includegraphics[width=1\columnwidth]{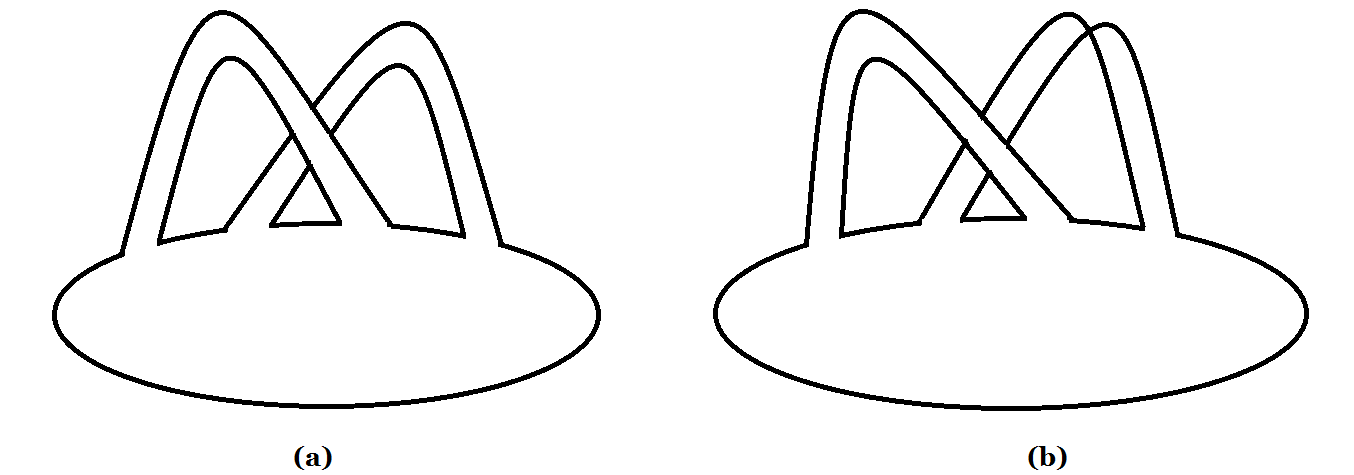}
\end{center}
\begin{center}
Figure 5.3
\end{center}
$D \cup S_i \cup S_j $ is  homeomorphic to a torus with a  disk removed and will be called a handle (or a Klein bottle with a disk removed in (b)).

(2) Also $M$ is orientable if and only if no cross-cap is attached  to $\widetilde{D}$ in the above representation of $M'$. 

Thus such a $M'$ can always be embedded in  $\R^3$ and can be represented as in the Figure 5.4 below.

\begin{center}
\includegraphics[width=1\columnwidth]{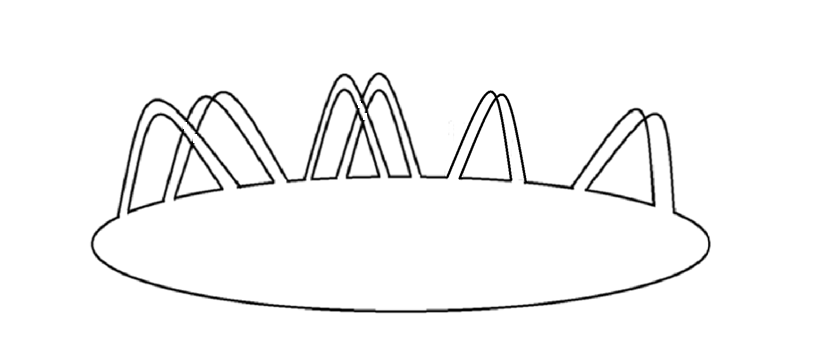}
\end{center}

\begin{center}
Figure 5.4
\end{center}

\textbf{Example:}

(1) A representation for torus can be seen in the Figure 5.5 below, where a pair of intertwined annular strips are attached to a disk.\\

\begin{center}
\includegraphics[width=1 \columnwidth]{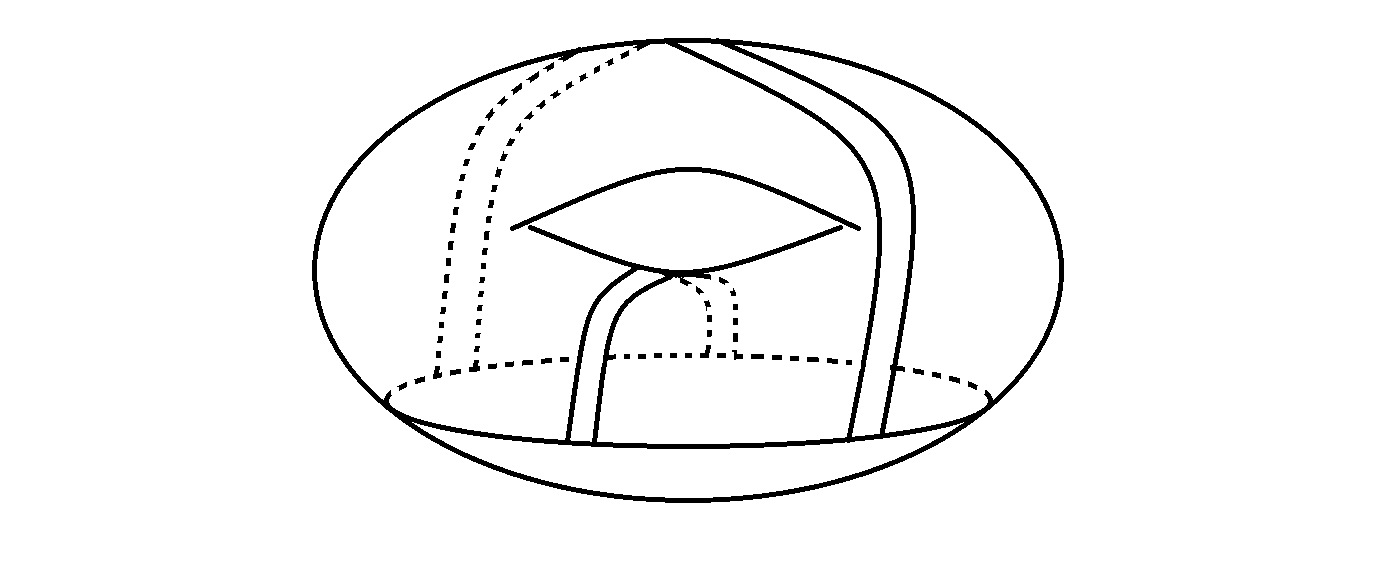}
\end{center}
\begin{center}
Figure 5.5 
\end{center}
\begin{center}
\includegraphics[width=1\columnwidth]{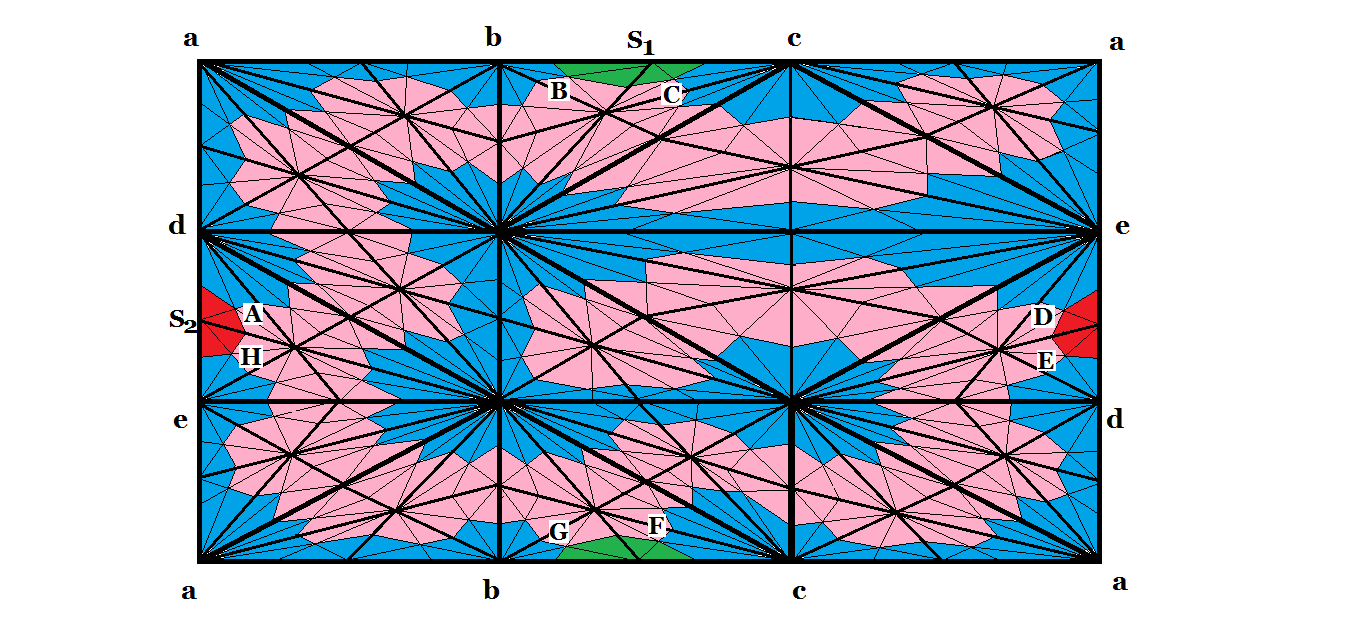}
\end{center}
\begin{center}
Figure 5.6 
\end{center}

(2) Let $M$ be the Klein bottle obtained by identifying the edges of the square $\I \times \I$, as discussed earlier. Consider  a triangulation of $M$ given in Figure 5.6. We shall express $M$  as a union of disjoint disk as discussed above. For this we take the second barycentric subdivision of the triangulation of $M$ and follow the procedure described above.
\begin{center}
\includegraphics[width=1\columnwidth]{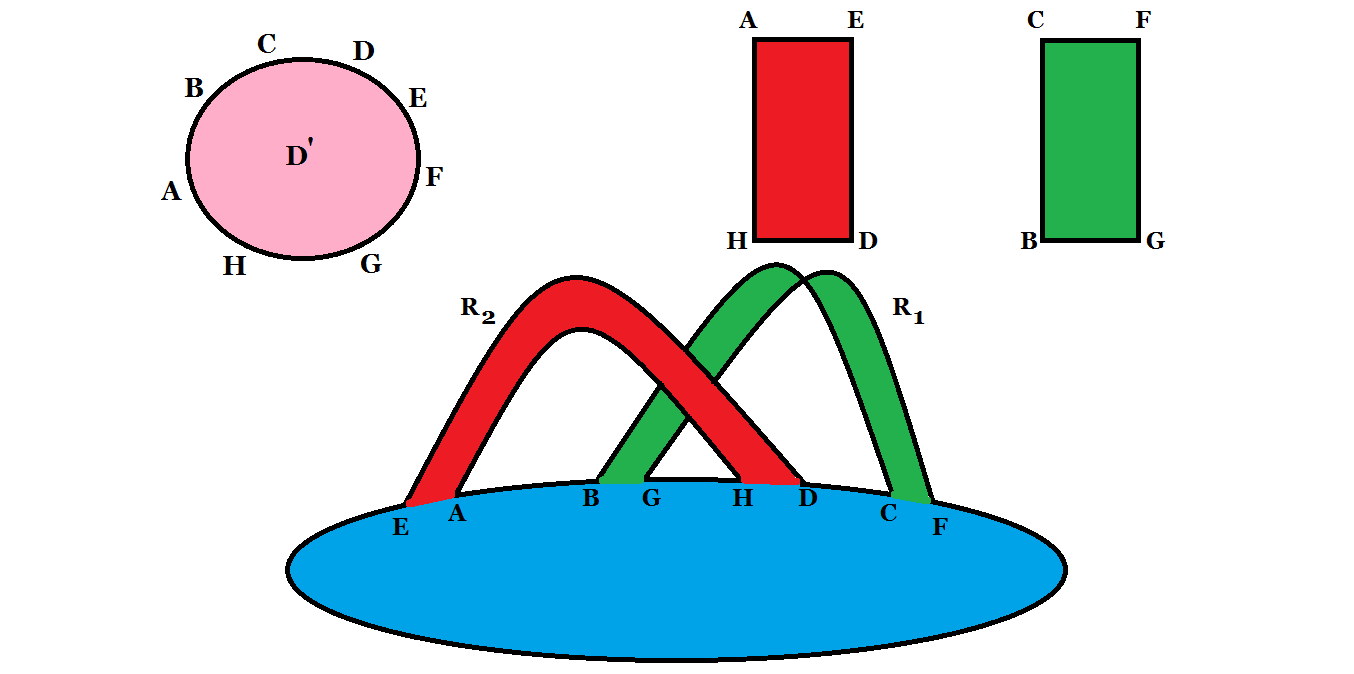}
\end{center}
\begin{center}
Figure 5.7
\end{center}

Now we simplify the above representation homeomorphically to show that $M$ is homeomorphic to $P^2  \#   P^2 $. Slide the strip $R_2$ along the boundary of $D\cup R_1$. Then $R_2$ will be twisted as shown in Figure 5.8. Now consider a simple closed curve OPQ and a corresponding curve OP$'$Q in $D'$ as shown in Figure 5.9. Then $D \cup R_1 \cup R_2 $ reduces to 2 disjoint  M$\ddot{o}$bius  bands and $D'$ is a union of 2 disks $D_1$  and $D_2$. Identifying  the corresponding segments of each M$\ddot{o}$bius band with the segments of $D_1$ and $D_2$  we get  two disjoint M$\ddot{o}$bius  band immersed in $\R^3$ with self intersection and a boundary circle OPQP$'$O. Identifying the boundary circles in two Mobius band we can conclude that a Klein bottle is homeomorphic to the connected sum of two projective planes.

\begin{center}
\includegraphics[width=1\columnwidth]{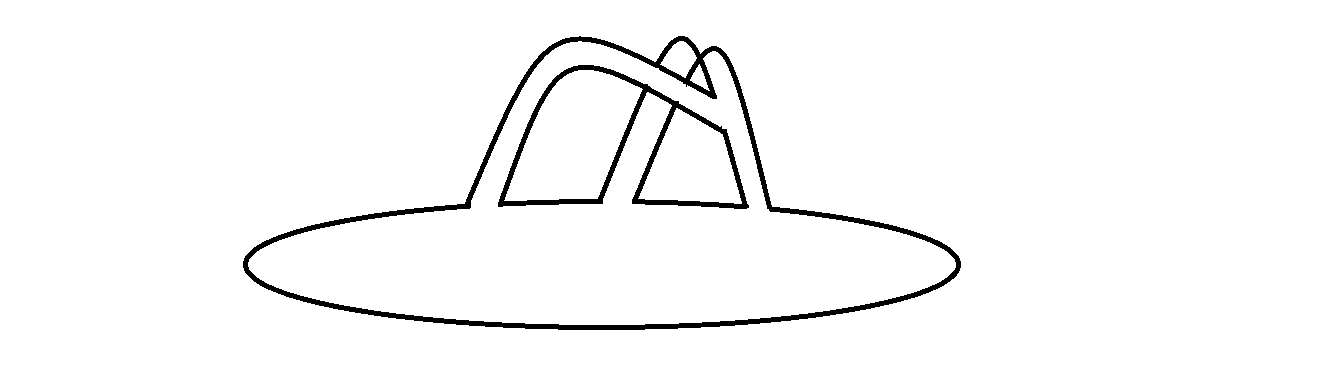}
\includegraphics[width=1\columnwidth]{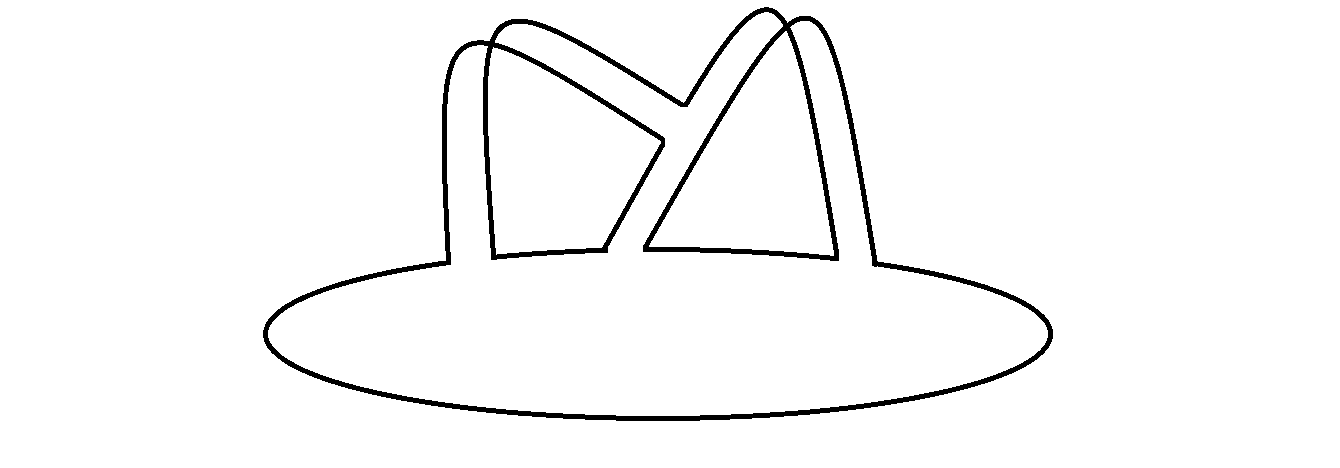}
\end{center}
\begin{center}
\includegraphics[width=1\columnwidth]{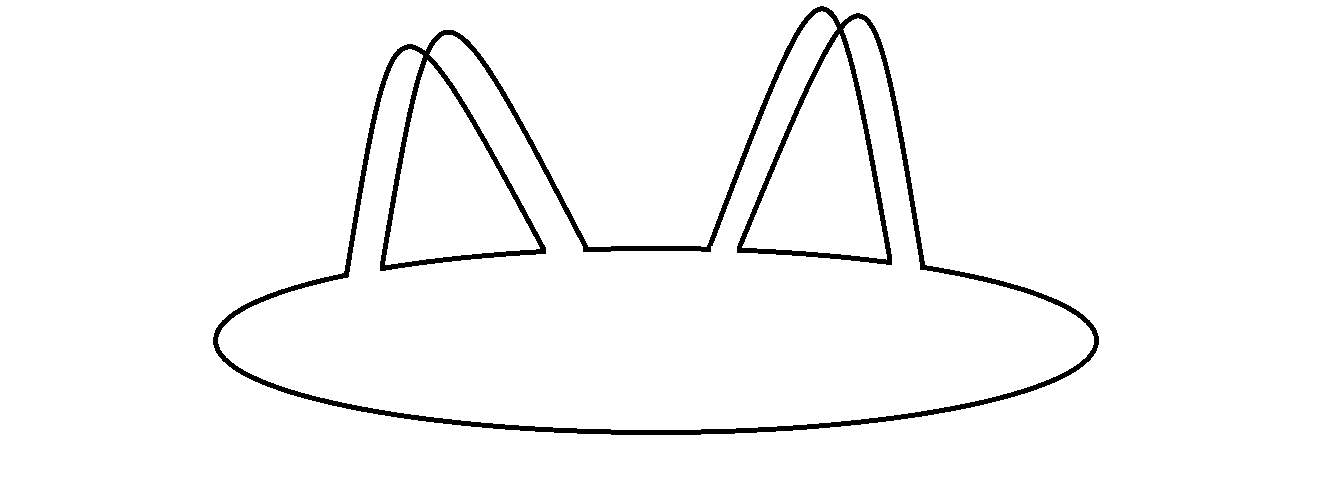}
\end{center}
\begin{center}
Figure 5.8
\end{center}

\begin{center}
\includegraphics[width=1\columnwidth]{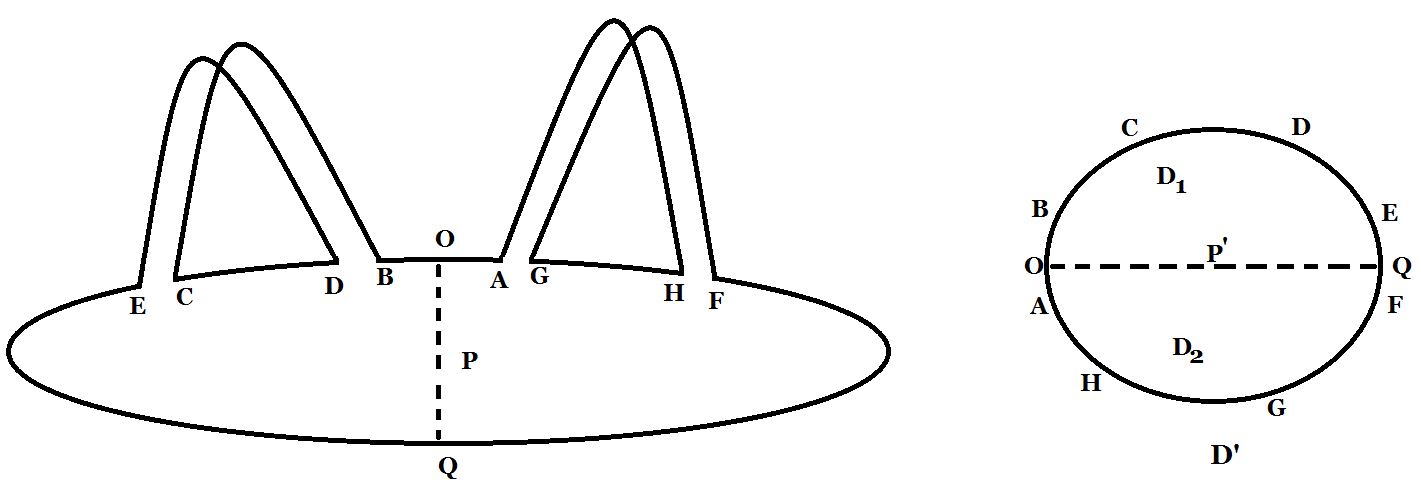}
\end{center}

\begin{center}
Figure 5.9
\end{center}

We shall now simplify the representation of $M$ in Figure 5.4 with the help of following operations:

(1) \textbf{Operation $ \alpha$:} Suppose that $R_i$  and $R_j$ are intertwined annular strips so that $D \cup S_i \cup S_j$ is a handle and boundary of $D \cup S_i \cup S_j$ is a circle.
\begin{center}
\includegraphics[width=1.5\columnwidth]{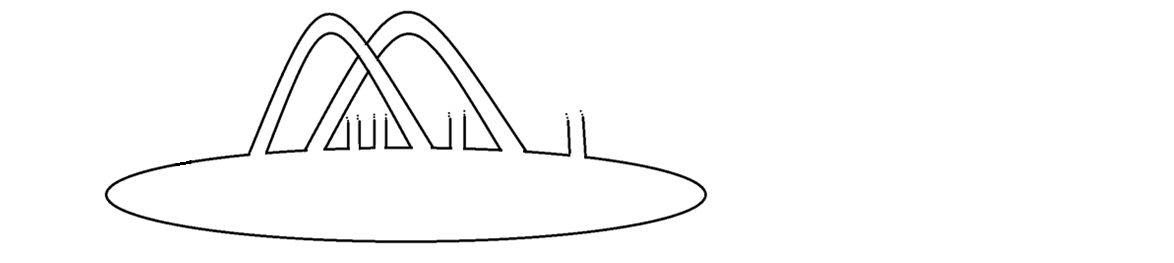}
\end{center}

\begin{center}
Figure 5.10
\end{center}

As in the above example, we can replace the old representation by the new homeomorphically, by sliding the strips $R_d,(d \ne i,j)$ along the boundary so as to get a situation in which $ \widetilde{D}\cap (R_i \cup Rj)$ lies in a segment in $\partial \widetilde{D}$ which intersect no $R_d,(d \ne i,j)$. We do this for each such handle. If we have no twisted strips, that is $M$ is orientable then the new representation will be  as in the Figure 5.11 below. Now cutting along the dotted curve and identifying the corresponding segments of $D'$,  we have the following: 
\begin{center}
\includegraphics[width=1\columnwidth]{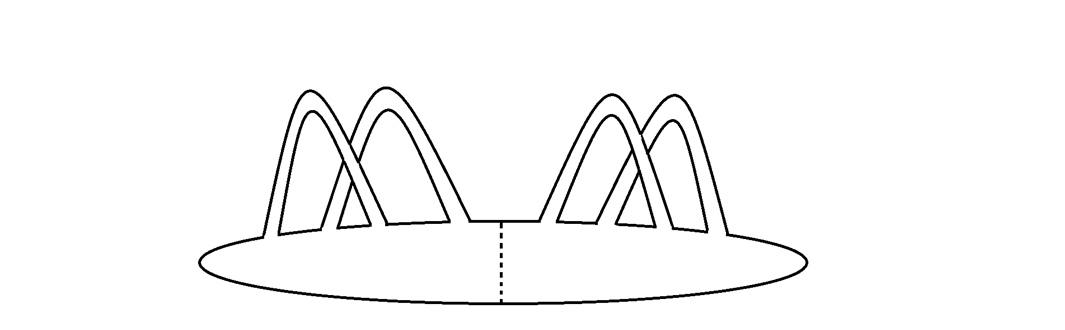}
\end{center}

\begin{center}
Figure 5.11
\end{center}

\begin{theorem}
A compact orientable surface without boundary is homeomorphic to $\Sigma_g$ for some $g \ge0$.
\end{theorem}
Otherwise if $M$ happens to be nonorientable then consider 

(2) \textbf{Operation  $\beta$:} Suppose $R_i$  and  $R_j $ are two intertwined strips, one annular say $R_i$ and another twisted say $R_j$. In that case $D \cup R_i \cup R_j$ is a Klein bottle.
\begin{center}
\includegraphics[width=1.5\columnwidth]{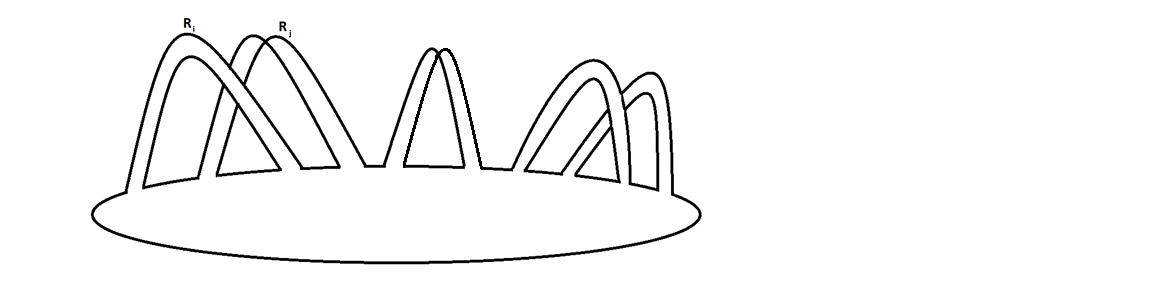}
\end{center}

\begin{center}
Figure 5.12
\end{center}

Then proceed as in above example, replacing the old representation by new homeomorphically by sliding that strip $R_i$ along the boundary of $\widetilde{D} \cup R_j$. Then $R_i$ will be no longer annular and we get a new twisted strip. We replace each such pair of intertwined strips and the new representation will have more twisted strips and fewer annular strips. The new representation will be  as in the Figure 5.13 below.
\begin{center}
\includegraphics[width=1.5\columnwidth]{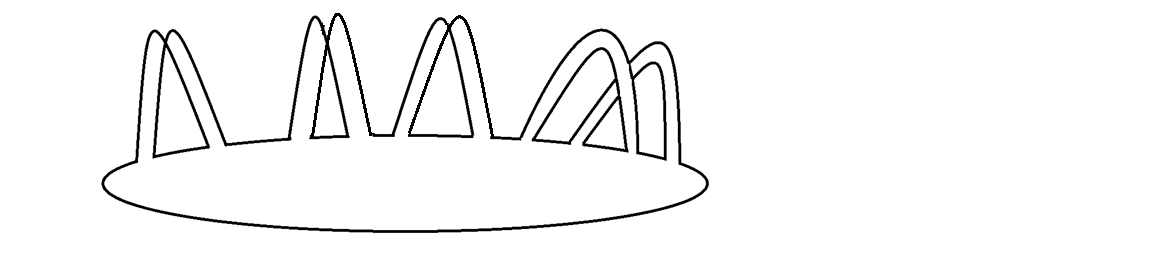}
\end{center}

\begin{center}
Figure 5.13
\end{center}

Consider a pair of intertwined annular strips $R_1$  and $R_2$ and a twisted strip $R_3$. Since $\widetilde{D}$ with a pair of intertwined annular strip is homeomorphic to a torus with a disk removed. So where the twisted strip $R_3$  is attached to the boundary component doesn't matter, we can slide the twisted strip along the boundary and make it intertwined with an oriented strip. Then we have the following representation.
\begin{center}
\includegraphics[width=1.5\columnwidth]{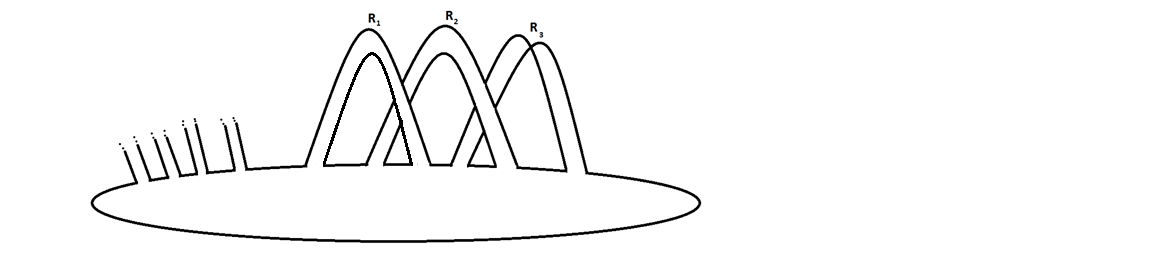}
\end{center}

\begin{center}
Figure 5.14
\end{center}

By an operation of type $\beta$ we slide oriented strips along the boundary of $\widetilde{D} \cup R_3$ one by one so that $R_1$ and  $R_2$ will be twisted and the representation reduces as in Figure 5.15. Thus we have the following 
\begin{theorem}\label{ch5,sec1,th2}
 The connected sum of a torus and a projective plane is homeomorphic to the connected sum of three projective planes.
\end{theorem}

\begin{center}
\includegraphics[width=1.5\columnwidth]{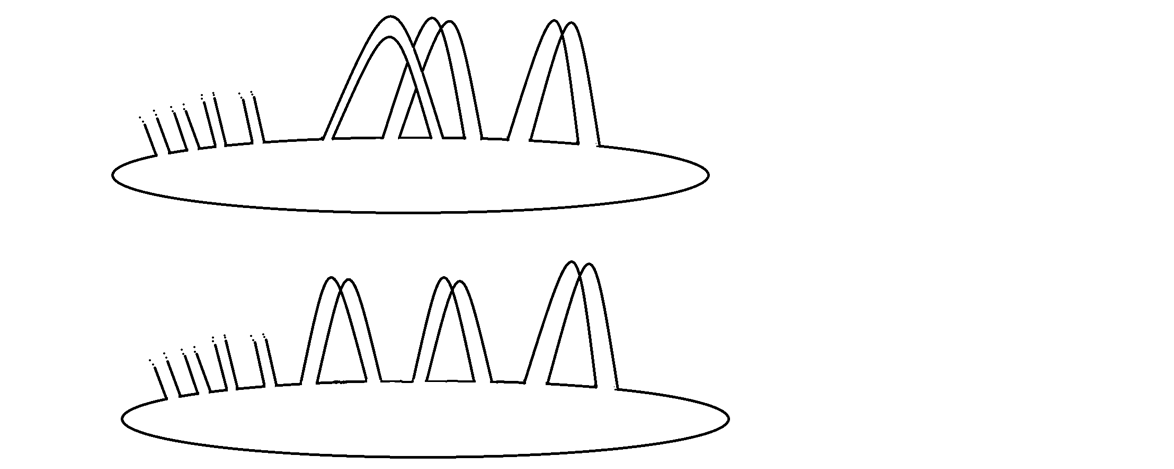}
\end{center}

\begin{center}
Figure 5.15
\end{center}

At the final stage we have the following representation 

\begin{center}
\includegraphics[width=1.5\columnwidth]{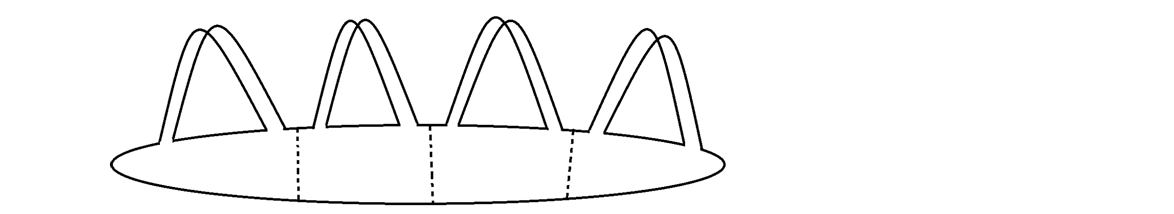}
\end{center}

\begin{center}
Figure 5.16
\end{center}

Again cutting along the dotted curve and identifying the corresponding segments of $D'$, we have the following 
\begin{theorem}
A compact nonorientable surface without boundary is homeomorphic to $U_g$ for some $g \ge 1$.
\end{theorem}

 Also any compact surface with boundary can be represented by attaching as many disjoint oriented strips on the boundary segments of $\widetilde{D}$  as there are boundary components see Figure 5.17 and Figure 5.18 below.
\begin{center}
\includegraphics[width=1\columnwidth]{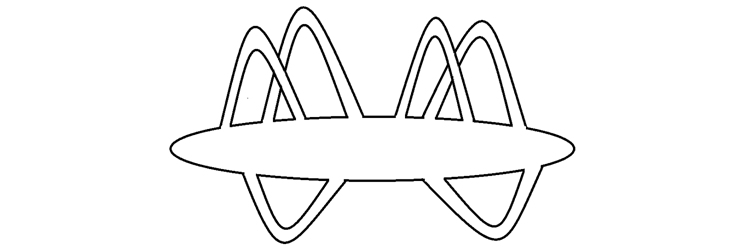}
\end{center}

\begin{center}
Figure 5.17
\end{center}

\begin{center}
\includegraphics[width=1\columnwidth]{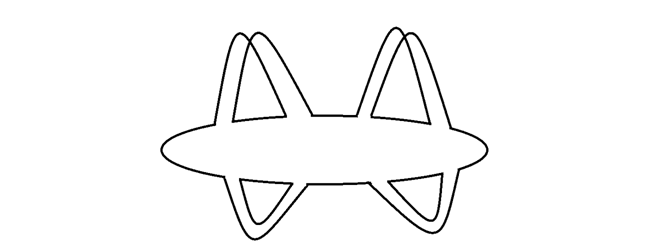}
\end{center}

\begin{center}
Figure 5.18
\end{center}

\begin{theorem}
 $\chi(\Sigma _{g,k})=2-2g-k $ and $\chi(U_{g,k})=2-g-k$.
\end{theorem}

\begin{proof}
We know that $\chi(\Sigma_0)=2$. Also $\chi(\Sigma_{g,k})$ is the surface homeomorphic to connected sum of $g$ copies of torus and $k$ boundary components. And corresponding to each torus in the above representation, we have a pair of intertwined annular strips and for each boundary component we have an annular strip attached to it. Thus from Equation \ref{ch5,sec0,eq1} we have
\begin{equation}\begin{split}\notag
\chi(\Sigma_{g,k})&=2-  {\rm{no. \ of \ strips \ attached}}\\
&= 2-2g-k.
\end{split}
\end{equation}
Similarly we have $\chi(U_{g,k})=2-g-k .$
\end{proof}

\section{The Fundamental Group of Compact Surfaces}

With the help of above representation we can compute the fundamental group of compact surfaces with or without boundaries. We define an open cell decomposition of a surface $M$. An $n$-$cell$ is a space homeomorphic to an $n$-simplex. We choose a point $p$ in the interior of $D$ and define a collection $ \{ C_i \} $ of 1-cell for each strip such that each of them initiates from $p$ through the corresponding strip and then terminates at $p$, and each of them has a single point $p$ in common. Finally a 2-cell is given by $ M \setminus \{  \cup C_i \} $. 

In case, if $M$ is a bordered surface with $k$ boundary components we define a cell decomposition of $M$ as follows. Consider the 0-cell and 1-cells defined above. In addition, we take vertices $\{p_i\}_{i=1,k}$ on each  boundary component  and insert a collection $\{H_i\}_{i=1,k}$  of 1-cell corresponding to each  boundary component that join $p_i$ to $p$. And a collection $\{B_i\}_{i=1,k}$ of 1-cells that initiates and terminates at $p_i$ as in Figure 5.19 and in Figure 5.20 below. Also a 2-cell is given by $M \setminus [\cup C_i \bigcup \cup B_i \bigcup \cup H_i]$.
\begin{center}
\includegraphics[width=1\columnwidth]{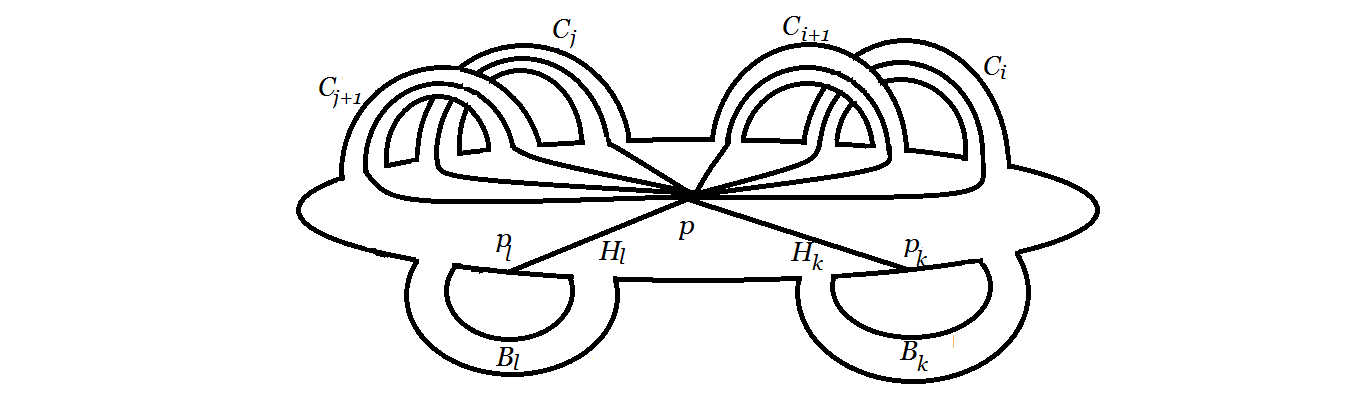}
\end{center}
\begin{center}
Figure 5.19
\end{center}

\begin{center}
\includegraphics[width=1\columnwidth]{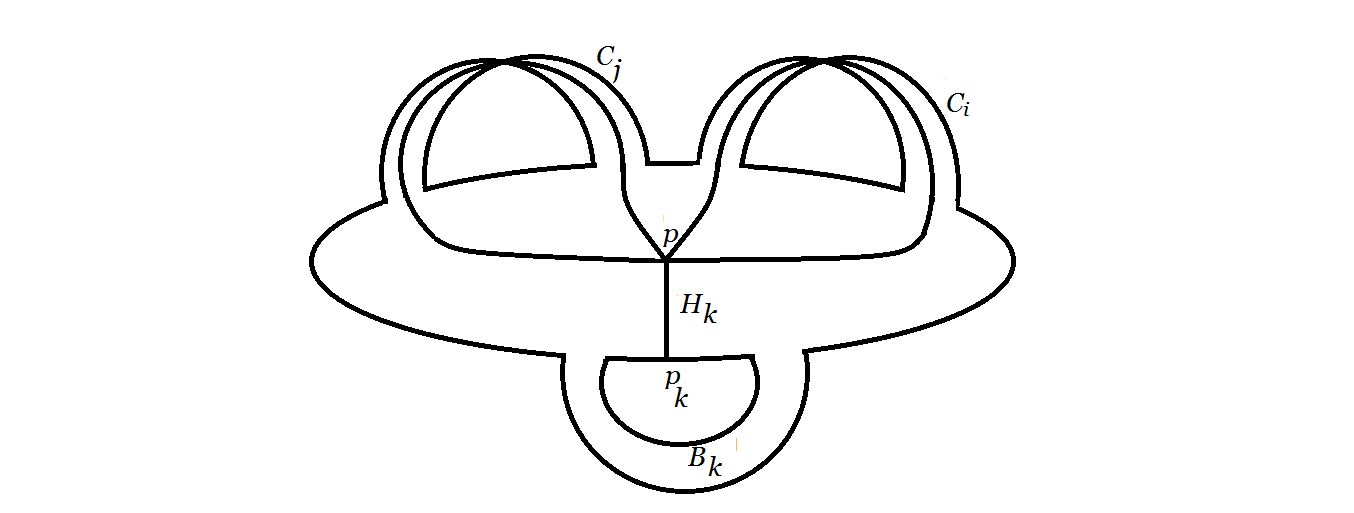}
\end{center}
\begin{center}
Figure 5.20
\end{center}

With the help of the following theorem  and the above cell decomposition of $M$  we shall determine $\pi_1(M)$, the fundamental group of $M$.
\begin{theorem}\label{ch5,sec2,th1}
$($see ~\cite{singh}.$)$  Let M be the space obtained by attaching a disk $\D ^2$ to a path connected Hausdorff space X by a continuous map			 $ f \colon \s ^1 \to X$. Let $j \colon X \hookrightarrow M $ be the canonical embedding. If $z_0 \in \s ^1\  and \ x_0 = f(z_0)$. Then $\pi_1(M,j(x_0))$ is isomorphic to the quotient group $\pi_1(X,x_0)/ Ker(j_ \#)$ and $ Ker (j_\#)$ is the smallest normal subgroup containing the image of $f_\# \colon \pi_1(\s^1,z_0) \to \pi_1(X,x_0)$.
\end{theorem}

Since $\Sigma_0$ is simply connected $\pi_1(\Sigma_0)$ is trivial. Now if $M= \Sigma_{g,k}$, a compact orientable surface with genus $g$ and $k$ boundary components. This has a cell decomposition with $(1+k)$ 0-cells, $(2g+2k)$ 1-cells and one 2-cell. As in above theorem take $X$ to be a wedge of $2g$ circles and $k$ copies of space $Q$ which is a one point union  of a circle and a line segment see Figure 5.21 below. And for attaching a disk $\D^2$ to $X$, we shall ignore an open disk $B$ from the interior of $\D^2$  and attach the remaining annulus $A=\D^2 \setminus B$ onto $X$ via the map $f \colon \partial \D^2 \to X $ so that $X\cup A$ can be embedded in $\R^3$ and M can be realize by adding $B$ onto the boundary of $X \cup A$.\\

\begin{center}
\includegraphics[width=1\columnwidth]{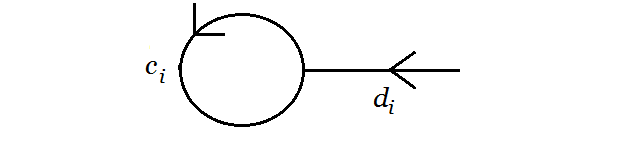}
\end{center}
\begin{center}
Figure 5.21
\end{center}

Since $\pi_1(\s^1) = \Z$, let $\{ a_i,b_i \}$ denotes the generator for each g pairs of  circles. Also $\pi_1 (Q)= \{c_i,d_i \vert d_i \} $, let $\{c_i ,d_i\}$ denotes the generator for each $k$ copies of $Q$ space  in $X$ with the specified orientation as indicated in the Figure 5.22 below.

\begin{center}
\includegraphics[width=1\columnwidth]{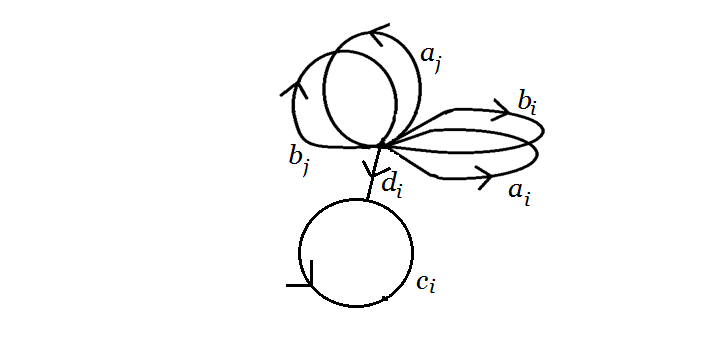}
\end{center}

\begin{center}
Figure 5.22
\end{center}

Then $A$ is attached onto $X$ mapping $\s^1$ along $X$ as described in the Figure 5.23 below. Now if $\alpha$ represents a generator of $\pi_1(\s^1)$ then
$$f_{\#}\circ \alpha = a_1b_1a_1^{-1}b_1^{-1}\ldots a_gb_ga_g^{-1}b_g^{-1}d_1c_1d_1^{-1}\ldots d_kc_kd_k^{-1} .$$

Thus by Theorem \ref{ch5,sec2,th1} $\pi_1(M)$ is the quotient of the free group on the $2g+k$ generators $a_1,b_1,\ldots ,a_g,b_g,c_1,\ldots ,c_k$ by the normal subgroup generated by  the element 
$ a_1b_1a_1^{-1}b_1^{-1}\ \ldots $ $a_g b_g a_g^{-1}b_g^{-1}d_1c_1d_1^{-1}\ldots d_kc_k d_k^{-1}$,  that is,

\begin{center}
\includegraphics[width=1\columnwidth]{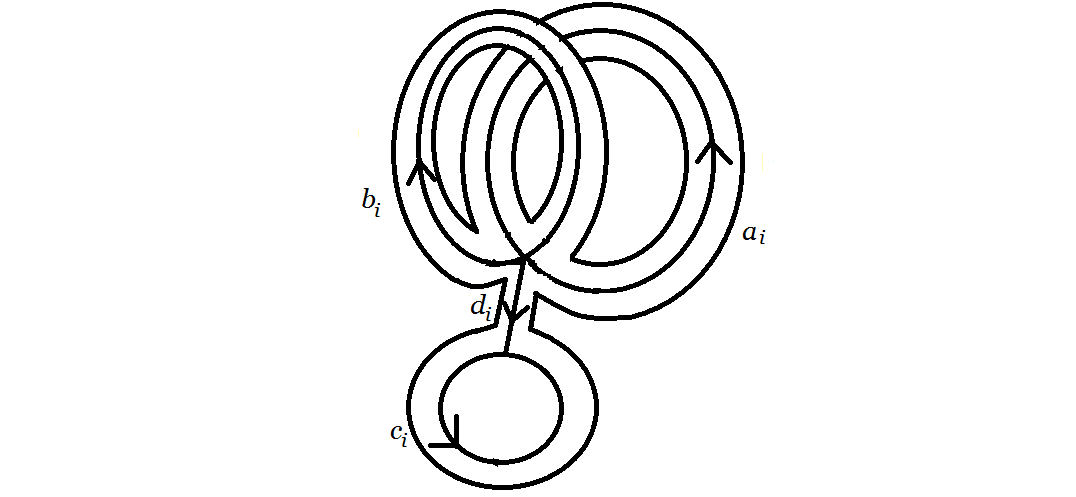}
\end{center}

\begin{center}
Figure 5.23
\end{center}

\begin{equation}\notag
\begin{split}
\pi_1(M)= &<a_1,b_1,\ldots ,a_g,b_g,c_1,d_1,\ldots ,c_k,d_k \  \vert \\
&\quad  a_1b_1a_1^{-1}b_1^{-1}\ldots a_gb_ga_g^{-1}b_g^{-1}d_1c_1d_1^{-1}\ldots d_kc_kd_k^{-1},d_1,d_2,\ldots ,d_k>
\end{split}
\end{equation}
or
$$ \pi_1(M)=\ <a_1,b_1,\ldots ,a_g,b_g,c_1,\ldots ,c_k \ | \  a_1b_1a_1^{-1}b_1^{-1}\ldots a_gb_ga_g^{-1}b_g^{-1}c_1\ldots c_k>$$

Thus $\pi_1(M)$ is a group with $2g+k-1$ generators. 

Similarly if $M=U_{g,k}$, a compact nonorientable surface with genus $g$ and $k$ boundary components. This has a cell decomposition with $(1+k)$ 0-cells, $(g+2k)$ 1-cells and one 2-cell. As done above taking $X$ to be wedge of $g$ circles and $k$ copies of $Q$ space, we attach the annulus $A= \D^2 \setminus B$ onto $X$ so that $X\cup A$ can be embedded in $\R^3$. 

Let $a_i$ denotes the generator for each $g$ circles and $\{c_i,d_i\}$ denotes generator for each $k$ copies of $Q$ space in $X$ with specified orientation as in the Figure 5.24 below.

\begin{center}
\includegraphics[width=1\columnwidth]{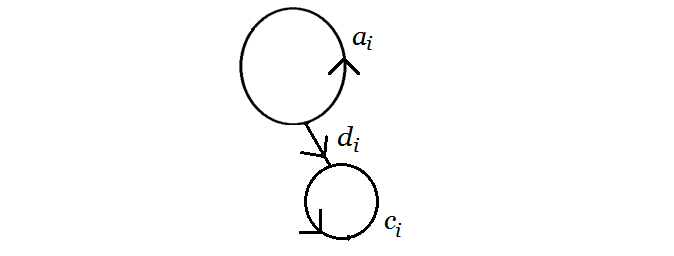}
\end{center}

\begin{center}
Figure 5.24
\end{center}

We attach $A$ onto $X$ as indicated in the Figure 5.25. Again if $\alpha$  represents a generator of $\pi_1(\s^1)$ then
$$f_{\#}\circ \alpha = a_1^2a_2^2\ldots a_g^2d_1c_1d_1^{-1}\ldots d_kc_kd_k^{-1}.$$

Thus by Theorem \ref{ch5,sec2,th1} $\pi_1(M)$ is the quotient of the free group on the $g+k$ generators $a_1,\ldots ,a_g,c_1,\ldots ,c_k$ by the least normal subgroup generated by the element $ a_1^2a_2^2...a_g^2d_1c_1d_1^{-1}$ $\ldots d_kc_kd_k^{-1}$, that is,
\begin{center}
\includegraphics[width=1\columnwidth]{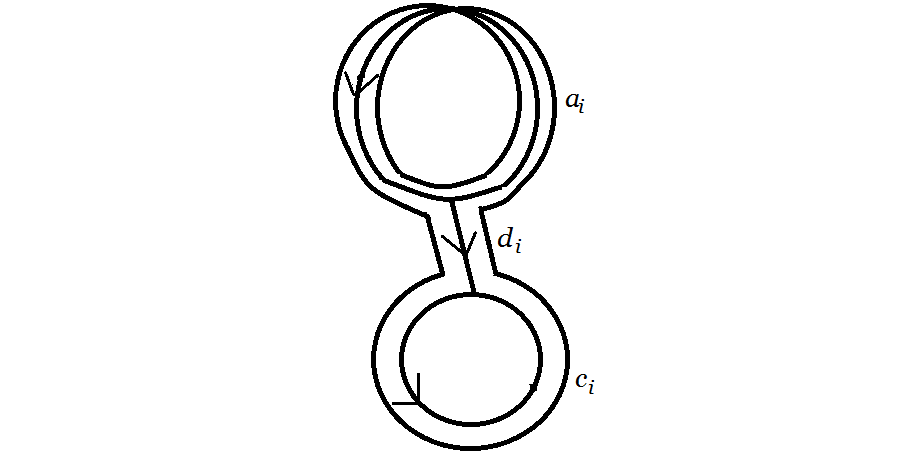}
\end{center}

\begin{center}
Figure 5.25
\end{center}

\begin{equation}\notag
\begin{split}
\pi_1(M)= &<a_1,\ldots ,a_g,c_1,d_1,\ldots ,c_k,d_k \ | \\
& \quad  a_1^2a_2^2\ldots a_g^2d_1c_1d_1^{-1}\ldots d_kc_kd_k^{-1},d_1,d_2,\ldots ,d_k>
\end{split}
\end{equation}
or
$$\pi_1(M)=\ <a_1, \ldots ,a_g,c_1, \ldots ,c_k \ | \   a_1^2a_2^2 \ldots a_g^2c_1 \ldots c_k> \quad$$

Thus $\pi_1(M)$ is a group with $g+k-1$  generators.

Having determined the fundamental group of a compact surface we can write its first homology group by making the fundamental group abelian and writing it additively using  Theorem \ref{ch3,sec6,th3}. 

If $M=\Sigma_{g.k}$ the homology group is thus a  abelian group with $2g+k-1$ generators, the generators will now satisfy the relation $c_1+ \ldots +c_k=0$. If $k=0$ the homology group is a free abelian group with $2g$ generators and is  isomorphic to $\Z^{2g}$.

In case $M=U_{g,k}$ the homology group is a abelian group with $g+k-1$ generators, that satisfy the relation $2a_1+ \ldots +2a_g+c_1+ \ldots +c_k$. If $k=0$ then the relation $2(a_1+ \ldots +a_g)=0$ shows that $M$ has a torsion element of order 2, since there is no other  relation the homology group is isomorphic to $\Z_2 \oplus \Z \oplus \ldots \oplus \Z \  (g-$1 copies of $\Z$) or $\Z_2 \oplus\Z^{g-1}$. 

First note that if $M=U_g$, then its first homology group contains a torsion element of order two, whereas the first homology group of $\Sigma_g$ is torsion free. Therefore, no $U_g$ can be homeomorphic to a $\Sigma_g$. Thus if a surface admits an oriented presentation then it is homeomorphic to a sphere or a connected sum of tori.

Next if  $M=\Sigma_g$, then its first homology group is $\Z^{2g}$. Thus if $\Sigma_g$ is homeomorphic to $\Sigma_k$ then
$$2g=2k$$ 
that is $$  \quad g=k .\ $$

Similarly, if $U_g$ is homeomorphic to $U_k$ then again we have $g=k$.

Thus orientability, genus and Euler characteristic of a compact surface is topological invariant. Since the number of connected components of the boundary is topological invariant, we have the following
\begin{theorem}
Let M be a compact surface witk k, (k $ \ge  0$)  boundary components. Then M is homeomorphic to precisely one of the following:\\
$(1)\ \Sigma_{g,k} \	 \	  g,k \ge 0$  \\
$(2)\ U_{g,k} \	\	g\ge 1 ,k\ge 0 $
\end{theorem}

.

\chapter{Classification of  Noncompact 2-Manifolds}\label{ch6}

\section{Introduction and Basic Definitions}

In this chapter we shall discuss a topological classification of noncompact surfaces. We begin by defining the ideal boudary of a surface.

\textbf{Definition:}
Let $S$ be a noncompact surface then $S$ can be written as  an increasing sequence of its compact subspaces
$$C_1 \subset C_2 \subset \ldots \subset S$$
with $C_i \subset int(C_{i+1})$. Each connected component $Q$ of $C_n^c$, complement of $C_n$ in $S$,  is contained in some component of complement of $C_{n-1}$ in $S$. A decreasing sequence $q=Q_1\supset Q_2 \supset Q_3 \supset \ldots $ of connected components of $C_1^c\supset C_2^c\supset C_3^c\supset  \ldots $ respectively, whose closure  in $S$ is not compact but has a compact boundary,  is called an end of $S$.

The above definition is independent of the choice of sequence of compact sets $C_n$. Infact if there is  another sequence $\{ D_k \}$ of compact subsets then there is a unique sequence $p=P_1\supset P_2 \supset P_3 \supset \ldots $ with respect to $D_n$ such that for any  $p$  there exists a $q$  such that $P_q \subset Q_p$ and vice versa. And we say that the two ends are equivalent.  Let $q^*$ denotes the equivalence class of end containing $q$.

Another equivalent definition is given by Ahlfors as follows

\textbf{Definition:}
An end  is a non empty collection $q$ of non empty regions $Q$, connected open sets, in $S$, whose closure in $S$ is not compact but has a compact boundary, which satisfies the following conditions\\
(1) If $Q_0 \in q$ and $Q \supset Q_0$ , then $Q \in q$. \\
(2) If $Q_1,Q_2 \in q$, there exists  a $ Q_3 \subset Q_1 \cap Q_2$ such that $Q_3 \in q$. \\
(3) The intersection of all closures $\overline{Q}, \  Q \in q$ is empty.

It is evident from the above definitions that for any compact subset $A$ of $S$, there exists  some $Q_i$ such that $Q_i\cap A = \phi$. In other words, the sequence of regions does not have any common point.

\textbf{Definition:}
The  ideal boundary $B(S)$ of a surface $S$ is the topological space having the equivalence class of ends of $S$ as elements.

 For any region $Q \subset S$ whose boundary in $S$ is compact, we define $B(Q)$ to be the set of all  equivalence classes  of ends $q^* \in B(S)$ such that $Q_n \subset Q$ for some $Q_n \in q$. Also if $Q$  is contained in a compact set of $S$ then $B(Q)=\phi$. For $M=S \cup B(S)$, a basis for the topology of $M$ consists of all open sets of $S$ and all sets of the form $Q \cup B(Q)$.

\textbf{Example:}
Examples of open surfaces include: the Euclidean plane, a sphere punctured at a point, a sphere with infinite handles attached, the Euclidean plane with one cross cap, etc.

We can construct every open surface from the following five bordered surfaces, three sphere with one, two and three holes and a torus and projective plane, each with two holes.
\begin{center}
\includegraphics[width=2\columnwidth]{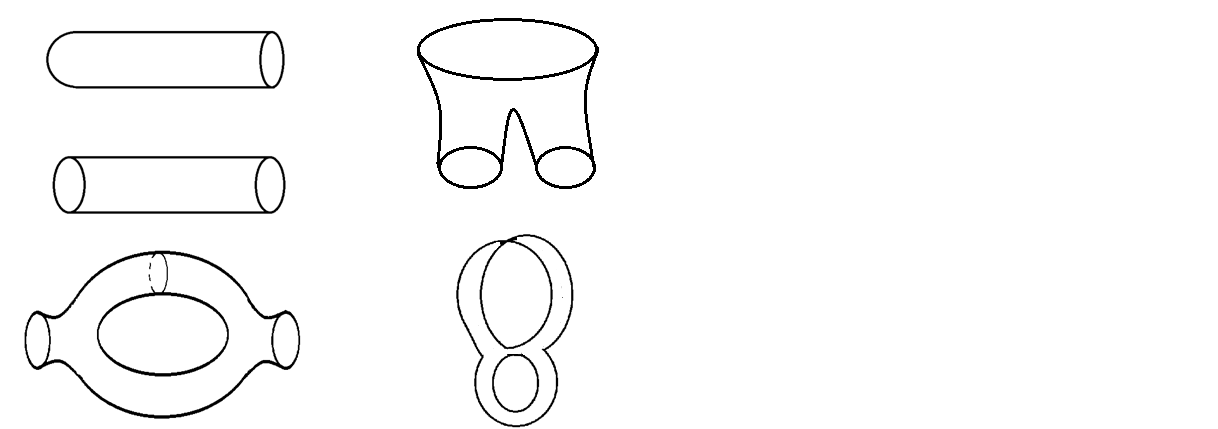}
\end{center}

\begin{center}
Figure 6.1 
\end{center}

\begin{center}
\includegraphics[width=1\columnwidth]{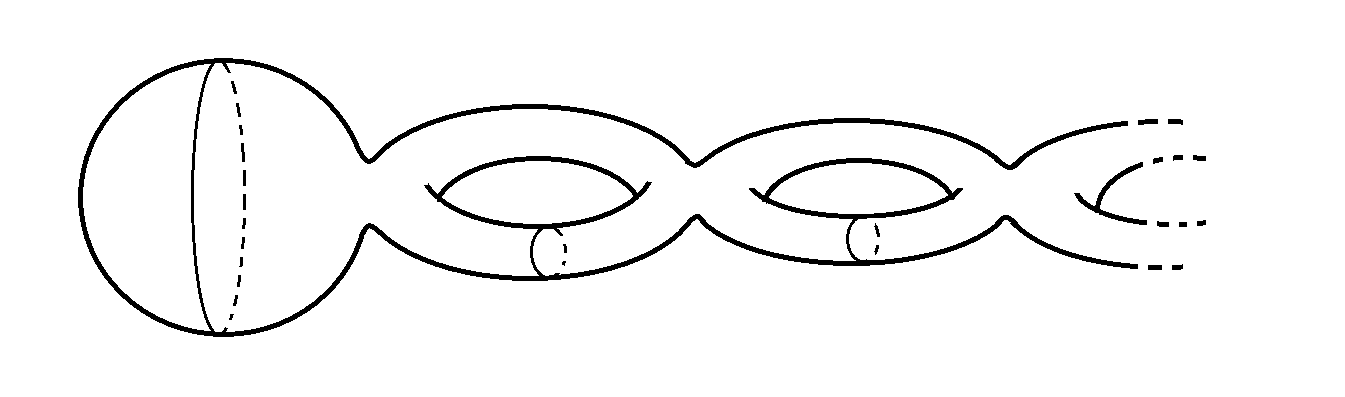}
\end{center}

\begin{center}
Figure 6.2 
\end{center}

\textbf{Definition:}
A $q^* \in B(S)$ represented by $q$ is called planar if the sets $Q_n \in q$ are planar for all except finitely many $n$. And $q^*$ is called orientable if the sets $Q_n \in q$ are orientable for all  except finitely many $n$.

\textbf{Definition:}
A surface $S$ is of infinite genus if there is no compact subset $A$ of $S$ such that $S\setminus A$ is of genus zero. In contrary, if $S\setminus A$ is of genus zero, then the genus of $M$ is defined to be the genus of $A$.

\textbf{Definition:}
We define four orientability classes of a surface 
$S$ as follows:\\
(1) Orientable: If $S$ is orientable.\\
(2) Infinitely nonorientable: If there is no compact subset $A$ of $S$ such that $S\setminus A$ is orientable.\\
(3) Odd nonorientability: If $ S \setminus A$ is orientable for  some $A \subset S$, and the number of cross-caps in any compact subsurface containing $A$ is always odd.\\
(4) Even nonorientability: If $S \setminus A$ is orientable for  some $A \subset S$, and the number of cross-caps in any compact subsurface containing $A$ is always even.\\

\section{Some Preliminary Lemmas}

\textbf{Definition:}
A compact bordered subsurface $S'$ of $S$ is called a canonical subsurface if it has the following properties:\\
(1) The closure of each component $U$ of $S \setminus S'$ is noncompact and meets $S'$ in exactly one simple closed curve $ \partial(U)$. \\
(2) Every component of $S\setminus S'$ is either planar or of infinite genus and is either orientable or infinitely nonorientable.

From (2) it follows that, if $S$ has finite genus, then $S'$ has the same genus and is orientable if and only if $S$ is. Whenever, $S$ is of odd or even nonorientability class then $S'$ will contain a projective plane or a Klein bottle respectively, with boundary.

\begin{theorem}
Every open surface has a canonical exhaustion that is a collection $S_1,S_2,S_3, \ldots $ of canonical subsurfaces of S such that $S_1 \subset S_2 \subset \ldots  \subset S$, each contained in the interior of the one following it.
\end{theorem}

We shall construct a sequence of compact subsurfaces of $S$ satisfying the condition (1) above. And condition (2) can be satisfied by arbitrary $S_n$, because if a component $U$ of $S \setminus S_n$ were of finite but non zero genus and/or finitely nonorientable we could add some compact portion of $U$ to $S_n$.

\begin{lemma}\label{ch6,sec2,lm2}
Let $S$ be an open surface and $A$ be a finite connected subcomplex of a triangulation $K$ of $S$. Then there is a compact bordered subsurface of $S$ containing $|A|$.
\end{lemma}

\begin{proof}
Let $P$ be the union of the $b(St(v))$, barycentric star of $v$,  of the vertices of $A$. Then $P$ is compact connected subcomplex of  $bK$ containing $bA$. Thus the polyhedron $|P|$ is the desired subsurface of $S$.
\end{proof}

\begin{lemma}\label{ch6,sec2,lm3}
Let S and A be as above then there exist a compact bordered subsurface L of S satisfying the condition (1) in the above definition.
\end{lemma}

\begin{proof}
For every $n$ consider the set $Q$ of all polyhedrons $|P_{\al}|$ such that each $P_{\al}$ is a subcomplex of $b^nK$ and contains $b^nA$. By above lemma $Q$ is non empty. Take $|P| \in Q$ be the polyhedron with the minimum  number of boundary components. Let $P$  be the subcomplex of $b^iK$ containing $b^iA$ for some $i$. Also all the components of $S\setminus |P|$ are noncompact otherwise a  compact component could be added to $|P|$ reducing the number of boundary components.

Now if two boundary components $B_1  $ and $ B_2$ of $P$ belong to the same complementary component $U$ then we can  choose vertices $b_1$ and $b_2$ on $B_1$ and $B_2$ respectively and join them by a simple polygon $\sigma$, where $\sigma \setminus \{b_1,b_2 \}\subset U $. Consider $L=B_1 \cup \sigma \cup B_2$ as a 1-dimensional subcomplex of $b^iK$.

As  in Lemma \ref{ch6,sec2,lm2} taking $bP \cup bL$ as a subcomplex of $b^{i+1}K$, there is a subsurface of $S$. Let $P'$ be the corresponding subcomplex of $b^{i+2}K$. Writing $K$ for $b^iK$ and A for $P\cup L$ we can see that $ P'=N(A)$, second barycentric neighborhood of $A$ in $K$. Then the boundary of $N(L)$ in $U$  has a single component.

Thus $P'$ has fewer boundary components than $P$ which is a contradiction as $P$ has the minimal number of boundary components. Hence $|P|$ satisfies condition (1).
\end{proof}

Proof of the theorem

\begin{proof}
For the construction of sequence $\{S_i \}$ of canonical subsurfaces, we first find a sequence of triples $\{A_i,n_i,P_i \}$, where $A_i$ is a finite connected subcomplex of $K$, and $n_i\in \N$, satisfying the following properties:\\
(a) $|P_i|$ satisfies condition  (1)\\
(b) $b^{n_i}A_i \subset P_i$ \\
(c) $P_i \subset b^{n_i }A_{i+1}$ \\
(d) $|P_i|$ is contained in interior of $|P_{i+1}|$

We begin by choosing $A_1$ as a single vertex. Then by Lemma \ref{ch6,sec2,lm3}, given $A_i$ there exists a compact bordered subsurface $S_i$ of $S$ satisfying condition (1) with the polyhedron $|P_i|$ and $P_i$  contains $b^{n_i}A_i$. Also we define $A_{i+1}= \{ \sigma^2 \in K |  b^{n_i}\sigma^2 \cap P_i \ne \phi\}$. Thus $P_i \subset b^{n_i}A_{i+1}$. And from (b) and (c) we have (d).
\end{proof}

We now show that the space $M=S \cup B(S)$ is compact. It follows that  $B(S)$  is compact, being a closed subset of $M$.

\begin{theorem}
The space M is compact.
\end{theorem}
\begin{proof}
Suppose $M$ has an open covering by the sets $\{ U_{\alpha} \}$, and $\{ S_i \}$ be a canonical exhaustion of $S$. If $M$ were not compact there is a component $Q_1$ of $S_1^c$ which does not have a finite subcovering. Similarly we can find a component, $Q_2 \subset Q_1$, of $S_2^c$ with the same property. In this way we obtain a sequence $Q_1\supset Q_2\supset Q_3 \supset \ldots$ of components that have no finite subcovering. This defines an end $q$ (say).

But $q \in U_{\alpha}$ for some $\alpha$. This implies $q \in Q \cup B(Q) \subset U_{\alpha}$ (some basic neighborhood of $q$). By definition some $Q_n\in q$ is contained in $Q \subset U_{\alpha}$, a contradiction to construction of $Q_n$.
\end{proof}

\begin{lemma}\label{ch6,sec2,lm5}
Let U and V be subsets of a surface S whose boundaries in S are compact. Then $B(U\cup V)=B(U)\cup B(V)$ and $B(U \cap V)= B(U)\cap B(V)$.
\end{lemma}

\begin{proof}
Let $q^* \in B(U\cup V)$ then for some $Q_n \in q$, $Q_n \subset U\cup V$. Since boundary of $U$ is compact there is a  $Q_i \subset Q_n$ and $Q_i \in q$ which does not intersects with boundary of $U$. Thus $Q_i$ is contained in $U$ or $V$, and $q^* \in B(U) \cup B(V)$. Also if $q^* \in B(U)\cup B(V)$  then some $Q_n \in q$ is contained in $U$ or $V$, that is, $Q_n \subset U\cup V$, so that $q^*\in B(U\cup V)$.

Also $q^* \in B(U\cap V)$ if and only if some $Q_n \subset U\cap V$ for some $Q_n \in q$, that is, $Q_n \subset U$ and $Q_n \subset V$ this implies $q^* \in B(U)$ and $q^* \in B(V)$. Thus we have $q^* \in B(U \cap V)$.
\end{proof}

\begin{lemma}\label{ch6,sec2,lm6}
If $S'$ is a subsurface of an open surface S, Then B($S'$) contains non planar and/or nonorientable ends if and only if $S'$ is  of infinite genus and/or infinitely nonorientable.
\end{lemma}
\begin{proof}
Suppose $B(S')$ contains some non planar end $q^*$  represented by $q=Q_1\supset Q_2\supset Q_3\supset \ldots$ . Then $Q_n$ are non planar for all $n$. Then there does not exist any compact set $A\subset S'$ such that $S'\setminus A$ is of genus zero. Thus $S'$ is of infinite genus.

Similarly if $B(S')$ contains some nonorientable end $q^*$. Then $Q_n$ will be nonorientable for all $n$ and $S'$ will be infinitely nonorientable.
\end{proof}

\textbf{Definition:} Let $K$ be a triangulation of an orientable surface $S$, orient all its simplexes in a compatible manner. Suppose $|f| : |K| \to |K|$ be a simplicial homeomorphism and a simplex $\sg^2$ is  oriented by the ordering $<a_0,a_1,a_2>$, whose image $|f|(\sg^2)$ occurs with the orientation $<f(a_0),f(a_1),f(a_2)>$ induced by $f$, then the same hold for any other simplex of $K$, and we call $|f|$ orientation preserving. And for any homeomorphism $g : |K| \to |K|$, which need not be induced from a simplicial map, we call $g$ to be orientation preserving if $g$ is homotopic to some orientation preserving simplicial map.

\textbf{Definition:} Let $f,g : X \to Y$ be two homeomorphism. An isotopy between $f$ and $g$ is a homotopy $H : X \times \I \to Y$ such that for each $t\in \I$ the map $H_t : X \to Y$, defined by $H_t(x)=H(x,t)$  is a homeomorphism.

\begin{theorem}
{\rm{\textbf{(Isotopy Extension Theorem)} (see~\cite{kirby}, ~\cite{hirsch}, ~\cite{shastri}).}} Let $H: N\times \I \to M$ be an isotopy of a compact manifold $N$ into a manifold M. If either $H(N\times \I) \subset \partial M$ or $H(N\times \I) \subset  M \setminus \partial M$, then $H$ can be extended to an isotopy of M, that is there is an isotopy $G: M\times \I \to M$ such that $G_0=i_M$ and $H_t=G_t H_0$, and $G_t$ is the identity map outside a compact subset of $M$.
\end{theorem}

\begin{lemma}\label{ch6,sec2,lm7}
If $f: \s^1 \to \s^1$ is an orientation preserving homeomorphism then f is isotopic to the identity of $\s^1$
\end{lemma}

\begin{proof}
If $f$ has no fixed point then we can take a rotation map of $\s^1$, $g: \s^1 \to \s^1$ such that $g(f(1,0)) = (1,0)$.

Then $(g \circ f):\s^1 \to \s^1$ is an orientation preserving homeomorphism with a fixed point (1,0). Now $g \circ f$ can be lifted to a map $\hat{f}: [0,1] \to \R$ such that $\hat{f}(0) = 0$, making the diagram commute
$$\hat{f}$$
$$\I \	\longrightarrow \	\R $$
$$p \	\big\downarrow\	\	\qquad \	\big\downarrow \	 p$$
$$\s^1 \	\longrightarrow \	 \s^1$$
$$g \circ f$$
where $p: \R \to \s^1$ is the exponential map defined as $p(t)=(cos2\pi t,sin2\pi t)$.

Since $g \circ f$ is orientation preserving homeomorphism $\hat{f}$ is order preserving and we get $\hat{f}(1) = 1$. Then $\hat{f}:\I \to \I$ is a homeomorphism. Consider the function $H : \I\times \I \to \I$ defined as 
$$ H(x,t)= (1-t)\hat{f}((x) +tx .$$
This is an isotopy between $\hat{f}$ and $i_{\I}$ (identity of $\I$). It defines  an isotopy $G:\s^1 \times \I \to \s^1$ between $g \circ f$ and $i_{\s^1}$ (identity of $\s^1$) making the diagram commutative
$$\qquad H$$
$$\I\times \I \	\longrightarrow \	\I $$
$$p \times (i_{\I}) \	\big\downarrow\	\	\qquad \quad \	\big\downarrow \	 p \qquad$$
$$\s^1 \times \I\	\longrightarrow \	 \s^1$$
$$\qquad G$$
Now the map $G_1: \s^1 \times \I \to \s^1$ defined by $G_1(x,t)=g^{-1} \circ G(x,t)$ is an isotopy between $f$ and $g^{-1}$. Since $g$ is a rotation of $\s^1$ so is $g^{-1}$, that is, $g^{-1}(cos\theta ,sin\theta) = (cos(\theta +s), sin(\theta +s))$ for some $s$. Define $G_2: \s^1 \times \I \to \s^1$ as 
$$G_2[(cos\theta ,sin\theta),t] =[cos(\theta +st), sin(\theta +st)] .$$

Here, for each $t$, $G_2$ is a rotation of $\s^1$. Hence $G_2$ is an isotopy between $g^{-1}$ and $i_{\s^1}$. Since being isotopic to is an equivalence relation, we have, by transitivity, $f$ is isotopic to $i_{\s^1}$.
\end{proof}

\begin{lemma}\label{ch6,sec2,lm8}
Let S be compact bordered surface, C is a simple closed curve in $\partial S$ and $f: C\to C$ is any homeomorphism. Then $f$ can be extended to a homeomorphism of S onto itself so that every curve in $\partial S$ is invariant.
\end{lemma}

\begin{proof}

If $f: C\to C$ is an orientation preserving homeomorphism. Let $h: C \to \s^1$ be a homeomorphism. Then $h^{-1}: \s^1 \to S$ is an embedding.

Since $hfh^{-1}: \s^1 \to \s^1$ is an orientation preserving homeomorphism, by Lemma \ref{ch6,sec2,lm7}, there is an isotopy $H:\s^1 \times \I \to \s^1 $ such that $H(x,1) = hfh^{-1}(x)$ and $H(x,0) = x$.

Define $H^1:\s^1 \times \I \to S $ as $H^1(x,t) = h^{-1}(H(x,t))$.

Then by `$isotopy \ extension \ theorem$'  there is an isotopy
$$G: S \times \I \to S $$
such that $G(x,0)=x ,\  \forall x\in S$ and $H^1_t=G_tH^1_0$. Furthermore, the isotopy $G_t$ is identity map on the complement of a compact neighborhood of $C$.

Thus $G_1 : S \to S$ is the required homeomorphism, where $G_1=G|_{S \times \{1\}}$ .

If the map $f: C\to C$ is orientation reversing, we will take a copy $S'$ of S  and a homeomorphism
$$\Phi : S \to S'$$
such that, each boundary component of $S$ is preserved, and for $\Phi(C)=C'$ in $\partial S'$, the map $\Phi|_C$ is orientation reversing. Also consider a homeomorphism 
$$\Psi : S \to S'$$
such that,  each boundary component of $S$ is preserved as above, and for $\Psi(C)=C'$ in $\partial S'$, the map $\Psi|_C$ is orientation preserving.  Then the map
$$\Psi  f \Phi^{-1} : C' \to C'$$
is an orientation preserving homeomorphism and we proceed as above.
\end{proof}

\begin{lemma}\label{ch6,sec2,lm9}
Let S be a compact bordered surface of genus $g$ and let $\Gamma _1, \Gamma _2, \ldots ,\Gamma _p$ be a partition of  $\partial S$. Then for any $k < g$ there exist $p-1$ non intersecting simple closed curves $C_1,C_2, \ldots ,C_{p-1}$ in S which divide S into $p$ components $U_1,U_2, \ldots ,U_p$ so that $\Gamma _i \subset U_i$ for each $i$. Also $U_1$ is of genus $k$ and $U_i$ is of genus zero for $1<i<p$. If S is nonorientable and $k$ is even, then $U_1$ can be made either orientable of genus $k/2$.
\end{lemma}

\begin{proof}
Consider the representation of surfaces as discussed in Chapter \ref{ch5}.

\begin{center}
\includegraphics[width=1\columnwidth]{cpt17.png}
\end{center}
\begin{center}
Figure 6.3
\end{center}

We shall work out the proof for the nonorientable case, and the other case can be done on the same line. Let cardinality of  $\Gamma _i = l_i$. If $k=0$ then $C_1$ is an arc joining $l_1$ boundary curves of $\Gamma_1$ as shown in Figure 6.4 below and the complete closed curve $C_1$ is this arc with the corresponding arc on $D'$ (the disk removed from $S$, described in Chapter \ref{ch5}). Similarly all other closed curves will be completed like this and here we shall draw arcs only.

\begin{center}
\includegraphics[width=1\columnwidth]{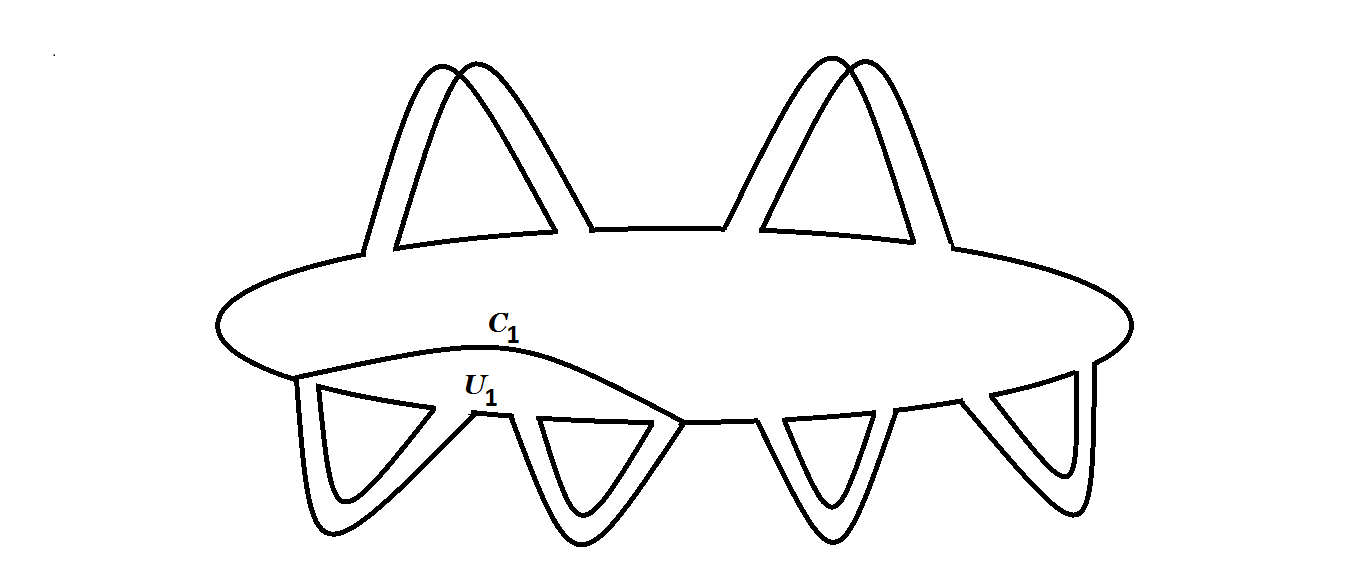}
\end{center}
\begin{center}
Figure 6.4
\end{center}

\begin{center}
\includegraphics[width=1\columnwidth]{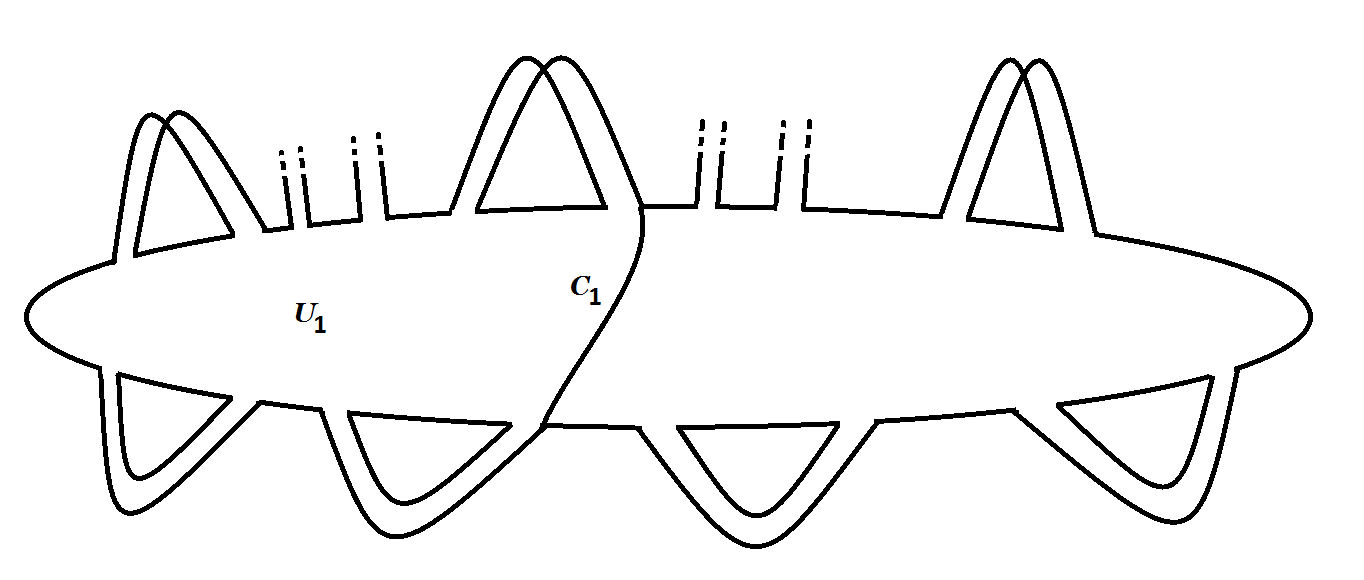}
\end{center}

\begin{center}
Figure 6.5
\end{center}

Now if $k>0$ then $C_1$ will join $k$ twisted strips with $l_1$ boundary curves of $\Gamma_1$ and corresponding component will be $U_1$

Also for each $1<i<p$, arc $C_i$ will join only $l_i$ boundary curves of $\Gamma_i$ and corresponding component will be called $U_i$.

And the set $$U_p \ = \ S \setminus\{ \bigcup_{i=1}^{p-1} U_i\} \ .$$ 

Now if $k$ is even and as $k<g$ by Theorem \ref{ch5,sec1,th2} we can simplify the representation of 

\begin{center}
\includegraphics[width=1\columnwidth]{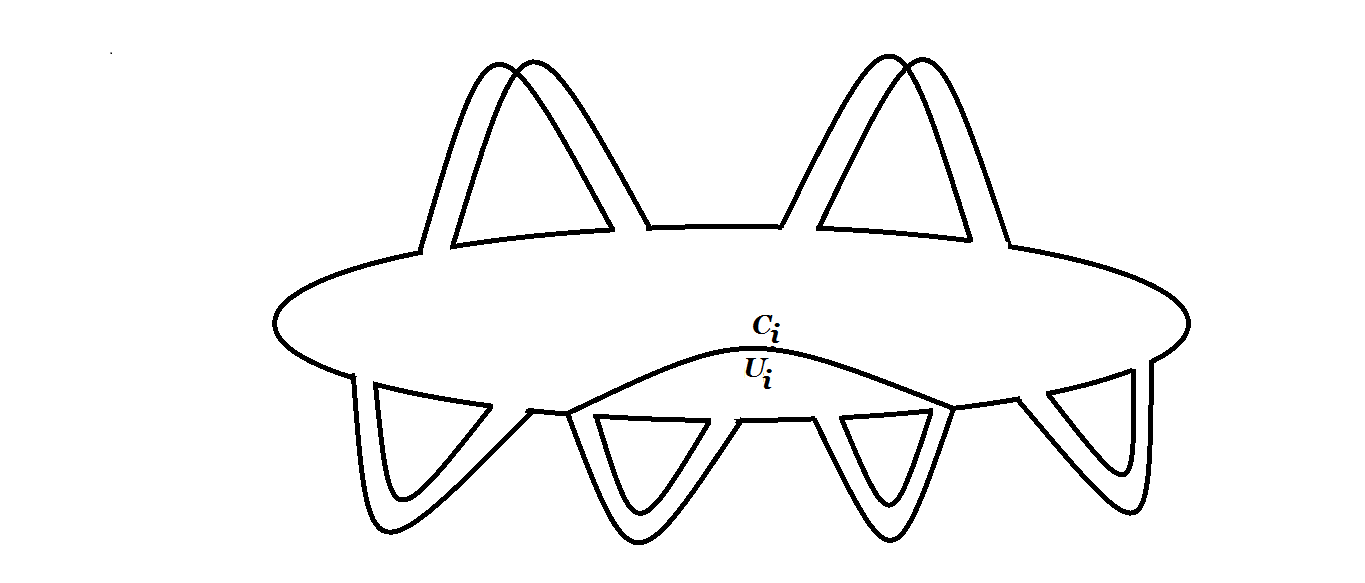}
\end{center}

\begin{center}
Figure 6.6
\end{center}

\begin{center}
\includegraphics[width=1\columnwidth]{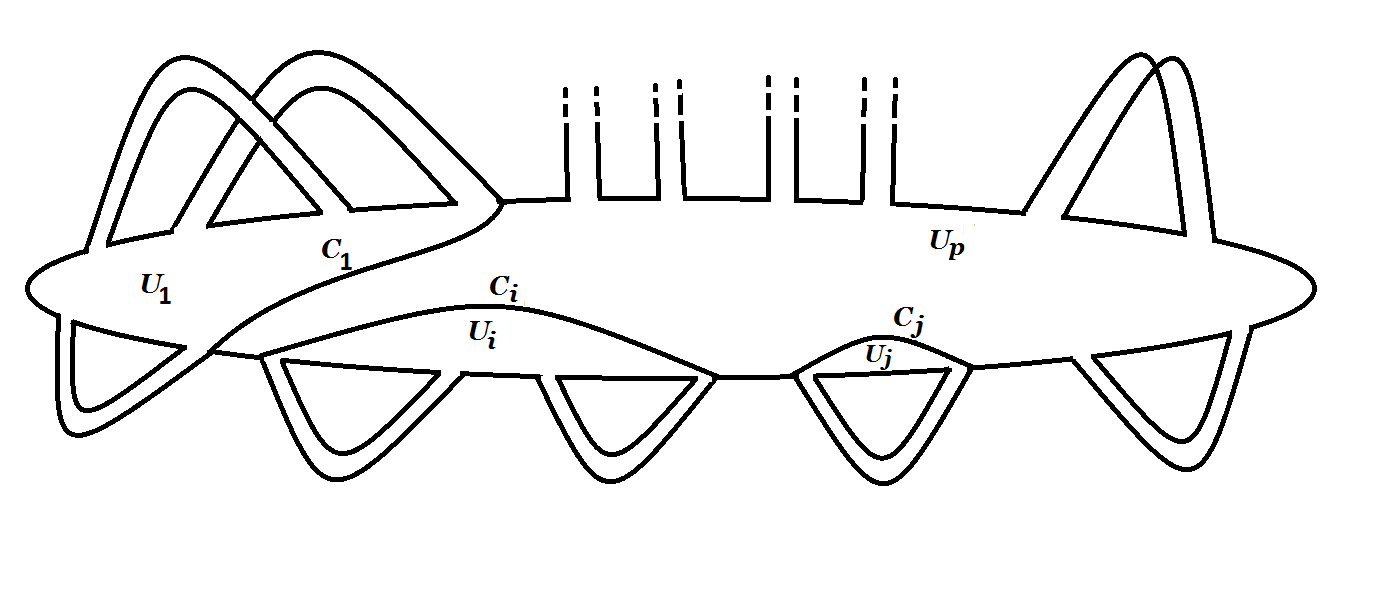}
\end{center}

\begin{center}
Figure 6.7
\end{center}
$U_1$ by turning each pair of twisted strips into a pair of intertwined annular strips making $U_1$ orientable and reducing its genus to $k/2$.
\end{proof}

\section{Ker\text{$\acute{\text{e}}$}kj\text{$\acute{\text{a}}$}rt\text{$\acute{\text{o}}$}'s  Theorem}

Now we shall discuss the classification theorem for noncompact surfaces given by B. V. Ker$\acute{e}$kj$\acute{a}$rt$\acute{o}$ (see ~\cite{kerekjarto}).

\begin{theorem} \label{ch6,sec3,th1}
Let S and $S'$ be two open surfaces of the same genus and orientability class. Then S and $S'$ are homeomorphic if and only if their ideal boundaries are homeomorphic, and the sets of planar and orientable ends are homeomorphic too.
\end{theorem}
\begin{proof}
Suppose that $S$ and $S'$ are of same genus and orientability class and let $h \colon B(S) \to B(S')$ be the homeomorphism between their ideal boundaries.

Let $\{ K_i \}$ and $\{ K_i' \}$ be canonical exhaustions of $S$ and $S'$. In particular take $K_0$ (and  $K_0'$)  having only one boundary component. If $S$ has infinite genus and $S$ is orientable or infinitely nonorientable take $K_0$ (and respectively $K_0'$) to be a disk. If $S$ is of odd or even nonorientability class, take $K_0$ (and respectively $K_0'$) to be a projective plane or a Klein bottle respectively with an open disk removed. Then by the classification theorem for compact bordered surface we have a homeomorphism $f_0 \colon K_0 \to K_0'$  and $f_0^* \colon \partial K_0 \to \partial K_0'$ where $f_0^* =f_0|_{\partial K_0}$.

Since $K_0$ and $K_0'$ are compact we have $B(S)= B(K_0^c)$ and $B(S')= B(K_0'^c)$, we have the homeomorphism $h \colon B(K_0^c) \to B(K_0'^c)$.

For the construction of the homeomorphism $F \colon S \to S'$ we replace the exhaustion $\{ K_i \}$ and $\{ K_i' \}$ by the canonical exhaustions $ \{ L_i \}$ and $\{ L_i' \}$ of $S$ and $S'$ respectively. Also for each $n \in \N$ construct a homeomorphism $g_n \colon L_n \to L_n'$ extending $g_{n-1}$, satisfying the condition that if $U$ and $U'$ are the corresponding components of $S \setminus L_n$ and $S' \setminus L_n'$ such that $g_n(\partial (U)) =\partial (U')$ then $B(U)$  will be homeomorphic to $B(U')$.

For the construction of desired exhaustion we proceed by induction. Let $ L_0=K_0$ and $L_0'=K_0'$ and $g_0=f_0$. Furthermore if $n$ is even, we make $L_n$ large enough so that $K_n \subset L_n$ and if $n$ is odd we alternate the process and make $K_n' \subset L_n'$.

Let $n$ is even and suppose that we have $g_n \colon L_n \to L_n'$ be a homeomorphism of canonical subsurfaces. For the construction of $L_{n+1}'$ , pick a $K_p'$ containing both $L_n'$ and $K_{n+1}'$ in its interior and let $L_{n+1}' = K_p'$.

$Claim \ 1$: For sufficiently large $M$, their is a $K_M$ containing $L_n$ in its interior such that if $U$ and $U'$ are corresponding components of  $S \setminus L_n$ and $S' \setminus L_n'$  then \\
(i) if $U' \cap (L_{n+1}' \setminus L_n')$ is nonorientable then $U \cap (K_M \setminus L_n)$  is nonorientable. \\
(ii) if genus of $U' \cap (L_{n+1}' \setminus L_n')$ is some $m \in \N$ then genus of $U \cap (K_M \setminus L_n)$ is more than 2$m$.

If $U' \cap (L_{n+1}' \setminus L_n')$ is nonorientable then $U'$ is infinitely nonorientable, then by Lemma \ref{ch6,sec2,lm6},  $B(U')$ contains nonorientable component. Since $B(U)$ is homeomorphic to $B(U')$, it  also contains nonorientable component and hence again by Lemma \ref{ch6,sec2,lm6}, $U$ is infinitely nonorientable. Thus there is a canonical subsurface $L$ containing $L_n$ in its interior such that $U \cap (L \setminus L_n)$ is nonorientable.

If $U' \cap (L_{n+1}' \setminus L_n')$ has positive genus then $U'$ has infinite genus and by Lemma \ref{ch6,sec2,lm6}, $B(U')$ contains non planar component and  so is $B(U)$. Again Lemma \ref{ch6,sec2,lm6} implies that $U$ has infinite genus. Thus there is a canonical subsurface $N$ containing $L_n$ in its interior such that genus of $U \cap (N \setminus L_n)$ is more than twice of the genus of $U' \cap (L_{n+1}' \setminus L_n')$.

Take $K_M$  to be large enough containing $L_n$ and all $L's$ and $N's$ (one for each component $U$) in its interior.

$Claim \ 2$: For $M$ sufficiently large, there is a partition $\mathcal{P}=\{ P\}$ of the set of components of $S\setminus K_M$ such that members of  $\mathcal{P}$ are in 1-1 correspondence with the components of $S' \setminus L_{n+1}'$ as follows: if $P \in \mathcal{P}$ corresponds to the component $U' \subset S' \setminus L_{n+1}'$ then
$$h^{-1}(B(U'))= \{ B(U)| \ U \in P\} .$$

Also if  $P$ corresponds to $U' \subset V' \subset S' \setminus L_n'$ and $V$  is a component of $S\setminus L_n$ corresponding to $V'$ ( that is, $g_n(\partial (V)) = \partial (V')$) then every $U \in P$ is contained in $V$.

For each $m\in \N $ let $\{ U_i^m \}$ be the components of $S \setminus K_m$. Then open sets $\{ B(U_i^m)\}$ form a basis for $B(S)$. Also let $\{U_j' \}$ be the components of $S'\setminus L_{n+1}'$. Since $L_{n+1}'$ is compact $B(L_{n+1}')=\phi$ and the collection $\{ B(U_j')\}$ covers $B(S')$. Thus $\{ h^{-1}(B(U_j'))\}$ covers $B(S)$ and each $h^{-1}(B(U_j'))$ is union of elements from the  collection $\{ B(U_i^m)\}$. Since $B(S)$  is compact a finite collection $Q$ of open sets from $\{B(U_i^m)\}$ which in turn lies in some $h^{-1}(B(U_j'))$ covers $B(S)$. Let 
$$M=max\{m|\  B(U_i^m)\in Q\} .$$

And corresponding to each component $U'$ of $S' \setminus L_{n+1}'$ take
$$P=\{U |\  U \ is \ a \ component \ of \ S \setminus K_M, B(U) \subset h^{-1}(B(U')) \} .$$

Also $U \in P$ implies $B(U) \subset h^{-1}(B(U')) \subset h^{-1}(B(V'))= B(V)$. Thus we have $U \subset V$.

Let $M$ to be large enough so that both of the above claim holds. Then we shall remove certain part of each component of $K_M \setminus L_n$ to get the desired set $L_{n+1}$.

By claim 1 each component of $K_M \setminus L_n$ is of sufficiently large genus and right orientability so that, using Lemma \ref{ch6,sec2,lm9}, we can diminish $K_M$ to a bordered subsurface $L_{n+1}$ containing $L_n$ so that each component of  $L_{n+1} \setminus L_n$ has the same genus as the corresponding component of $L_{n+1}' \setminus L_n'$  and is orientable if and only if the latter is. Also several components of $S \setminus K_M$ corresponds to one component of $S' \setminus L_{n+1}'$. Because of claim 2 we can assume, by diminishing$K_M \setminus L_n$, that each element of  $\mathcal{P}$ contains exactly one component of  $S \setminus L_{n+1}$, that is for every component $U'$ of $S' \setminus L_{n+1}'$, the corresponding  $P \in \mathcal{P}$ will consist of single component $U$ of $S \setminus L_{n+1}$ and $h^{-1}(B(U'))= B(U)$.

Now each component $U_i$ of $L_{n+1}\setminus L_n$ has the same genus and number of boundary components as the corresponding component of $L_{n+1}' \setminus L_n'$ and is orientable if and only if the latter is. Hence by the classification theorem of compact bordered surface there are  homeomorphisms $\theta _i \colon \overline{U}_i \to \overline{U}_i'$, for each component. We can take the homeomorphisms such that $\theta_i(\overline{U}_i\cap L_n)= \overline{U}_i'\cap L_n'$ and $\theta (\partial(U))=\partial(U')$ where $U \subset S\setminus L_{n+1}$ and $h(B(U))= B(U')$. Then $\theta \colon (\overline{L_{n+1}\setminus L_n}) \to (\overline{L_{n+1}'\setminus L_n'})$ is a homeomorphism where $\theta|_{\overline{U}_i}=\theta_i$. Also we may assume that each curve in the boundary $\partial(L_n)$ is invariant under the map $\theta^{-1}g_n \colon \partial(L_n) \to \partial(L_n)$. Let $\phi = \theta^{-1}g_n$ on $\partial(L_n)$ and extend $\phi$ to $(\overline{L_{n+1}\setminus L_n})$ by Lemma \ref{ch6,sec2,lm8}. Define $g_{n+1}: L_{n+1} \to L_{n+1}'$ as follows

\begin{equation}\notag
g_{n+1}(x)=
\begin{cases}

\theta \phi (x), & x\in L_{n+1}\setminus L_n\\
g_n (x), & x\in L_n .
\end{cases}
\end{equation}
\\

Thus we have exhaustions $\{ L_n \}$ and $\{ L_n'\}$ and homeomorphisms $g_{n} \colon L_n \to L_n'$ extending $g_{n-1}$. Hence the required homeomorphism $F \colon S \to S'$ is given by $F|_{L_n}=g_n$.
\end{proof}

\section{Noncompact Surfaces with Boundary}

The components of boundary of noncompact surfaces includes circles (compact) and open intervals (noncompact). In the case of circles, they can be contractible to a point or glued up by a disk. By gluing them by disks we get a noncompact surface, provided the boundary consists of circles only. Thus classification of noncompact surfaces with boundary reduces to a classification of noncompact surfaces without boundary, and two surfaces are homeomorphic if and only if the corresponding set of compact boundary components are homeomorphic too.

While dealing with open intervals in the boundary, we represent these components as  circles which are punctured at finite or infinite number of points. The ends corresponding to punctured points are called boundary ends or ends of boundary component. Two boundary ends are adjacent if there exists a boundary component for which they are the ends of it.

We can subdivide the noncompact boundary components and boundary ends into finite or infinite cyclic sequence in which any two consecutive element make a boundary component and its corresponding boundary end. Such a sequence is called a boundary cycle which is same as circle with punctured point(s).

Let $C(S)$ denotes the set of boundary ends of $S$. Define an equivalence relation `$\thicksim$' on the set $C(S)$  as follows: any two boundary ends are in same equivalence class if they are in same boundary cycle. Thus, $D(S) = C(S)/ \thicksim\ $ is a set of boundary cycles.

\begin{theorem}
Let $S$ and $S'$ be two noncompact surfaces with boundary, having same genus and orientability class. Then $S$ and $S'$ are homeomorphic if and only if their ideal boundaries, set of compact boundary components, C(S) with C($S'$) and D(S) with D($S'$) are homeomorphic and the set of planar and orientable ends and boundary ends are homeomorphic too.
\end{theorem}

\begin{proof}
Given $S$ and $S'$ are of same genus and orientability class, also the set of compact boundary components are homeomorphic. Then by gluing the boundary circles by disk, the noncompact surfaces $int(S)$ and $int(S')$ are homeomorphic. Consider the canonical exhaustions $\{L_n\}$ and $\{L_n'\}$ of $int(S)$ and $int(S')$ respectively, as constructed in Theorem \ref{ch6,sec3,th1} and homeomorphisms $g_n:L_n \to L_n'$ extending $g_{n-1}$ for each $n$.

We shall extend the exhaustion $\{L_n\}$  to $\{\widetilde{L_n}\}$ such that each $L_n$ is homeomorphic to $\widetilde{L_n}$ and $\{\widetilde{L_n}\}$ covers $S$. The same can be done for $S'$. Then for each $n$ there is a homeomorphism             $\widetilde{g_n}: \widetilde{L_n} \to \widetilde{L_n'}$ extending $\widetilde{g}_{n-1}$. This defines a homeomorphism $F:S \to S'$ such that $F|_{\widetilde{L_n}}=\widetilde{g_n}$.

Let $\{\al_k\}$ be the noncompact boundary components of $S$. Then we can arrange them in circles with  punctured points. Each $\al_k$ can be represented as an increasing sequence of closed segments $\{\beta^i_k\}_i$, each contained in the interior of the one following it. We continuously connect  $\beta^i_k$ with compact bordered surfaces $L_i$,  for each $i\ge k$, by a path $\gamma_k$ \\
\begin{center}
\includegraphics[width=1\columnwidth]{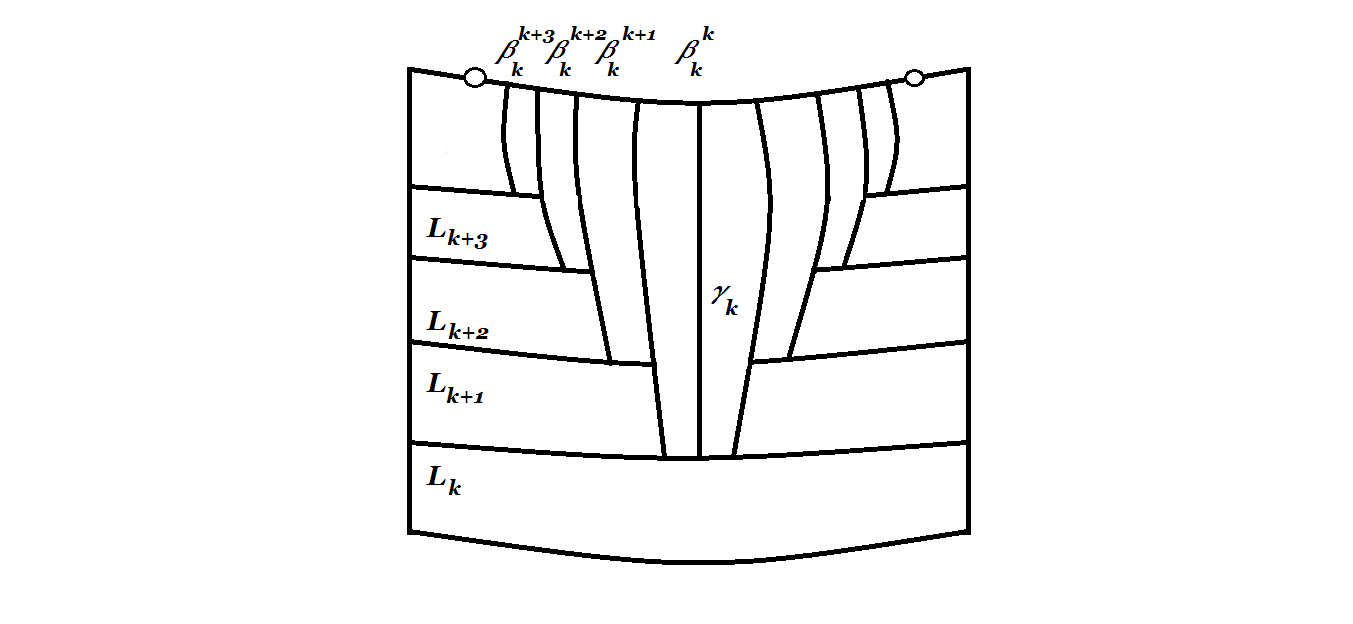}
\end{center}

\begin{center}
Figure 6.8
\end{center}
joining $\beta^k_k$ with $L_k$. We require that crossing at boundary of every $L_i$ be transversal in one point only. Also if a path goes out from the boundary of any $L_k$, then it crosses the boundary of $L_n$ for every $n>k$. To get a surface with boundary, we extend this path to a closed (planar) neighborhood containing $\beta^i_k$. This can be done in such a way that a neighborhood of $\gamma_k$ in each $L_i$ would contain the previous one, see Figure 6.8. At every step of gluing of the strips we have only finite number of segments. Let $\widetilde{L_i}$ be the compact bordered surface obtained by attaching the closed strips to $L_i$.

By a similar construction on $S'$ we can construct $\{\widetilde{L_i'}\}$. For this, we first rearrange the noncompact boundary components $\{ \al_k' \}$ of $S'$ such that the adjacent boundary ends of each $\al_k$ corresponds to adjacent boundary ends of $\al_k'$ and proceed as above.
\end{proof}


\backmatter

\end{document}